\documentclass{amsart}

\usepackage{amssymb}
\usepackage[all]{xy}
\usepackage{mathrsfs}
\usepackage{upgreek}
\usepackage{stmaryrd}
\usepackage{rotating}
\usepackage{tikz}
\usepackage[colorlinks=true]{hyperref}

\newcommand{\bk}{\Bbbk}
\newcommand{\Z}{\mathbb{Z}}
\newcommand{\C}{\mathbb{C}}
\newcommand{\Q}{\mathbb{Q}}
\newcommand{\F}{\mathbb{F}}
\newcommand{\bO}{\mathbb{O}}
\newcommand{\scO}{\mathscr{O}}
\newcommand{\Gm}{\mathbb{G}_{\mathrm{m}}}
\newcommand{\fR}{\mathfrak{R}}
\newcommand{\fS}{\mathfrak{S}}

\newcommand{\fg}{\mathfrak{g}}
\newcommand{\fb}{\mathfrak{b}}
\newcommand{\fn}{\mathfrak{n}}
\newcommand{\fp}{\mathfrak{p}}

\newcommand{\weyl}{\mathsf{M}}
\newcommand{\coweyl}{\mathsf{N}}
\newcommand{\tilt}{\mathsf{T}}
\newcommand{\irr}{\mathsf{L}}

\newcommand{\LQ}{\mathsf{L}_q}
\newcommand{\TQ}{\mathsf{T}_q}

\newcommand{\oVG}{\overline{V} \hspace{-2pt} {}_{G_1}}

\newcommand{\Gp}{\mathbf{G}}
\newcommand{\BGp}{\mathbf{B}}
\newcommand{\TGp}{\mathbf{T}}
\newcommand{\PGp}{\mathbf{P}}
\newcommand{\LGp}{\mathbf{L}}
\newcommand{\Fr}{\mathrm{Fr}}
\newcommand{\tcN}{{\widetilde{\mathcal{N}}}}
\newcommand{\cN}{\mathcal{N}}
\newcommand{\UQ}{\mathsf{U}_q}
\newcommand{\uQ}{\mathsf{u}_q}

\newcommand{\bX}{\mathbf{X}}
\newcommand{\bY}{\mathbf{Y}}
\newcommand{\Phis}{\Phi_{\mathrm{s}}}
\newcommand{\ext}{{\mathrm{ext}}}
\newcommand{\Wext}{W_{\mathrm{ext}}}
\newcommand{\uw}{\underline{w}}
\newcommand{\Wf}{W_{\mathrm{f}}}
\newcommand{\Sf}{S_{\mathrm{f}}}
\newcommand{\fW}{ {}^{\mathrm{f}} \hspace{-1pt} W}
\newcommand{\fWext}{{}^{\mathrm{f}} \hspace{-1pt} W_{\mathrm{ext}}}
\newcommand{\fWf}{ {}^{\mathrm{f}} \hspace{-1pt} W^{\mathrm{f}}}

\newcommand{\cE}{\mathcal{E}}
\newcommand{\cF}{\mathcal{F}}
\newcommand{\cG}{\mathcal{G}}
\newcommand{\cH}{\mathcal{H}}
\newcommand{\cJ}{\mathcal{J}}
\newcommand{\cK}{\mathcal{K}}
\newcommand{\cL}{\mathcal{L}}
\newcommand{\cO}{\mathcal{O}}
\newcommand{\cV}{\mathcal{V}}
\newcommand{\cT}{\mathcal{T}}

\newcommand{\Ex}{\mathscr{E}}
\newcommand{\PEx}{\mathfrak{E}}
\newcommand{\fL}{\mathfrak{L}}

\newcommand{\bc}{\mathbf{c}}
\newcommand{\Haff}{\mathcal{H}}
\newcommand{\Hf}{\mathcal{H}_{\mathrm{f}}}
\newcommand{\puH}{{}^p \hspace{-1pt} \underline{H}}
\newcommand{\uH}{\underline{H}}
\newcommand{\uN}{\underline{N}}
\newcommand{\puN}{{}^p \hspace{-1pt} \underline{N}}
\newcommand{\Masph}{\mathcal{M}_{\mathrm{asph}}}

\newcommand{\pRleq}{\leq^p_{\mathrm{R}}}
\newcommand{\Rleq}{\leq_{\mathrm{R}}}
\newcommand{\Rgeq}{\geq_{\mathrm{R}}}
\newcommand{\pRsim}{\sim^p_{\mathrm{R}}}
\newcommand{\Rsim}{\sim_{\mathrm{R}}}
\newcommand{\Tleq}{\leq_{\mathsf{T}}}
\newcommand{\Tsim}{\sim_{\mathsf{T}}}

\newcommand{\qTleq}{\leq_{\mathsf{T}}^q}
\newcommand{\qTsim}{\sim_{\mathsf{T}}^q}

\newcommand{\Gr}{\mathrm{Gr}}
\newcommand{\Fl}{\mathrm{Fl}}
\newcommand{\Iw}{\mathrm{Iw}}
\newcommand{\Iwa}{\mathbf{I}^\wedge}

\newcommand{\cC}{\mathcal{C}}
\newcommand{\cA}{\mathscr{A}}
\newcommand{\scS}{\mathscr{S}}
\newcommand{\scT}{\mathscr{T}}
\newcommand{\Irr}{\mathrm{Irr}}
\newcommand{\Lgr}{L^{\mathrm{gr}}}
\newcommand{\dgr}{\Delta^{\mathrm{gr}}}
\newcommand{\ngr}{\nabla^{\mathrm{gr}}}

\newcommand{\mix}{{\mathrm{mix}}}
\newcommand{\Perv}{\mathsf{Perv}}
\newcommand{\Par}{\mathsf{Parity}}
\newcommand{\Tilt}{\mathsf{Tilt}}
\newcommand{\Rep}{\mathsf{Rep}}
\newcommand{\Coh}{\mathsf{Coh}}
\newcommand{\Db}{D^{\mathrm{b}}}
\newcommand{\Dmix}{D^\mix}

\DeclareMathOperator{\Hom}{Hom}
\DeclareMathOperator{\Ext}{Ext}
\DeclareMathOperator{\End}{End}
\DeclareMathOperator{\Ind}{Ind}
\DeclareMathOperator{\pr}{pr}

\newcommand{\id}{\mathrm{id}}
\newcommand{\simto}{\xrightarrow{\sim}}
\newcommand{\la}{\langle}
\newcommand{\ra}{\rangle}
\newcommand{\supp}{\mathrm{Supp}}
\newcommand{\Spec}{\mathrm{Spec}}

\makeatletter
\def\lotimes{\@ifnextchar_{\@lotimessub}{\@lotimesnosub}}
\def\@lotimessub_#1{\mathchoice{\mathbin{\mathop{\otimes}^L}_{#1}}%
  {\otimes^L_{#1}}{\otimes^L_{#1}}{\otimes^L_{#1}}}
\def\@lotimesnosub{\mathbin{\mathop{\otimes}^L}}
\makeatother

\numberwithin{equation}{section}
\newtheorem{thm}{Theorem}[section]
\newtheorem{lem}[thm]{Lemma}
\newtheorem{prop}[thm]{Proposition}
\newtheorem{cor}[thm]{Corollary}
\newtheorem{conj}[thm]{Conjecture}
\theoremstyle{definition}
\newtheorem{defn}[thm]{Definition}

\theoremstyle{remark}
\newtheorem{rmk}[thm]{Remark}
\newtheorem{ex}[thm]{Example}

\title[On the Humphreys conjecture on support varieties]{On the Humphreys conjecture on support varieties of tilting modules}

 \author{Pramod N. Achar}
 \address{Department of Mathematics\\
   Louisiana State University\\
   Baton Rouge, LA 70803\\
   U.S.A.}
 \email{pramod@math.lsu.edu}
 
  \author{William Hardesty}
   \address{Department of Mathematics\\
   Louisiana State University\\
   Baton Rouge, LA 70803\\
   U.S.A.}
  \email{whardesty@lsu.edu}
 
 \author{Simon Riche}
 \address{Universit\'e Clermont Auvergne, CNRS, LMBP, F-63000 Clermont-Ferrand, France.
 }
 \email{simon.riche@uca.fr}
 
 \thanks{P.A. was supported by NSF Grant No.~DMS-1500890. S.R. was partially supported by ANR Grant No.~ANR-13-BS01-0001-01. This project has received funding from the European Research Council (ERC) under the European Union's Horizon 2020 research and innovation programme (grant agreement No. 677147).}

\begin{document}

\begin{abstract}
Let $G$ be a simply-connected semisimple algebraic group over an algebraically closed field of characteristic $p$, assumed to be larger than the Coxeter number.  The ``support variety'' of a $G$-module $M$ is a certain closed subvariety of the nilpotent cone of $G$, defined in terms of cohomology for the first Frobenius kernel $G_1$. In the 1990s, Humphreys proposed a conjectural description of the support varieties of tilting modules; this conjecture has been proved for $G = \mathrm{SL}_n$ in earlier work of the second author. 

In this paper, we show that for any $G$, the support variety of a tilting module always contains the variety predicted by Humphreys, and that they coincide (i.e., the Humphreys conjecture is true) when $p$ is sufficiently large.  We also prove variants of these statements involving ``relative support varieties.''
\end{abstract}

\maketitle

\section{Introduction}

\subsection{Support varieties}

Let $\bk$ be an algebraically closed field of characteristic $p>0$, and let $G$ be a simply-connected semisimple algebraic group over $\bk$. We assume that $p>h$, where $h$ is the Coxeter number of $G$.

If $\dot G$ is the Frobenius twist of $G$ and $G_1 \subset G$ is the kernel of the Frobenius morphism $\Fr : G \to \dot{G}$, then the algebra $\Ext^\bullet_{G_1}(\bk,\bk)$ admits a natural action of $G$, which factors through an action of $\dot{G}$, and it is well known that there exists a $\dot G$-equivariant isomorphisms of algebras
\[
\Ext^\bullet_{G_1}(\bk,\bk) \cong \cO(\cN),
\]
where $\cN$ is the nilpotent cone of $\dot{G}$. If $M$ is a $G$-module, then the $\bk$-vector space $\Ext^\bullet_{G_1}(M,M)$ admits a natural action of $G$ which factors through $\dot{G}$, and a natural (compatible) structure of module over $\Ext^\bullet_{G_1}(\bk,\bk)$. In this way this space defines a $\dot{G}$-equivariant quasi-coherent sheaf on $\cN$, and the \emph{support variety} $V_{G_1}(M)$ is defined as the support of this sheaf (a closed $\dot{G}$-stable subvariety of $\cN$). In a similar way, one can define the \emph{relative support variety} $\oVG(M)$ of $M$ as the support of $\Ext^\bullet_{G_1}(\bk, M)$. (More generally, one can consider relative support varieties for any pair of $G_1$-modules, see~\cite[\S 2.2]{npv}. But in this paper we will only consider the case of pairs of the form $(M,M)$ or $(\bk,M)$.)

\subsection{The Humphreys conjecture}

We fix a Borel subgroup $B \subset G$ and a maximal torus $T \subset G$, denote by $\bX$ the character lattice of $T$, and let $\bX^+ \subset \bX$ be the subset of dominant weights (for the choice of positive roots such that $B$ is the \emph{negative} Borel subgroup).  Let $\Wf := N_G(T)/T$ be the Weyl group of $G$, and let $W$ be the affine Weyl group (the semidirect product of $\Wf$ with the root lattice).  For $\lambda \in \bX^+$, let $\tilt(\lambda)$ denote the indecomposable tilting $G$-module of highest weight $\lambda$.

The \emph{Humphreys conjecture}~\cite{humphreys} gives a conjectural description of the support varieties $V_{G_1}(\tilt(\lambda))$. To describe this answer, consider the set $\fW \subset W$ of elements $w$ that are of minimal length in the right coset $\Wf w$.
Then the alcoves in $\mathbb{R} \otimes_\Z \bX$ which meet $\bX^+$ are in a natural bijection with $\fW$. (In the definition of this bijection we use the ``$p$-dilated dot action'' $\cdot_p$ of $W$ on $\mathbb{R} \otimes_\Z \bX$.) On the other hand, results of Lusztig~\cite{lusztig} provide a bijection between the set of $\dot{G}$-orbits in $\cN$ and the set of \emph{two-sided cells} in $W$, and results of Lusztig--Xi~\cite{lx} show that intersecting with $\fW$ provides a bijection between the set of \emph{two-sided cells} in $W$ and the set of \emph{right cells} in $W$ which are contained in (or equivalently which meet) $\fW$. Combining these results we obtain a bijection between right cells intersecting $\fW$ and $\dot{G}$-orbits in $\cN$; for $w \in \fW$ we denote by $\scO_w$ the orbit corresponding to the cell containing $w$. With this at hand, the Humphreys conjecture states that for all $\lambda \in \bX^+$ we have
\[
V_{G_1}(\tilt(\lambda)) = \overline{\scO_w}
\]
if $w \in \fW$ is the element such that $\lambda$ belongs to the lower closure of the alcove corresponding to $w$. As of this date, this conjecture is open except in the case $G=\mathrm{SL}_n(\bk)$, where it was proved by the second author~\cite{hardesty} under the assumption that $p>n+1$.

This conjecture has a natural ``relative'' version, which can be stated as follows. We denote by $\fWf \subset W$ the subset consisting of elements $w$ which are minimal in $\Wf w \Wf$. Then it seems natural to conjecture that for any $w \in \fWf$ we have
\[
\oVG(\tilt(w \cdot_p 0)) = \overline{\scO_w}.
\]
(Note that $\oVG(\tilt(\lambda))=\varnothing$ if $\lambda \in \bX^+ \smallsetminus (\fWf \cdot_p 0)$, so that only the dominant weights in $\fWf \cdot_p 0$ need to be considered for this question.) To the best of our knowledge, this ``relative'' version is open in all cases.

\subsection{Main results}

The main results of this paper are summarized in the following statement.

\begin{thm}
\label{thm:main-intro}
\leavevmode
\begin{enumerate}
\item
\label{it:thm-intro-1}
For any $\lambda \in \bX^+$ and $w \in \fW$ such that $\lambda$ belongs to the lower closure of the alcove corresponding to $w$, we have
\[
V_{G_1}(\tilt(\lambda)) \supset \overline{\scO_w}.
\]
Moreover, if $w \in \fWf$ we have
\[
\oVG(\tilt(w \cdot_p 0)) \supset \overline{\scO_w}.
\]
\item
\label{it:thm-intro-2}
There is an integer $N > 0$ (depending only on the root system of $G$) such that if $p>N$, then the inclusions in~\eqref{it:thm-intro-1} are equalities for all $\lambda$ and $w$.
\end{enumerate}
\end{thm}

In other words, part~\eqref{it:thm-intro-2} says that the Humphreys conjecture and its relative version are true when $p > N$.  At the moment, we do not know how to determine $N$ explicitly (except for $\mathrm{SL}_n$).  In view of~\cite{hardesty}, this result has the following consequence.

\begin{cor}
For any given $G$ and $\bk$, if the Humphreys conjecture holds then its relative version also holds. In particular, the relative version holds if $G=\mathrm{SL}_n(\bk)$ and $p>n+1$.
\end{cor}

We also use Theorem~\ref{thm:main-intro} to check the Humphreys conjecture under some mild (and explicit) assumptions on $p$ in some cases when $G$ has rank~$2$.

\subsection{The quantum case}

Our proof of Theorem~\ref{thm:main-intro} builds on some results of Ostrik and Bezrukavnikov obtained in the course of the proof of the quantum analogue of the Humphreys conjecture. Namely, consider Lusztig's quantum group $\UQ$ associated with the root system of $G$, specialized at a complex primitive $\ell$-th root of unity. (Here $\ell>h$ is odd, and prime to $3$ if $G$ has a component of type $\mathbf{G}_2$.) 
This algebra has a subalgebra $\uQ \subset \UQ$, called the ``small quantum group,'' which plays a role analogous to $G_1$. Moreover, the algebra $\Ext^\bullet_{\uQ}(\C,\C)$ (where $\C$ is the trivial module) has a natural action of the complex group $G_\C$ with the same root system as $G$, and is isomorphic as a $G_\C$-equivariant algebra to functions on the nilpotent cone $\cN_\C$ of $G_\C$.

For each $\lambda \in \bX^+$, let $\TQ(\lambda)$ be the corresponding indecomposable tilting $\UQ$-module. It makes sense to consider the support varieties $V_{\uQ}(\TQ(\lambda))$ and $\overline{V} \hspace{-2pt}_{\uQ}(\TQ(\lambda))$.  In this setting, there are obvious analogues of the Humphreys conjecture and its relative version, stated in terms of the $\ell$-dilated dot action $\cdot_\ell$ of $W$ on $\mathbb{R} \otimes_\Z \bX$.  We call these statements the \emph{quantum Humphreys conjecture} and the \emph{relative quantum Humphreys conjecture}, respectively.

The quantum Humphreys conjecture was proved by Ostrik~\cite{ostrik-supp} in type $\mathbf{A}$, and the relative version was proved by Bezrukavnikov~\cite{bezru} in general. As we explain in~\S\ref{ss:quantum-Humphreys} (and as was probably known to experts), one can in fact deduce the quantum Humphreys conjecture from the relative version in general.

\subsection{Outline of the proof}

Bezrukavnikov's proof of the relative quantum Hum\-phreys conjecture relies on results of~\cite{abg} to reduce the computation of $\overline{V}\hspace{-2pt}_{\uQ}(\TQ(\lambda))$
to the computation of the support of certain complexes of the form $\pi_* \mathcal{L}$, where $\pi : \tcN_\C \to \cN_\C$ is the Springer resolution, and $\mathcal{L}$ is a simple object in the heart of the \emph{exotic t-structure} on $\Db \Coh^{G_\C}(\tcN_\C)$. Then these objects are identified with the simple \emph{perverse coherent} sheaves on $\cN_\C$, whose support is known.

A ``modular analogue'' of the results of~\cite{abg} have been obtained by the first and third authors in~\cite{prinblock}. As in the quantum case, this reduces the problem to computing the supports of certain objects of the form $\pi_*\cL$, where $\pi: \tcN \to \cN$ is the Springer resolution over $\bk$.  However, $\cL$ is now \emph{not} in general a simple exotic sheaf.  Instead, it is an indecomposable object of $\Db \Coh^{G \times \Gm}(\tcN)$ that obeys certain \emph{parity-vanishing conditions} with respect to the exotic t-structure.  (The principle of replacing simple objects by ``parity objects,'' which has been advocated in the work of Soergel and Juteau--Mautner--Williamson, has already proved to be very fruitful.)

These supports are difficult to compute, but we are able to bound them below by a ``change-of-scalars'' argument.
In large characteristic, they can be deduced from their characteristic-$0$ counterparts. This implies Theorem~\ref{thm:main-intro}.

\begin{rmk}
As suggested above, our proof of Theorem~\ref{thm:main-intro} relies on Bezrukavnikov's results in the quantum setting. But in fact our constructions yield a significant simplification of his proof: namely, one can replace the ``Positivity Lemma'' of~\cite{bezru} by parity considerations, making the proof independent of the results of~\cite{ab}; see Remark~\ref{rmk:proof-Humphreys-quantum} for more details.
\end{rmk}

\subsection{Weight cells}
\label{ss:intro-weight-cells}

In the course of the proof of Theorem~\ref{thm:main-intro} we obtain a ``modular version'' of another result of Ostrik on quantum groups, which might be of independent interest.

Namely, define the ``quantum weight preorder'' on $\bX^+$ by setting $\lambda \Tleq^q \mu$ if $\TQ(\lambda)$ is a direct summand of $\TQ(\mu) \otimes M$ for some tilting $\UQ$-module $M$.  The ``modular weight preorder'' $\Tleq$ is defined in the same way, but using tilting $G$-modules instead of tilting $\UQ$-modules.
(Of course, these preorders depend on $\ell$ or on $p$. In the quantum case this preorder was introduced in~\cite{ostrik}; see also~\cite{andersen} for a thorough study of these orders in both settings.) We denote by $\Tsim^q$ and $\Tsim$ the associated equivalence relations. 

It is easily checked that two weights in the lower closure of the some alcove are always equivalent, so that these preorders and equivalence relations descend to alcoves meeting $\bX^+$, i.e.~to $\fW$. The main result of~\cite{ostrik} states that the order $\Tleq^q$ on $\fW$ coincides with the restriction of the right Kazhdan--Lusztig order on $W$; in particular, the corresponding equivalence classes are the Kazhdan--Lusztig right cells contained in $\fW$. Our ``modular counterpart'' states that the order $\Tleq$ on $\fW$ coincides with the order defined by the same rule, replacing the Kazhdan--Lusztig basis by the $p$-Kazhdan--Lusztig basis of~\cite{jw}. The equivalence classes for this preorder are called \emph{right $p$-cells}. (See~\cite{jensen} for a general study of such cells.)

Our proof of this claim is essentially identical to Ostrik's proof, replacing Soergel's description of characters of quantum tilting modules by its modular counterpart conjectured in~\cite{rw} and proved in~\cite{mkdkm}. (This proof is independent of the rest of the paper, except for the definitions given in Section~\ref{sec:antispherical-cells}.)

\subsection{Contents}

The paper is organized as follows. 

Sections~\ref{sec:parity}--\ref{sec:support} form the ``geometric'' part of the paper.
In Section~\ref{sec:parity} we develop an analogue for derived categories of graded highest weight categories of the theory of ``parity complexes'' from~\cite{jmw}. Then in Section~\ref{sec:exotic-parity} we study these objects further in the case of the heart of Bezruakvnikov's \emph{exotic t-structure} on $\Db \Coh^{G \times \Gm}(\tcN)$. In particular we provide a ``Bott--Samelson type'' construction of these objects, which lets us study them via modular reduction arguments. In Section~\ref{sec:support} we prove some results on the support of certain objects constructed out of these parity exotic sheaves.

Sections~\ref{sec:antispherical-cells} and~\ref{sec:finite-generation} form the ``combinatorial'' part of the paper. In Section~\ref{sec:antispherical-cells} we recall the theory of $p$-cells in the affine Hecke algebra and its antispherical module, and in Section~\ref{sec:finite-generation} we prove that the (usual) right cells in the antispherical module are ``finitely generated'' in an appropriate sense. (A variant of this result is stated without proof in~\cite{andersen}; but as indicated in~\cite{andersen} it can be deduced from results of Xi~\cite{xi}.)

Finally, Sections~\ref{sec:cells}--\ref{sec:examples} form the ``representation-theoretic'' part of the paper. In Section~\ref{sec:cells} we prove the analogue of Ostrik's theorem described in~\S\ref{ss:intro-weight-cells}. Section~\ref{sec:humphreys-conj} provides a reminder on the Humphreys conjecture, and the relation with its relative version. In Section~\ref{sec:exotic-tilting} we prove Theorem~\ref{thm:main-intro}. We conclude the paper with some examples in Section~\ref{sec:examples}, including many cases in type $\mathbf{G}_2$ (if $p> 7$), as well as a complete proof of the Humphreys conjecture for groups of type $\mathbf{C}_2$ (if $p>5$).

\subsection{Acknowledgements}

We thank Lars Thorge Jensen and Geordie Williamson for helpful conversations on cells and $p$-cells, and we thank the anonymous referees for their comments and suggestions.

\section{Parity objects for graded highest weight categories}
\label{sec:parity}

In this section, we develop a theory of ``parity objects'' in the derived category of a graded highest weight abelian category.  This notion is an algebraic or categorical analogue of the theory of parity sheaves from~\cite{jmw}, in much the same way that highest weight categories themselves are an algebraic or categorical counterpart to perverse sheaves.

\subsection{Graded highest weight categories}
\label{ss:hwcat}

Let $\bk$ be a field.
Let $\cA$ be a finite-length $\bk$-linear abelian category, and  
assume that $\cA$ is equipped with an automorphism $\la 1\ra: \cA \to \cA$.  Let $\Irr(\cA)$ be the set of isomorphism classes of irreducible objects of $\cA$, and let $\scS = \Irr(\cA)/\Z$, where $n \in \Z$ acts on $\Irr(\cA)$ by $\la n\ra$.  Assume that $\scS$ is equipped with a partial order $\le$, and that for each $s \in \scS$, we have a fixed representative simple object $\Lgr_s$. Assume also that we are given, for any $s \in \scS$, objects $\dgr_s$ and $\ngr_s$, and morphisms $\dgr_s \to \Lgr_s$ and $\Lgr_s \to \ngr_s$. For $\scT \subset \scS$, we denote by $\cA_{\scT}$ the Serre subcategory of $\cA$ generated by the objects $\Lgr_t \la n \ra$ for $t \in \scT$ and $n \in \Z$. We write $\cA_{\leq s}$ for $\cA_{\{t \in \scS \mid t \leq s\}}$, and similarly for $\cA_{<s}$.

\begin{defn}
\label{defn:qhered}
The category $\cA$ is said to be \emph{graded highest weight} if the following conditions hold:
\begin{enumerate}
\item For any $s \in \scS$, the set $\{t \in \scS \mid t \leq s\}$ is finite.\label{it:qh-def-fin}
\item For each $s \in \scS$, we have 
\[
\Hom(\Lgr_s,\Lgr_s \la n \ra) = \begin{cases}
\bk & \text{if $n=0$;} \\
0 & \text{otherwise.}
\end{cases}
\]
 \label{it:qh-def-split}
\item For any $\scT \subset \scS$ closed (for the order topology) and such that $s \in \scT$ is maximal, $\dgr_s \to \Lgr_s$ is a projective cover in $\cA_{\scT}$ and $\Lgr_s \to \ngr_s$ is an injective envelope in $\cA_{\scT}$.
\item The kernel of $\dgr_s \to \Lgr_s$ and the cokernel of $\Lgr_s \to \ngr_s$ belong to $\cA_{<s}$.
\label{it:qh-def-ker}
\item We have $\Ext^2(\dgr_s, \ngr_t\la n\ra) = 0$ for all $s, t \in \scS$ and $n \in \Z$.\label{it:qh-def-ext2}
\end{enumerate}
\end{defn}

If $\cA$ satisfies Definition~\ref{defn:qhered}, then $\Db(\cA)$ is a Krull--Schmidt category by~\cite[Corollary~B]{le-chen}.
The objects $\dgr_s\la n\ra$ are called \emph{standard objects}, and the objects $\ngr_s\la n\ra$ are called \emph{costandard objects}.

For the following lemma, see~\cite[Lemma~2.2]{modrap3}. (See also~\cite[\S 7.3]{riche-hab} for more details.)

\begin{lem}
\label{lem:qher-closed-subset}
Let $\scT \subset \scS$ be a closed subset in the order topology. 
\begin{enumerate}
\item
The subcategory $\cA_{\scT} \subset \cA$ is a graded highest weight category, with standard (resp.~costandard) objects $\dgr_t$ (resp.~$\ngr_t$) for $t \in \scT$. Moreover, the functor $\iota_{\scT}: \Db (\cA_{\scT})\to \Db (\cA)$ induced by the inclusion $\cA_{\scT} \subset \cA$ is fully faithful.
\item
The Serre quotient $\cA / \cA_{\scT}$ is a graded highest weight category for the order on $\scS \smallsetminus \scT$ obtained by restriction from the order on $\scS$. The standard (resp.~costandard) objects are the images in the quotient of the objects $\dgr_s$ (resp.~$\ngr_s$) for $s \in \scS \smallsetminus \scT$.
\item
The natural functor $\Db(\cA) / \Db(\cA_{\scT}) \to \Db(\cA/\cA_{\scT})$ (where the left-hand side is the Verdier quotient) is an equivalence. Moreover, the functors $\Pi_{\scT} : \Db(\cA) \to \Db(\cA/\cA_{\scT})$ and $\iota_{\scT}$ admit left and right adjoints, denoted $\Pi_{\scT}^R$, $\Pi_{\scT}^L$, $\iota_{\scT}^R$, $\iota_{\scT}^L$, which satisfy
\begin{equation}
\label{eqn:adjoints-pi-d-n}
\Pi_\scT^L \circ \Pi_{\scT}(\dgr_s) \cong \dgr_s, \qquad \Pi_\scT^R \circ \Pi_{\scT}(\ngr_s) \cong \ngr_s
\end{equation}
for $s \in \scS \smallsetminus \scT$
and such that, for any $M$ in $\Db(\cA)$, the adjunction morphisms induce functorial distinguished triangles
\[
\iota_{\scT} \iota_{\scT}^R M \to M \to \Pi_\scT^R \Pi_\scT M \xrightarrow{[1]}, \qquad \Pi_\scT^L \Pi_\scT M \to M \to \iota_{\scT} \iota_{\scT}^L M \xrightarrow{[1]}.
\]
\end{enumerate}
\end{lem}

When $\scT = \{t \in \scS \mid t \leq s\}$, we will write $\iota_{\le s}$, $\Pi_{\le s}$, etc., for the functors introduced in this lemma, and likewise when $\scT = \{t \in \scS \mid t < s \}$.

\subsection{Parity objects}
\label{ss:parity-objects}

We fix a graded $\bk$-linear highest weight category $\cA$ with weight poset $(\scS,\leq)$.  Assume we are given a function $\dag : \scS \to \{0,1\}$. Following~\cite{jmw} we consider the following definition.

\begin{defn}
An object $M$ in $\Db \cA$ will be called $(!,\dag)$-\emph{even} if it satisfies
\[
\Hom_{\Db \cA}(\dgr_s, M \langle n \rangle [m]) \neq 0 \quad \Rightarrow \quad m=n \text{ and } n \equiv \dag(s) \pmod 2. \]
It will be called $(*,\dag)$-\emph{even} if it satisfies
\[
\Hom_{\Db \cA}(M \langle n \rangle [m], \ngr_s) \neq 0 \quad \Rightarrow \quad m=n \text{ and } n \equiv \dag(s) \pmod 2.
\]
It will be called $\dag$-\emph{even} if it is both $(!,\dag)$-\emph{even} and $(*,\dag)$-\emph{even}. Finally, $M$
will be called \emph{parity} if $M \cong M_0 \oplus M_1$ where $M_0$ and $M_1 \langle 1 \rangle[1]$ are even.
\end{defn}

\begin{lem}
\label{lem:parity-vanishing}
Let $M$ be $(*,\dag)$-even and $N$ be $(!,\dag)$-even. Then
\[
\Hom_{\Db \cA}(M, N \langle n \rangle[m])=0
\]
unless $n=m$ and $n$ is even.
\end{lem}

\begin{proof}
We prove the claim by induction on the cardinality of the smallest finite closed subset $\scT \subset \scS$ such that $M$ belongs to the essential image of $\Db \cA_{\scT}$. If $\scT=\{s\}$, then $M$ is a direct sum of copies of $\Lgr_s\langle k\rangle[k] \cong \dgr_s \langle k \rangle[k]$ with $k \equiv \dag(s) \pmod 2$. Then the desired vanishing follows from the definitions.

Now, assume that $\#\scT \geq 2$, and let $s \in \scT$ be maximal. Let $\scT':=\scT \smallsetminus \{s\}$. We can consider the functors $\iota_{\scT}$ and $\iota_{\scT'}$, and their right adjoints as in Lemma~\ref{lem:qher-closed-subset}. It follows directly from the definition that
\begin{equation}
\label{eqn:claim-parity-1}
\text{$\iota_{\scT'}^R(N)$ is $(!,\dag)$-even.}
\end{equation}
(Here and below, in a minor abuse of notation, we write $\dag$ for the restriction of this function to $\scT$ or $\scT'$.)

Let $M' \in \Db \cA_{\scT}$ be such that $M=\iota_{\scT}(M')$. It is easy to see that $M'$ is $(*,\dag)$-even. We consider the functors
\[
\tilde{\iota}_{\scT'} : \Db \cA_{\scT'} \to \Db \cA_{\scT}, \quad \tilde{\Pi}_{\scT'} : \Db \cA_{\scT} \to \Db (\cA_{\scT}/\cA_{\scT'})
\]
and their adjoints as in Lemma~\ref{lem:qher-closed-subset}, and
the associated distinguished triangle
\begin{equation}
\label{eqn:triangle-parity-vanishing}
\tilde{\Pi}_{\scT'}^L \tilde{\Pi}_{\scT'} M' \to M' \to \tilde{\iota}_{\scT'} \tilde{\iota}_{\scT'}^L (M') \xrightarrow{[1]}.
\end{equation}
We claim that
\begin{gather}
\label{eqn:claim-parity-2}
\text{$\tilde{\iota}_{\scT'}^L(M')$ is $(*,\dag)$-even;}\\
\label{eqn:claim-parity-3}
\text{$\tilde{\Pi}_{\scT'}^L \tilde{\Pi}_{\scT'} M'$ is a direct sum of objects $\dgr_s \langle n \rangle [n]$ with $n \equiv \dag(s) \pmod 2$.}
\end{gather}
In fact,~\eqref{eqn:claim-parity-2} follows directly from the definitions (as for~\eqref{eqn:claim-parity-1} above). For~\eqref{eqn:claim-parity-3}, we observe that the category $\cA_{\scT} / \cA_{\scT'}$ is graded highest weight with a weight poset consisting of one element; hence it is semisimple, with simple objects of the form $\tilde{\Pi}_{\scT'}(\Lgr_s) \langle n \rangle$. Therefore, $\tilde{\Pi}_{\scT'} M'$ is a direct sum of objects of the form 
\[
\tilde{\Pi}_{\scT'}(\dgr_s) \langle n \rangle [m] = \tilde{\Pi}_{\scT'}(\ngr_s) \langle n \rangle [m] = \tilde{\Pi}_{\scT'}(\Lgr_s) \langle n \rangle [m].
\]
Since $M'$ is $(*,\dag)$-even, using adjunction and the second isomorphism in~\eqref{eqn:adjoints-pi-d-n}, we see that in fact $\tilde{\Pi}_{\scT'} M'$ is a direct sum of objects as above with $n=m$ and $n \equiv \dag(s) \pmod 2$. Using the first isomorphism in~\eqref{eqn:adjoints-pi-d-n}, we deduce~\eqref{eqn:claim-parity-3}.

Now we consider the image
\[
\iota_{\scT} \tilde{\Pi}_{\scT'}^L \tilde{\Pi}_{\scT'} M' \to M \to \iota_{\scT'} \tilde{\iota}_{\scT'}^L (M') \xrightarrow{[1]}
\]
of the distinguished triangle~\eqref{eqn:triangle-parity-vanishing} under $\iota_\scT$. Applying $\Hom(-,N\langle n \rangle [m])$ we deduce an exact sequence
\begin{multline*}
\Hom(\iota_{\scT'} \tilde{\iota}_{\scT'}^L (M'), N\langle n \rangle [m]) \to \Hom(M,N\langle n \rangle [m]) \\
\to \Hom(\iota_{\scT} \tilde{\Pi}_{\scT'}^L \tilde{\Pi}_{\scT'} M',N\langle n \rangle [m]).
\end{multline*}
By induction,~\eqref{eqn:claim-parity-1} and~\eqref{eqn:claim-parity-2}, the first term vanishes unless $n=m$ and $n$ is even. By~\eqref{eqn:claim-parity-3}, the last term also vanishes unless $n=m$ and $n$ is even. We deduce the same property for the middle term.
\end{proof}

\begin{cor}
\label{cor:parity-surjectivity}
Let $\scT \subset \scS$ be a closed subset in the order topology.
If $M$ and $N$ are $\dag$-even, then the morphism
\[
\Hom_{\Db \cA}(M,N) \to \Hom_{\Db(\cA/\cA_{\scT})}(\Pi_{\scT}(M), \Pi_{\scT}(N))
\]
induced by $\Pi_{\scT}$ is surjective.
\end{cor}

\begin{proof}
By adjunction, the morphism under consider identifies with the morphism
\[
\Hom_{\Db \cA}(M,N) \to \Hom_{\Db \cA}(M,\Pi_{\scT}^R \Pi_{\scT}(N))
\]
induced by the adjunction morphism $N \to \Pi_{\scT}^R \Pi_{\scT}(N)$. Consider the distinguished triangle
\[
\iota_{\scT} \iota_{\scT}^R (N) \to N \to \Pi_\scT^R \Pi_\scT (N) \xrightarrow{[1]}
\]
from Lemma~\ref{lem:qher-closed-subset}. Applying $\Hom(M,-)$ we obtain an exact sequence
\[
\Hom(M,N) \to \Hom(M,\Pi_{\scT}^R \Pi_{\scT}(N)) \to \Hom(M,\iota_{\scT} \iota_{\scT}^R N[1]).
\]
Now one can easily check that $\iota_{\scT} \iota_{\scT}^R N$ is $(!,\dag)$-even. Hence by Lemma~\ref{lem:parity-vanishing} we have $\Hom(M,\iota_{\scT} \iota_{\scT}^R N[1])=0$, and the desired vanishing follows.
\end{proof}

\subsection{Classification of parity objects}
\label{ss:unicity-parity}

As in~\S\ref{ss:parity-objects}, we fix a graded highest weight category $\cA$ with weight poset $(\scS,\leq)$.  

\begin{prop}
\label{prop:classification-parity}
For any $s \in \scS$, there exists at most one indecomposable parity object $E_s$ in $\Db \cA_{\leq s}$ such that $\Pi_{<s}(E_s) \cong \Pi_{<s}(\dgr_s)$, where $\Pi_{<s} : \Db \cA_{\leq s} \to \Db (\cA_{\leq s}/\cA_{<s})$ is the quotient functor. Moreover, the objects
\[
\iota_{\leq s}(E_s) \langle n \rangle [n] \quad \text{for $s \in \scS$ such that $E_s$ exists and $n \in \Z$}
\]
form a set of representatives of the isomorphism classes of parity objets in $\Db \cA$.
\end{prop}

\begin{proof}
We start with the uniqueness of $E_s$. Suppose that $E_s$ and $E_s'$ are two indecomposable parity objects in $\Db \cA_{\leq s}$ such that $\Pi_{<s}(E_s) \cong \Pi_{<s}(E_s') \cong \Pi_{<s}(\dgr_s)$. By Corollary~\ref{cor:parity-surjectivity}, the morphism
\[
\Hom_{\Db \cA_{\leq s}}(M,N) \to \Hom_{\Db(\cA_{\leq s} / \cA_{<s})}(\Pi_{<s}(M), \Pi_{<s}(N))
\]
is surjective, where $M$ and $N$ are either $E_s$ or $E_s'$. Hence there exist morphisms
\[
\varphi : E_s \to E_s', \qquad \psi : E_s' \to E_s
\]
such that $\Pi_{<s}(\varphi)$ and $\Pi_{<s}(\psi)$ are invertible. Then $\varphi \circ \psi$ does not belong to the maximal ideal of the local ring $\End(E_s')$, so it is invertible. Similarly $\psi \circ \varphi$ is invertible, so we conclude that $\varphi$ and $\psi$ are isomorphisms.

Now, let $E$ be an indecomposable parity object, and let $\scT$ be the smallest finite closed subset of $\scS$ such that $E$ belongs to the essential image of $\iota_{\scT}$. We denote by $E'$ the object of $\Db(\cA_{\scT})$ such that $E\cong \iota_\scT(E')$. 

We claim that $\scT$ admits a unique maximal element. Indeed, if $s$ and $t$ are distinct maximal elements in $\scT$, then setting $\scT':=\scT \smallsetminus \{s,t\}$ we can consider the quotient functor
\[
\Pi_{\scT'} : \Db(\cA_{\scT}) \to \Db(\cA_{\scT}/\cA_{\scT'}).
\]
By Corollary~\ref{cor:parity-surjectivity} the morphism
\[
\End_{\Db(\cA_{\scT})}(E') \to \End_{\Db(\cA_{\scT}/\cA_{\scT'})}(\Pi_{\scT'}(E'))
\]
is surjective, so that $\Pi_{\scT'}(E')$ is indecomposable. Now since $s$ and $t$ are not comparable, the category $\Db(\cA_{\scT}/\cA_{\scT'})$ is semisimple, so that $\Pi_{\scT'}(E')$ must be either isomorphic to some $\Pi_{\scT'}(\Lgr_s) \langle n \rangle [m]$ or some $\Pi_{\scT'}(\Lgr_s) \langle n \rangle [m]$. In any case, this contradicts the minimality of $\scT$.

Since $\scT$ admits a unique maximal element, it must be of the form $\{t \in \scS \mid t \leq s\}$ for some $s$. Moreover, the same argument as above shows that  $\Pi_{<s}(E') \cong \Pi_{<s}(\dgr_s) \langle n \rangle [m]$ for some $n,m \in \Z$. Since $E'$ is parity we must have $n=m$, and then by uniqueness, we have that $E_s$ exists and $E' \cong E_s \langle n \rangle [n]$.
\end{proof}

\section{Exotic parity complexes}
\label{sec:exotic-parity}

\subsection{Definitions}
\label{ss:def-exotic}

In this section we fix a split connected reductive group scheme $\Gp_\Z$ over $\Z$ with simply-connected derived subgroup, a Borel subgroup $\BGp_\Z \subset \Gp_\Z$, and a (split) maximal torus $\TGp_\Z \subset \BGp_\Z$. 
We also set $\bX:=X^*(\TGp_\Z)$.

We denote by $\fg_\Z$ and $\fb_\Z$ the respective Lie algebras of $\Gp_\Z$ and $\BGp_\Z$, and
consider the Springer resolution
\[
\tcN_\Z := \Gp_\Z \times^{\BGp_\Z} (\fg_\Z/\fb_\Z)^*,
\]
where $(-)^*$ means the dual $\Z$-module. This scheme admits a natural action of $\Gp_\Z \times (\Gm)_\Z$, where $\Gp_\Z$ acts via left multiplication on itself and $x \in (\Gm)_\Z$ acts by multiplication by $x^{-2}$ on $(\fg_\Z/\fb_\Z)^*$. We will then consider the derived category $\Db \Coh^{\Gp \times \Gm}(\tcN)_\Z$ of $(\Gp \times \Gm)_\Z$-equivariant coherent sheaves on $\tcN_\Z$; see~\cite[Appendix~A]{mr} for a review of equivariant coherent sheaves on such schemes. (Here and below, we will usually indicate coefficients only once as a subscript to simplify notations.) We denote by $\langle 1 \rangle$ the autoequivalence of $\Db \Coh^{\Gp \times \Gm}(\tcN)_\Z$ given by tensoring with the tautological rank-$1$ $(\Gm)_\Z$-module.

We will denote by $\Wf$ the Weyl group of $(\Gp_\Z, \TGp_\Z)$, and by $\Sf \subset \Wf$ the subset of simple reflections determined by $\BGp_\Z$. We also denote by $\Phi$ the root system of $(\Gp_\Z, \TGp_\Z)$, by $\Phi^+ \subset \Phi$ the system of positive roots defined as the $\TGp_\Z$-weights in $\fg_\Z/\fb_\Z$, and by $\Phis \subset \Phi$ the subset of simple roots. We denote by $\bX^+ \subset \bX$ the subset of dominant weights, by $\bY \subset \bX$ the root lattice, and set $\bY^+ := \bY \cap \bX^+$.  Note that our conventions make $\BGp_\Z$ the ``negative'' Borel subgroup.

We consider the affine Weyl group $W:=\Wf \ltimes \bY$, and the extended affine Weyl group $\Wext:=\Wf \ltimes \bX$. There exists a natural length function $\ell : \Wext \to \Z$ determined by
\begin{equation}\label{eqn:wext-length}
\ell(w \cdot t_\lambda) = \sum_{\substack{\alpha \in \Phi^+ \\ w(\alpha) \in \Phi^+}} |\langle \lambda, \alpha^\vee \rangle| + \sum_{\substack{\alpha \in \Phi^+ \\ w(\alpha) \in -\Phi^+}} |1 + \langle \lambda, \alpha^\vee \rangle|
\end{equation}
for $w \in \Wf$ and $\lambda \in \bX$ (where $t_\lambda$ denotes the image of $\lambda$ in $\Wext$).
Then $\Omega := \{w \in \Wext \mid \ell(w)=0\}$ is a subgroup of $\Wext$, and we have $\Wext = W \rtimes \Omega$. Moreover, if $S=\{w \in W \mid \ell(w)=1\}$, then $(W,S)$ is Coxeter system, with length function again given by~\eqref{eqn:wext-length}. Finally, the restriction of~\eqref{eqn:wext-length} to $\Wf$ is the length function for the Coxeter system $(\Wf,\Sf)$. We will denote by $\leq$ the Bruhat order on $W$ (which restricts to the Bruhat order on $\Wf$).

We will denote by $B_{\mathrm{ext}}$ the braid group associated with $\Wext$, i.e.~the group generated by symbols $T_w$ for $w \in \Wext$, and with relations $T_v T_w = T_{vw}$ if $\ell(vw)=\ell(v)+\ell(w)$. The main result of~\cite{br}  (see also~\cite[\S 3.3]{mr}) provides a (weak) right action of the group $B_{\mathrm{ext}}$ on $\Db \Coh^{\Gp \times \Gm}(\tcN)_\Z$; we will denote by
\[
\mathscr{J}_{b} : \Db \Coh^{\Gp \times \Gm}(\tcN)_\Z \simto \Db \Coh^{\Gp \times \Gm}(\tcN)_\Z
\]
the action of an element $b \in B_{\mathrm{ext}}$. (This functor is well defined up to isomorphism.)

For $\lambda \in \bX$, we will denote by $w_\lambda$ the unique shortest element in $\Wf t_\lambda \subset \Wext$. We also denote by $\fW \subset W$, resp.~$\fWext \subset \Wext$, the subset consisting of elements $w$ which are of minimal length in $\Wf w$. Then $\fWext = \{w_\lambda : \lambda \in \bX\}$ and $\fW = \{w_\lambda : \lambda \in \bY\}$. Finally, we will denote by $\fWf \subset W$ the subset of elements $w$ which are of minimal length in $\Wf w \Wf$. Then $\fWf = \{w_\lambda : \lambda \in -\bY^+\}$.

We fix once and for all a weight $\varsigma \in \bX$ such that $\langle \varsigma, \alpha^\vee \rangle = 1$ for any $\alpha \in \Phis$. (Our assumptions guarantee that such a weight exists, but it might not be unique.)

\subsection{Exotic standard and costandard sheaves}

In~\cite[\S 5.2]{mr2}, C.~Mautner and the third author consider certain objects
\[
\Delta_\Z(\lambda), \qquad \nabla_\Z(\lambda)
\]
in $\Db \Coh^{\Gp \times \Gm}(\tcN)_\Z$, for $\lambda \in \bX$. (In~\cite{mr2} these objects are in fact considered in a more limited setting, but these restrictions are not needed for the results we will use in the present paper.) By~\cite[Proposition~5.4]{mr2}, for $\lambda, \mu \in \bX$ and $m,n \in \Z$ we have
\begin{equation}
\label{eqn:vanishing-DN-Z}
\Hom_{\Db \Coh^{\Gp \times \Gm}(\tcN)_\Z}(\Delta_\Z(\lambda), \nabla_\Z(\mu) \langle n \rangle [m]) = \begin{cases}
\Z & \text{if $\lambda=\mu$ and $m=n=0$;}\\
0 & \text{otherwise.}
\end{cases}
\end{equation}

For any Noetherian commutative ring $\fR$ of finite global dimension we can also consider the base change $\Gp_\fR$, $\BGp_\fR$, $\tcN_\fR$ of $\Gp_\Z$, $\BGp_\Z$, $\tcN_\Z$ to $\fR$. 
Let
\[
\varphi_{\Z}^\fR : \Db \Coh^{\Gp \times \Gm}(\tcN)_\Z \to \Db \Coh^{\Gp \times \Gm}(\tcN)_\fR
\]
be the functor given by coherent pullback along $\tcN_\fR \to \tcN_\Z$.
We set
\[
\Delta_\fR(\lambda) := \varphi_\Z^\fR(\Delta_\Z(\lambda)), \qquad \nabla_\fR(\lambda) := \varphi_\Z^\fR(\nabla_\Z(\lambda)).
\]
From~\eqref{eqn:vanishing-DN-Z} it is not difficult to check (using e.g.~the same kind of arguments as in~\cite[Proof of Lemma~4.11]{mr}) that we have
\begin{equation}
\label{eqn:vanishing-DN-R}
\Hom_{\Db \Coh^{\Gp \times \Gm}(\tcN)_\fR}(\Delta_\fR(\lambda), \nabla_\fR(\mu) \langle n \rangle [m]) = \begin{cases}
\fR & \text{if $\lambda=\mu$, $m=n=0$;}\\
0 & \text{otherwise.}
\end{cases}
\end{equation}
If $\fS$ is a Noetherian commutative $\fR$-algebra of finite global dimension, we also have a functor
\[
\varphi_{\fR}^{\fS} : \Db \Coh^{\Gp \times \Gm}(\tcN)_\fR \to \Db \Coh^{\Gp \times \Gm}(\tcN)_\fS
\]
that sends $\Delta_\fR(\lambda)$ and $\nabla_\fR(\lambda)$ to $\Delta_\fS(\lambda)$ and $\nabla_\fS(\lambda)$ respectively. Moreover, the morphism
\begin{multline*}
\fS \otimes_\fR \Hom_{\Db \Coh^{\Gp \times \Gm}(\tcN)_\fR}(\Delta_\fR(\lambda), \nabla_\fR(\mu) \langle n \rangle [m]) \\
\to \Hom_{\Db \Coh^{\Gp \times \Gm}(\tcN)_\fS}(\Delta_\fS(\lambda), \nabla_\fS(\mu) \langle n \rangle [m])
\end{multline*}
induced by this functor is an isomorphism for any $\lambda, \mu \in \bX$ and $m,n \in \Z$.

In the following lemma, we say that an object $X$ in a triangulated category $\mathsf{D}$ admits a ``filtration'' with ``subquotients'' $X_1, \ldots, X_r$ if $X$ belongs to $\{[X_1]\} * \cdots * \{[X_r]\}$ in the notation of~\cite[\S 1.3.9]{bbd}.  In this situation, the number $r$ is called the \emph{length} of the ``filtration.''

\begin{lem}
\label{lem:Hom-filtrations-DN}
Let $\cF,\cG$ in $\Db \Coh^{\Gp \times \Gm}(\tcN)_\fR$. Assume that $\cF$ admits a ``filtration'' 
with ``subquotients'' of the form $\Delta_\fR(\lambda) \langle n \rangle [n]$ with $n \equiv \ell(w_\lambda) \pmod 2$, and that $\cG$ admits a ``filtration'' with ``subquotients'' of the form $\nabla_\fR(\lambda) \langle n \rangle [n]$ with $n \equiv \ell(w_\lambda) \pmod 2$. Then the $\fR$-module
\[
\Hom_{\Db\Coh^{\Gp \times \Gm}(\tcN)_\fR}(\cF,\cG \langle n \rangle [m])
\]
is free of finite rank, and vanishes
unless $m=n$ and $n$ is even. Moreover, for any $\fR$-algebra $\fS$ as above and any $n,m \in \Z$ the natural morphism
\begin{multline*}
\fS \otimes_\fR \Hom_{\Db\Coh^{\Gp \times \Gm}(\tcN)_\fR}(\cF,\cG \langle n \rangle [m]) \\
\to \Hom_{\Db\Coh^{\Gp \times \Gm}(\tcN)_\fS}(\varphi_{\fR}^{\fS}(\cF),\varphi_{\fR}^{\fS}(\cG) \langle n \rangle [m])
\end{multline*}
is an isomorphism.
\end{lem}

\begin{proof}
We prove the lemma by induction on the sum of the shortest lengths of ``filtrations'' of $\cF$ and $\cG$ as in the lemma. If this sum is at most $2$, then either the claim is obvious, or it follows from~\eqref{eqn:vanishing-DN-R}.

Otherwise, at least one of $\cF$ and $\cG$ admits no ``filtration'' of length less than $2$.  We will treat the case where this holds for $\cF$; the argument for $\cG$ is very similar.
Choose a ``filtration'' of $\cF$ of minimal length, and consider the first morphism
\[
\Delta_\fR(\lambda) \langle k \rangle [k] \to \cF
\]
in this ``filtration.'' (Here $\lambda \in \bX$ and $k\in \Z$ are such that $k \equiv \ell(w_\lambda) \pmod 2$.) Let $\cF'$ be the cone of this morphism, so that we have a distinguished triangle
\[
\Delta_\fR(\lambda) \langle k \rangle [k] \to \cF \to \cF' \xrightarrow{[1]}.
\]
Applying $\Hom(-,\cG \langle n \rangle [m])$ we obtain an exact sequence
\begin{multline*}
\Hom(\Delta_\fR(\lambda) \langle k \rangle [k], \cG \langle n \rangle [m-1]) \to
\Hom(\cF',\cG \langle n \rangle [m]) \to \Hom(\cF, \cG \langle n \rangle [m]) \\
\to \Hom(\Delta_\fR(\lambda) \langle k \rangle [k], \cG \langle n \rangle [m]) \to \Hom(\cF',\cG \langle n \rangle [m+1]).
\end{multline*}
By induction, the second and fourth terms vanish unless $n=m$ and $n$ is even. We deduce the same claim for the third term. If $n=m$ and $n$ is even, we also deduce that in the sequence above the first and last terms vanish, so that the second arrow is injective and the fourth one is surjective. Since the second and fourth terms are free of finite rank over $\fR$, this implies that the third term has the same property. Finally, the last assertion of the lemma follows from similar considerations and the 5-lemma.
\end{proof}

\subsection{Some ``wall-crossing'' functors}
\label{ss:wall-crossing}

For $s \in \Sf$, with associated simple root $\alpha_s$, we can consider the minimal standard parabolic subgroup $\PGp_{s,\Z} \subset \Gp_\Z$ associated with $s$, denote by $\fp_{s,\Z}$ its Lie algebra, and consider the ``parabolic Springer resolution'' $\tcN_{s,\Z}:= \Gp_\Z \times^{\PGp_{s,\Z}} (\fg_\Z / \fp_{s,\Z})^*$. There are natural maps
\[
\mathsf{e}_s : (\Gp_\Z / \BGp_\Z) \times_{\Gp_\Z / \PGp_{s,\Z}} \tcN_{s,\Z} \hookrightarrow \tcN_\Z, \quad \mu_s : (\Gp_\Z / \BGp_\Z) \times_{\Gp_\Z / \PGp_{s,\Z}} \tcN_{s,\Z} \twoheadrightarrow \tcN_{s,\Z}.
\]
Let us denote by
\begin{align*}
\Pi_s &: \Db \Coh^{\Gp \times \Gm}(\tcN)_\Z \to \Db \Coh^{\Gp \times \Gm}(\tcN_s)_\Z, \\
\Pi^s &: \Db \Coh^{\Gp \times \Gm}(\tcN_s)_\Z \to \Db \Coh^{\Gp \times \Gm}(\tcN)_\Z
\end{align*}
the functors defined by
\begin{align*}
\Pi_s(\cF) &= R(\mu_s)_* L(\mathsf{e}_s)^*(\cF \otimes_{\cO_{\tcN_\Z}} \cO_{\tcN_\Z}(-\varsigma)), \\
\Pi^s(\cG) &= R(\mathsf{e}_s)_* L(\mu_s)^*(\cG) \otimes_{\cO_{\tcN_\Z}} \cO_{\tcN_\Z}(\varsigma-\alpha_s) \langle -1 \rangle.
\end{align*}
Then we set
\[
\Psi_s := \Pi^s \Pi_s : \Db \Coh^{\Gp \times \Gm}(\tcN)_\Z \to \Db \Coh^{\Gp \times \Gm}(\tcN)_\Z.
\]

If $s \in S \smallsetminus \Sf$, then there exists $b \in B_\ext$ and $t \in \Sf$ such that $T_s = b T_t b^{-1}$ in $B_\ext$ (see~\cite[Lemma~6.1.2]{riche}). We fix such a pair once and for all, and set
\[
 \Psi_s := \mathscr{J}_{b^{-1}} \Psi_t \mathscr{J}_b : \Db \Coh^{\Gp \times \Gm}(\tcN)_\Z \to \Db \Coh^{\Gp \times \Gm}(\tcN)_\Z.
\]

If $s \in S$, the functor $\Psi_s$ fits into distinguished triangles of functors
\begin{gather}
\label{eqn:triangle-Psi-J}
\id \langle -1 \rangle [-1] \to \Psi_s \to \mathscr{J}_{T_s} \xrightarrow{[1]}, \\
\label{eqn:triangle-Psi-J-2}
\mathscr{J}_{(T_s)^{-1}} \to \Psi_s \to \id \langle 1 \rangle [1] \xrightarrow{[1]}.
\end{gather}
(See~\cite[(4.2)]{mr2} for the case $s \in \Sf$, and conjugate for the non-finite simple reflections.)

If $\fR$ is a commutative Noetherian ring of finite global dimension, then these constructions have obvious analogues for $\tcN_\fR$, and we will use similar notations in this setting.

\begin{rmk}
\label{rmk:Psi-adjoint}
It follows from~\cite[Lemma~9.4]{prinblock} that the functor $\Psi_s$ is self-adjoint, for any $s \in S$. The same remark applies to the endofunctor of $\Db \Coh^{\Gp \times \Gm}(\tcN)$ defined by the same formula.
\end{rmk}

\begin{lem}
\label{lem:filtration-Psi}
Let $\cF \in \Db \Coh^{\Gp \times \Gm}(\tcN)_\fR$, and assume that $\cF$ admits a ``filtration'' with ``subquotients'' of the form $\Delta_\fR(\lambda) \langle n \rangle [n]$ with $n \equiv \ell(w_\lambda) \pmod 2$ (resp.~of the form $\nabla_\fR(\lambda) \langle n \rangle [n]$ with $n \equiv \ell(w_\lambda) \pmod 2$). Then for any $s \in S$, $\Psi_s(\cF)\langle 1 \rangle [1]$ admits a ``filtration'' with ``subquotients'' of the form $\Delta_\fR(\lambda) \langle n \rangle [n]$ with $n \equiv \ell(w_\lambda) \pmod 2$ (resp.~of the form $\nabla_\fR(\lambda) \langle n \rangle [n]$ with $n \equiv \ell(w_\lambda) \pmod 2$).
\end{lem}

\begin{proof}
We treat the case of the objects $\Delta_\fR(\lambda)$; the other case is similar.

It follows from the associativity of the operation ``$*$'' (see~\cite[Lemme~1.3.10]{bbd})
that if $\cG$ and $\cK$ admit a ``filtration'' as in the lemma and if $\cH$ fits in a distinguished triangle
\[
\cG \to \cH \to \cK \xrightarrow{[1]},
\]
then $\cH$ also admits a ``filtration'' of the same form. This observation reduces the claim to the case $\cF=\Delta_\fR(\lambda)$.

Assume first that $w_\lambda s < w_\lambda$. Then $w_\lambda s$ is minimal in $\Wf w_\lambda s$, and hence of the form $w_\mu$ for some $\mu \in \bX$ with $\ell(w_\mu) = \ell(w_\lambda)-1$. We have 
 \[
  (T_{w_\mu^{-1}})^{-1} = (T_{s w_\lambda^{-1}})^{-1} = (T_{w_\lambda^{-1}})^{-1} \cdot T_s,
 \]
so by the definition of $\Delta_\fR(\lambda)$ we have
 \[
  \mathscr{J}_{T_s}(\Delta_\fR(\lambda)) \cong \mathscr{J}_{T_s} \circ \mathscr{J}_{(T_{w_\lambda^{-1}})^{-1}}(\cO_{\tcN_\fR}) \cong \mathscr{J}_{(T_{w_\mu^{-1}})^{-1}}(\cO_{\tcN_\fR}) \cong \Delta_\fR(\mu).
 \]
Using the triangle obtained by applying~\eqref{eqn:triangle-Psi-J} to $\Delta_\fR(\lambda)$, we deduce a distinguished triangle
\[
\Delta_\fR(\lambda) \to \Psi_s(\Delta_\fR(\lambda)) \langle 1 \rangle [1] \to \Delta_\fR(\mu) \langle 1 \rangle [1] \xrightarrow{[1]}.
\]
The claim follows.

Assume now that $w_\lambda s > w_\lambda$. In this case we have
 \[
  \mathscr{J}_{(T_s)^{-1}}(\Delta_\fR(\lambda)) \cong \mathscr{J}_{(T_{(w_\lambda s)^{-1}})^{-1}}(\cO_{\tcN_\fR}).
 \]
 If $w_\lambda s$ is minimal in $\Wf w_\lambda s$, then we deduce that
 \[
  \mathscr{J}_{(T_s)^{-1}}(\Delta_\fR(\lambda)) \cong \Delta_\fR(\mu)
 \]
 for some $\mu \in \bX$ with $\ell(w_\mu)=\ell(w_\lambda)+1$, so that we conclude as above using~\eqref{eqn:triangle-Psi-J-2} applied to $\Delta_\fR(\lambda)$.
 
 Finally, assume that $w_\lambda s > w_\lambda$ and that $w_\lambda s$ is not minimal in $\Wf w_\lambda s$. Then there exists $r \in \Sf$ such that $w_\lambda s = rw_\lambda$, so that
 \[
  \mathscr{J}_{(T_s)^{-1}}(\Delta_\fR(\lambda)) \cong \mathscr{J}_{(T_{(r w_\lambda)^{-1}})^{-1}}(\cO_{\tcN_\fR}) \cong \Delta_\fR(\lambda) \langle 1 \rangle,
 \]
 as in~\cite[(3.10)]{mr}.
 Hence the triangle obtained by applying~\eqref{eqn:triangle-Psi-J-2} to $\Delta_\fR(\lambda)$ takes the form
 \[
  \Delta_\fR(\lambda) \langle 1 \rangle \to \Psi_s(\Delta_\fR(\lambda)) \to \Delta_\fR(\lambda) \langle 1 \rangle [1] \xrightarrow{[1]}.
 \]
It is easy to check from the definitions that
\[
\End_{\Db \Coh^{\Gp \times \Gm}(\tcN)_\fR}(\Delta_\fR(\lambda)) = \fR,
\]
so that the connecting morphism in this triangle is $a \cdot \id$ for some $a \in \fR$.
 We claim that $a$ is invertible, so that $\Psi_s(\Delta_\fR(\lambda)) = 0$ in this case (which clearly implies the desired result).
 
 In fact, by compatibility of all our constructions with change of scalars, it suffices to prove this when $\fR$ is 
an algebraically closed field.
In this case, if $a$ is not invertible, then $a = 0$, so that we obtain an isomorphism
 \[
  \Psi_s(\Delta_\fR(\lambda)) \cong \Delta_\fR(\lambda) \langle 1 \rangle \oplus \Delta_\fR(\lambda) \langle 1 \rangle [1].
 \]
If $s \in \Sf$, then this is absurd since the restriction of any object $\Delta_\fR(\lambda)$ to the inverse image of the open orbit in the nilpotent cone $\cN_\fR$ is nonzero (because the functors in the braid group action send objects with nonzero restriction to this inverse image to objects with the same property), while the restriction of $\Psi_s(\Delta_\fR(\lambda))$ is $0$. If $s \in S \smallsetminus \Sf$, let $b$ and $t$ be as above. Then we obtain that
 \[
  \Psi_t \circ \mathscr{J}_b(\Delta_\fR(\lambda)) \cong \mathscr{J}_b(\Delta_\fR(\lambda)) \langle 1 \rangle \oplus \mathscr{J}_b(\Delta_\fR(\lambda)) \langle 1 \rangle [1],
 \]
 and we obtain a contradiction as above.
\end{proof}

We will use the term \emph{expression} to mean any word in $S$.
Given an expression $\uw=(s_1, \ldots, s_r)$ and an element $\omega \in \Omega$ we set
\[
 \PEx^\fR(\omega, \uw) := \Psi_{s_r} \circ \cdots \circ \Psi_{s_1} \circ \mathscr{J}_{T_\omega}(\cO_{\tcN_\fR}).
\]
It is clear that if $\fS$ is a commutative Noetherian $\fR$-algebra of finite global dimension, we have
\begin{equation}
\label{eqn:isom-varphi-E}
\varphi_{\fR}^{\fS}(\PEx^\fR(\omega, \uw)) \cong \PEx^\fS(\omega, \uw).
\end{equation}
(Here and several times below we use the fact that the functor $\varphi_{\fR}^{\fS}$ is compatible in the appropriate way with pullback and pushforward functors; this is obvious for pullback functors and follows from the general form of the flat base change theorem---as in~\cite[Theorem~3.10.3]{lipman}---for pushforward functors.)

\begin{cor}
\label{cor:Hom-E-R}
\leavevmode
\begin{enumerate}
\item
Let $\uw$ be an expression, let $\omega \in \Omega$, and let $\lambda \in \bX$. For any $n,m \in \Z$, the $\fR$-module
\[
\Hom(\Delta_\fR(\lambda), \PEx^\fR(\omega, \uw) \langle n \rangle [m])
\]
is free of finite rank, and
vanishes unless $n=m$ and $n \equiv \ell(\uw) + \ell(w_\lambda) \pmod 2$. Moreover, for any $\fS$ as above, the natural morphism
\[
\fS \otimes_{\fR} \Hom(\Delta_\fR(\lambda), \PEx^\fR(\omega, \uw) \langle n \rangle [m]) \to \Hom(\Delta_\fS(\lambda), \PEx^\fS(\omega, \uw) \langle n \rangle [m])
\]
is an isomorphism.
\item
Let $\uw$ be an expression, let $\omega \in \Omega$, and let $\lambda \in \bX$. For any $n,m \in \Z$, the $\fR$-module
\[
\Hom(\PEx^\fR(\omega, \uw), \nabla_\fR(\lambda) \langle n \rangle [m])
\]
is free of finite rank, and
vanishes unless $n=m$ and $n \equiv \ell(\uw) + \ell(w_\lambda) \pmod 2$. Moreover, for any $\fS$ as above, the natural morphism
\[
\fS \otimes_{\fR} \Hom(\PEx^\fR(\omega, \uw), \nabla_\fR(\lambda) \langle n \rangle [m]) \to \Hom(\PEx^\fS(\omega, \uw), \nabla_\fS(\lambda) \langle n \rangle [m])
\]
is an isomorphism.
\item
Let $\uw,\uw'$ be expressions, and let $\omega, \omega' \in \Omega$. For any $n,m \in \Z$, the $\fR$-module
\[
\Hom(\PEx^\fR(\omega, \uw), \PEx^\fR(\omega', \uw') \langle n \rangle [m])
\]
is free of finite rank. Moreover, for any $\fS$ as above, the natural morphism
\[
\fS \otimes_{\fR} \Hom(\PEx^\fR(\omega, \uw), \PEx^\fR(\omega', \uw') \langle n \rangle [m]) \to \Hom(\PEx^\fS(\omega, \uw), \PEx^\fS(\omega', \uw') \langle n \rangle [m])
\]
is an isomorphism.
\end{enumerate}
\end{cor}

\begin{proof}
Fix an expression $\uw$ of length $r$ and $\omega \in \Omega$.
If $\lambda \in \bX$ is the unique weight such that $\omega=w_\lambda$, then $\mathscr{J}_{T_\omega}(\cO_{\tcN_\fR}) = \Delta_\fR(\lambda) = \nabla_\fR(\lambda)$. Then, Lemma~\ref{lem:filtration-Psi} implies that $\PEx^\fR(\omega, \uw) \langle r \rangle [r]$ admits both a ``filtration'' with ``subquotients'' of the form $\Delta_\fR(\lambda) \langle n \rangle [n]$ with $n \equiv \ell(w_\lambda) \pmod 2$ and a ``filtration'' with ``subquotients'' of the form $\nabla_\fR(\lambda) \langle n \rangle [n]$ with $n \equiv \ell(w_\lambda) \pmod 2$. We deduce the desired claims using Lemma~\ref{lem:Hom-filtrations-DN}.
\end{proof}

\subsection{Parity exotic sheaves}
\label{ss:parity-exotic}

In this subsection we assume that $\fR$ is field, that we will denote $\bk$ to avoid confusion. We choose an order $\leq'$ on $\bX$ as in~\cite[\S 2.5]{mr}. Recall that this order satisfies in particular
\[
w_{\lambda} \leq w_\mu \quad \Rightarrow \quad \lambda \leq' \mu.
\]

In case $\bk$ is algebraically closed, it is proved in~\cite{mr} that the objects $\{\nabla_\bk(\lambda) : \lambda \in \bX\}$ form a graded exceptional sequence in the sense of~\cite[\S 2.1.5]{bezru}, for the order $\leq'$ on $\bX$, with dual exceptional sequence $\{\Delta_\bk(\lambda) : \lambda \in \bX\}$. From this case it is not difficult to deduce that this property holds for any field of coefficients. Using~\cite[Proposition~4]{bezru}, we deduce that if we define $D^{\leq 0}$, resp.~$D^{\geq 0}$, as the full subcategory of $\Db \Coh^{\Gp \times \Gm}(\tcN)_\bk$ generated under extension by the objects of the form $\Delta_\bk(\lambda) \langle n \rangle [m]$ with $m \geq 0$, resp.~by the objects of the form $\nabla_\bk(\lambda) \langle n \rangle [m]$ with $m \leq 0$, then $(D^{\leq 0}, D^{\geq 0})$ is a t-structure on $\Db \Coh^{\Gp \times \Gm}(\tcN)_\bk$. This t-structure will be called the \emph{exotic t-structure}, and its heart will be denoted $\Ex^{\Gp \times \Gm}(\tcN)_\bk$. By~\cite[Corollary~3.11]{mr}, the objects $\Delta_\bk(\lambda)$ and $\nabla_\bk(\lambda)$ belong to $\Ex^{\Gp \times \Gm}(\tcN)_\bk$. (Again, in~\cite{mr} it is assumed that $\bk$ is algebraically closed, but the general case follows.) It follows that $\Ex^{\Gp \times \Gm}(\tcN)_\bk$ is a graded highest weight category in the sense of Definition~\ref{defn:qhered} (with weight poset $(\bX,\leq')$), and that the realization functor
\[
\Db \Ex^{\Gp \times \Gm}(\tcN)_\bk \to \Db \Coh^{\Gp \times \Gm}(\tcN)_\bk
\]
is an equivalence of categories. (See also~\cite[\S8]{arider} for another approach to these claims.)   For $\lambda \in \bX$, we will denote by $\fL^\bk(\lambda)$ the simple object of $\Ex^{\Gp \times \Gm}(\tcN)_\bk$ parametrized by $\lambda$, i.e.~the image of the only nonzero morphism $\Delta_\bk(\lambda) \to \nabla_\bk(\lambda)$ (up to scalar).

Using this equivalence, we can consider parity objects in this category, in the sense of~\S\ref{ss:parity-objects}, for the function $\dag : \bX \to \{0,1\}$ defined by the property that $\dag(\lambda) \equiv \ell(w_\lambda) \pmod 2$. (For simplicity, this function will be dropped from the notation.) The indecomposable parity object associated with $\lambda$ as in Proposition~\ref{prop:classification-parity} will be denoted by $\PEx^\bk_\lambda$.

\begin{prop}
\label{prop:parity-BS}
\leavevmode
 \begin{enumerate}
  \item 
  \label{it:Psi-parity}
  For any $\omega$ and $\uw$, the object $\PEx^\bk(\omega, \uw)$ is a parity object in $\Db \Ex^{\Gp \times \Gm}(\tcN)_\bk$.
  \item
  \label{it:Psi-parity-2}
  If $\lambda \in \bX$ and if $w_\lambda = \omega s_1 \cdots s_r$ is a reduced expression, then $\PEx^\bk_\lambda$ is a direct summand of $\PEx^\bk(\omega, (s_1, \ldots, s_r))$. Moreover, all the other direct summands of this object are the form $\PEx^\bk_\mu \langle m \rangle[m]$ with $m \in \Z$ and $\mu <' \lambda$.
 \end{enumerate}
\end{prop}

\begin{proof}
\eqref{it:Psi-parity}
It follows from Corollary~\ref{cor:Hom-E-R} that $\PEx^\bk(\omega, \uw) \langle \ell(\uw) \rangle [\ell(\uw)]$ is even, and hence that $\PEx^\bk(\omega, \uw)$ is parity.
 
 \eqref{it:Psi-parity-2}
 It follows from the proof of Lemma~\ref{lem:filtration-Psi} that
 \[
 \Hom(\PEx(\omega,(s_1, \ldots, s_r)), \nabla(\mu) \langle n \rangle [m]) \neq 0 \quad \Rightarrow \quad w_\mu \leq w_\lambda,
 \]
 which implies that $\mu \leq' \lambda$. Moreover, if $\mu=\lambda$, this space is $1$-dimensional if $n=m=0$ and $0$ otherwise. The claim follows.
\end{proof}

\begin{cor}
\label{cor:parities-extension-scalars}
Let $\bk'$ be a field extension of $\bk$.
\begin{enumerate}
\item
\label{it:parities-extension-scalars-1}
The functor $\varphi_{\bk}^{\bk'}$ sends parity objects to parity objects. Moreover,
for any parity objects $\cF$ and $\cG$ in $\Db \Coh^{\Gp \times \Gm}(\tcN)_\bk$, the natural morphism
\[
\bk' \otimes_{\bk} \Hom_{\Db \Coh^{\Gp \times \Gm}(\tcN)_\bk}(\cF,\cG) \to \Hom_{\Db \Coh^{\Gp \times \Gm}(\tcN)_{\bk'}}(\varphi_{\bk}^{\bk'}(\cF),\varphi_{\bk}^{\bk'}(\cG))
\]
is an isomorphism.
\item
\label{it:parities-extension-scalars-2}
For any $\lambda \in \bX$, we have
$\varphi_{\bk}^{\bk'}(\PEx_\lambda^\bk) \cong \PEx_\lambda^{\bk'}$.
\end{enumerate}
\end{cor}

\begin{proof}
\eqref{it:parities-extension-scalars-1}
It follows from Proposition~\ref{prop:classification-parity} and Proposition~\ref{prop:parity-BS} that $\cF$ and $\cG$ are direct sums of direct summands of objects of the form $\PEx^\bk(\omega, \uw) \langle n \rangle [n]$. Then the claim follows from~\eqref{eqn:isom-varphi-E} and Corollary~\ref{cor:Hom-E-R}.

\eqref{it:parities-extension-scalars-2}
This proof is copied from~\cite[Lemma~3.8]{williamson}.
By~\eqref{it:parities-extension-scalars-1}, $\varphi_{\bk}^{\bk'}(\PEx_\lambda^\bk)$ is a parity object. Hence to conclude it suffices to prove that it is indecomposable. Now we remark that the functor $\Pi_{<'\lambda}$ considered in Proposition~\ref{prop:classification-parity} induces a surjection
\[
\End_{\Db \Coh^{\Gp \times \Gm}(\tcN)_\bk}(\PEx_\lambda^\bk) \twoheadrightarrow \bk.
\]
Since the left-hand side is a local ring, the kernel of this surjection is the Jacobson radical of this algebra, so it is nilpotent. Applying $\bk' \otimes_{\bk} (-)$ and using~\eqref{it:parities-extension-scalars-1} we obtain a surjective morphism
\[
\End_{\Db \Coh^{\Gp \times \Gm}(\tcN)_{\bk'}}(\varphi_{\bk}^{\bk'}(\PEx_\lambda^\bk)) \to \bk'
\]
whose kernel is nilpotent. Hence the left-hand side is again a local ring, proving that $\varphi_{\bk}^{\bk'}(\PEx_\lambda^\bk)$ is indecomposable.
\end{proof}

\subsection{Integral form of parity exotic sheaves}

\begin{prop}
\label{prop:indec-parity-integers}
For any $\lambda \in \bX$,
there exists $M \in \Z_{\geq 1}$ (depending on $\lambda$) and an object
\[
\PEx_\lambda^{\Z[\frac{1}{M!}]} \quad \in \Db \Coh^{\Gp \times \Gm}(\tcN)_{\Z[\frac{1}{M!}]}
\]
such that for any field $\bk$ of characteristic either $0$ or greater than $M$ we have
\[
\varphi_{\Z[\frac{1}{M!}]}^\bk \left( \PEx_\lambda^{\Z[\frac{1}{M!}]} \right) \cong \PEx_\lambda^\bk.
\]
\end{prop}

\begin{proof}
Choose a reduced expression $w_\lambda = \omega s_1 \cdots s_r$. Then $\PEx^\Q_\lambda$ is a direct summand in $\PEx^\Q(\omega, (s_1, \ldots, s_r))$ by Proposition~\ref{prop:parity-BS}\eqref{it:Psi-parity-2}. We denote by $f$ the idempotent in $\End ( \PEx^\Q(\omega, (s_1, \ldots, s_r)))$ given by projecting to $\PEx^\Q_\lambda$. By Corollary~\ref{cor:Hom-E-R} we have a canonical isomorphism
\[
\Q \otimes_\Z \End \left( \PEx^\Z(\omega, (s_1, \ldots, s_r)) \right) \cong \End \left( \PEx^\Q(\omega, (s_1, \ldots, s_r)) \right).
\]
Hence there exists an $M$ such that $f$ belongs to
\[
\Z \left[ \frac{1}{M!} \right] \otimes_\Z \End \left( \PEx^\Z(\omega, (s_1, \ldots, s_r)) \right) \cong \End \left( \PEx^{\Z[\frac{1}{M!}]}(\omega, (s_1, \ldots, s_r)) \right).
\]
Now the category $\Db \Coh^{\Gp \times \Gm}(\tcN)_{\Z[\frac{1}{M!}]}$ is Karoubian by \cite[Corollary~2.10]{balmer-schlichting}. In particular, the object $\PEx^{\Z[\frac{1}{M!}]}(\omega, (s_1, \ldots, s_r))$ admits a direct summand $\PEx_\lambda^{\Z[\frac{1}{M!}]}$ such that the projection to $\PEx_\lambda^{\Z[\frac{1}{M!}]}$ is $f$. It is clear that
\[
\varphi_{\Z[\frac{1}{M!}]}^\Q(\PEx_\lambda^{\Z[\frac{1}{M!}]}) \cong \PEx_\lambda^{\Q}.
\]
Using Corollary~\ref{cor:parities-extension-scalars}\eqref{it:parities-extension-scalars-2}, we deduce that
\begin{equation}
\label{eqn:extension-scalars-Elambda}
\varphi_{\Z[\frac{1}{M!}]}^\bk \left( \PEx_\lambda^{\Z[\frac{1}{M!}]} \right) \cong \PEx_\lambda^\bk.
\end{equation}
for any field $\bk$ of characteristic $0$.

Using Corollary~\ref{cor:parities-extension-scalars}\eqref{it:parities-extension-scalars-2} again, to conclude the proof it suffices to prove that for any $p>M$ we have
\[
\varphi_{\Z[\frac{1}{M!}]}^{\F_p} \left( \PEx_\lambda^{\Z[\frac{1}{M!}]} \right) \cong \PEx_\lambda^{\F_p}.
\]
And it is not difficult to see that for this it suffices to prove that the left-hand side is indecomposable. Assume for a contradiction that this is not the case. We have
\[
\F_p \otimes_{\Z_p} \End \left( \varphi_{\Z[\frac{1}{M!}]}^{\Z_p} \Bigl( \PEx_\lambda^{\Z[\frac{1}{M!}]} \Bigr) \right) \cong \End \left( \varphi_{\Z[\frac{1}{M!}]}^{\F_p} \Bigl( \PEx_\lambda^{\Z[\frac{1}{M!}]} \Bigr) \right)
\]
by Corollary~\ref{cor:parities-extension-scalars}. By assumption the right-hand side admits a nontrivial idempotent, which can then be lifted to $\End \left( \varphi_{\Z[\frac{1}{M!}]}^{\Z_p} \Bigl( \PEx_\lambda^{\Z[\frac{1}{M!}]} \Bigr) \right)$, by, say,~\cite[Theorem~21.31]{lam}. We then obtain a nontrivial idempotent in
\[
\Q_p \otimes_{\Z_p} \End \left( \varphi_{\Z[\frac{1}{M!}]}^{\Z_p} \Bigl( \PEx_\lambda^{\Z[\frac{1}{M!}]} \Bigr) \right) \cong \End \left( \varphi_{\Z[\frac{1}{M!}]}^{\Q_p} \Bigl( \PEx_\lambda^{\Z[\frac{1}{M!}]} \Bigr) \right),
\]
contradicting the case $\bk=\Q_p$ of~\eqref{eqn:extension-scalars-Elambda}.
\end{proof}

\subsection{Relation with tilting perverse sheaves on affine Grassmannians}
\label{ss:relation-Gr}

From now on we assume that $\Gp_\Z$ satisfies the conditions of~\cite[\S 4.2]{mr2}. In other words, $\Gp_\Z$ is a product of (simply connected) quasi-simple groups and general linear groups $\mathrm{GL}_{n,\Z}$. We denote by $N$ the product of all the prime numbers that are not very good for some quasi-simple factor of $\Gp_\Z$, and set $\fR:=\Z[\frac{1}{N}]$.

Let $\TGp^\vee$ be the complex torus which is dual to $\TGp_\C$, and let
$\Gp^\vee$ be the complex connected reductive group with maximal torus $\TGp^\vee$ which is Langlands dual to $\Gp_\Z$. We let $\BGp^\vee_+$ be the Borel subgroup of $\Gp^\vee$ containing $\TGp^\vee$ whose roots are the positive coroots of $(\Gp_\Z, \TGp_\Z)$, and denote by $\Iw \subset \Gp^\vee(\scO)$ the corresponding Iwahori subgroup, where $\scO:=\C[ \hspace{-1pt} [t] \hspace{-1pt} ]$. Then if $\bk$ is a field we can consider the affine Grassmannian $\Gr$ of $\Gp^\vee$, and the category $\Par_{(\Iw)}(\Gr,\bk)$ of $\Iw$-constructible parity complexes on $\Gr$ with coefficients in $\bk$ in the sense of~\cite{jmw}. The \emph{mixed derived category} $\Dmix_{(\Iw)}(\Gr,\bk)$ is defined as the bounded homotopy category of the additive category $\Par_{(\Iw)}(\Gr,\bk)$ (see~\cite{modrap2}). The Tate twist $\langle 1 \rangle$ is defined as $\{-1\}[1]$, where $\{n\}$ is the cohomological shift by $n$ in $\Par_{(\Iw)}(\Gr,\bk)$ (viewed as a subcategory of $\Db_{(\Iw)}(\Gr, \bk)$).

The main result of~\cite{mr2} is the construction of an equivalence of additive categories between $\Par_{(\Iw)}(\Gr,\bk)$ and the category of tilting objects in $\Ex^{\Gp \times \Gm}(\tcN)_\bk$ which intertwines the functors $\{1\}$ and $\langle -1 \rangle$, under the assumption that there exists a ring morphism $\fR \to \bk$. (See~\cite[Remark~11.3]{prinblock} for comments on the difference of conventions between the present paper---where the conventions are similar to those in~\cite{prinblock}---and~\cite{mr2}. Note also that in~\cite{mr2} the field is assumed to be algebraically closed; but the results above make it easy to generalize the theorem to any field.) Passing to homotopy categories we deduce an equivalence of triangulated categories
\[
\Upsilon : \Dmix_{(\Iw)}(\Gr,\bk) \simto \Db \Coh^{\Gp \times \Gm}(\tcN)_\bk
\]
which satisfies $\Upsilon \circ \langle 1 \rangle \cong \langle 1 \rangle [1] \circ \Upsilon$.  (See also~\cite{arider} for a different construction of this equivalence.)

Consider the perverse t-structure on $\Dmix_{(\Iw)}(\Gr,\bk)$ (see~\cite{modrap2}). Its heart, which we will denote $\Perv^\mix_{(\Iw)}(\Gr,\bk)$, has a natural structure of graded highest weight category, and the realization functor
\[
\Db \Perv^\mix_{(\Iw)}(\Gr,\bk) \to \Dmix_{(\Iw)}(\Gr,\bk)
\]
is an equivalence of categories. If $\cJ_!(\lambda)$ and $\cJ_*(\lambda)$ denote the (normalized) standard and costandard objects associated with $\lambda$ respectively, we have
\begin{equation}
\label{eqn:Phi-standard}
\Upsilon(\cJ_!(\lambda)) \cong \Delta_\bk(\lambda), \quad \Upsilon(\cJ_*(\lambda)) \cong \nabla_\bk(\lambda)
\end{equation}
for any $\lambda \in \bX$ (see~\cite[Theorem~1.2]{mr2} or~\cite[Theorem~8.3]{arider}).

The indecomposable tilting objects in $\Perv^\mix_{(\Iw)}(\Gr,\bk)$ are parametrized in a natural way by $\bX \times \Z$; we denote by $\cT(\lambda)$ the indecomposable object associated with $(\lambda,0)$ (so that the indecomposable object associated with $(\lambda,n)$ is $\cT(\lambda) \langle n \rangle$).

\begin{prop}
\label{prop:Phi-Tilt-Par}
For any $\lambda \in \bX$, we have $\Upsilon(\cT(\lambda)) \cong \PEx^\bk_\lambda$.
\end{prop}

\begin{proof}
For $\mu \in \bX$ we have
\begin{multline*}
\Hom_{\Db \Coh^{\Gp \times \Gm}(\tcN)_\bk}(\Delta_\bk(\mu), \Upsilon(\cT(\lambda)) \langle n \rangle [m]) \\
\overset{\eqref{eqn:Phi-standard}}{=} \Hom_{\Db \Coh^{\Gp \times \Gm}(\tcN)_\bk}(\Upsilon(\cJ_!(\mu)), \Upsilon(\cT(\lambda)\langle n \rangle) [m-n]) \\
\cong \Hom_{\Dmix_{(\Iw)}(\Gr,\bk)}(\cJ_!(\mu), \cT(\lambda) \langle n \rangle [m-n]).
\end{multline*}
Hence this space vanishes unless $m=n$. Using~\cite[Lemma~3.17]{modrap2}, it also vanishes unless $n \equiv \ell(w_\lambda)-\ell(w_\mu) \pmod 2$. The same arguments can be used to show that
\[
\Hom_{\Db \Coh^{\Gp \times \Gm}(\tcN)_\bk}(\Upsilon(\cT(\lambda)) \langle n \rangle [m], \nabla_\bk(\mu))
\]
vanishes unless $m=n$ and $n \equiv \ell(w_\lambda)-\ell(w_\mu) \pmod 2$. Thus $\Upsilon(\cT(\lambda))$ is a parity object. It is clearly indecomposable, and hence of the form $\PEx^\bk_\mu \langle n \rangle [n]$ for some $(\mu,n) \in \bX \times \Z$. It is easily checked that $\mu=\lambda$ and $n=0$, and the claim follows.
\end{proof}

\subsection{The case of characteristic \texorpdfstring{$0$}{0}}

We continue with the assumptions of~\S\ref{ss:relation-Gr}.

\begin{thm}
\label{thm:parities-char-0}
Assume that $\bk$ has characteristic $0$. Then
for any $\lambda \in \bX$, we have $\PEx^\bk_\lambda \cong \fL^\bk(\lambda)$.
\end{thm}

\begin{proof}
By definition of $\PEx^\bk_\lambda$, all its cohomology objects have their composition factors of the form $\fL^\bk(\mu) \langle n \rangle$ with $\mu \leq' \lambda$, and we have
\begin{multline*}
\Hom_{\Db \Coh^{\Gp \times \Gm}(\tcN)_\bk}(\Delta_\bk(\lambda), \PEx^\bk_\lambda \langle n \rangle [m]) = \\
\Hom_{\Db \Coh^{\Gp \times \Gm}(\tcN)_\bk}(\PEx^\bk_\lambda, \nabla_\bk(\lambda) \langle n \rangle [m])
= \begin{cases}
\bk & \text{if $n=m=0$;} \\
0 & \text{otherwise.}
\end{cases}
\end{multline*}
Hence $\fL^\bk(\lambda) \langle n \rangle$ appears once as a composition factor of a cohomology object of $\PEx^\bk_\lambda$ if $n=0$ (in degree $0$), and never if $n \neq 0$.
Below we will prove that if $\mu \neq \lambda$ we have
\[
\Hom_{\Db \Coh^{\Gp \times \Gm}(\tcN)_\bk}(\Delta_\bk(\mu), \PEx^\bk_\lambda \langle n \rangle [m]) = \Hom_{\Db \Coh^{\Gp \times \Gm}(\tcN)_\bk}(\PEx^\bk_\lambda, \nabla_\bk(\mu) \langle n \rangle [m])=0
\]
unless $m > 0$. 
This will imply that $\PEx^\bk_\lambda$ is in the heart of the exotic t-structure, and that it has no simple subobject or quotient of the form $\fL^\bk(\mu) \langle n \rangle$ with $\mu \neq \lambda$. 
This will conclude the proof.

So, we fix $\mu \neq \lambda$.
By a simple change-of-scalars arguments based on Lemma~\ref{lem:Hom-filtrations-DN} and Corollary~\ref{cor:parities-extension-scalars}, we can assume that $\bk=\overline{\Q_p}$ for some prime number $p$. As in the proof of Proposition~\ref{prop:Phi-Tilt-Par}, we have
\[
\Hom_{\Db \Coh^{\Gp \times \Gm}(\tcN)_\bk}(\Delta_\bk(\mu), \PEx^\bk_\lambda \langle n \rangle [m]) \cong \Hom_{\Dmix_{(\Iw)}(\Gr,\bk)}(\cJ_!(\mu), \cT(\lambda) \langle n \rangle [m-n]),
\]
and this space vanishes unless $m=n$. The objects $\cT(\lambda)$ are studied in the case of coefficients $\overline{\Q_p}$ in~\cite{ar-koszul}, based on the results of~\cite{yun}. In particular, it is explained in~\cite[\S 10.2]{ar-koszul} that these objects satisfy Yun's ``condition (W).'' This condition implies that if $j_\mu$ is the inclusion of the $\Iw$-orbit on $\Gr$ associated with $\mu$, and $j_\mu^!$ is the corresponding functor constructed in~\cite{modrap2}, then $j_\mu^! \cT(\lambda)$ is a direct sum of objects of the form $\underline{\bk} \langle i \rangle \{k\}$ with $i<0$, and then (using adjunction and the fact that $\cJ_!(\mu) = (j_{\mu})_! \underline{\bk} \{\ell(w_\mu)\}$) that
\[
\Hom_{\Db \Coh^{\Gp \times \Gm}(\tcN)_\bk}(\Delta_\bk(\mu), \PEx^\bk_\lambda \langle m \rangle [m]) \cong \Hom_{\Dmix_{(\Iw)}(\Gr,\bk)}(\cJ_!(\mu), \cT(\lambda) \langle m \rangle)
\]
vanishes unless $m>0$. One proves similarly that
\[
 \Hom_{\Db \Coh^{\Gp \times \Gm}(\tcN)_\bk}(\PEx^\bk_\lambda, \nabla_\bk(\mu) \langle n \rangle [m])=0
 \]
unless $m>0$ (and $n=m$), and the proof is complete.
\end{proof}

\begin{cor}
\label{cor:parity-simple}
Let $\lambda \in \bX$, and let $M$ be as in Proposition~{\rm \ref{prop:indec-parity-integers}}. If $\bk$ is a field such that $\mathrm{char}(\bk)>M$, we have $\PEx^\bk_\lambda \cong \fL^\bk(\lambda)$.
\end{cor}

\begin{proof}
Recall the object $\PEx_\lambda^{\Z[\frac{1}{M!}]}$ from Proposition~\ref{prop:indec-parity-integers}. By construction this object is a direct summand of an object of the form $\PEx^{\Z[\frac{1}{M!}]}(\omega, \uw)$. Hence, by Corollary~\ref{cor:Hom-E-R}, for any $\mu \in \bX$ and $n,m \in \Z$ the $\Z[\frac{1}{M!}]$-module $\Hom(\Delta_{\Z[\frac{1}{M!}]}(\mu), \PEx^{\Z[\frac{1}{M!}]}_\lambda \langle n \rangle [m])$ is free, and if $\mathrm{char}(\bk)$ is $0$ or greater than $M$, we have
\[
\bk \otimes_{\Z[\frac{1}{M!}]} \Hom \left( \Delta_{\Z[\frac{1}{M!}]}(\mu), \PEx^{\Z[\frac{1}{M!}]}_\lambda \langle n \rangle [m] \right) \cong \Hom(\Delta_{\bk}(\mu), \PEx^{\bk}_\lambda \langle n \rangle [m]).
\]
By Theorem~\ref{thm:parities-char-0}, if $\mathrm{char}(\bk)=0$ the right-hand side vanishes unless $\lambda=\mu$ and $n=m=0$ or $m=n$ and $m>0$. It follows that the same property holds if $\mathrm{char}(\bk)>M$, and as in the proof of Theorem~\ref{thm:parities-char-0} this implies that $\PEx^\bk_\lambda \cong \fL^\bk(\lambda)$.
\end{proof}

\subsection{Comparison between characteristic \texorpdfstring{$0$}{0} and positive characteristic}

Co\-rollary~\ref{cor:parity-simple} and its proof show that in characteristic larger than some bound (that is not explicitly known), the object $\PEx^\bk_\lambda$ behaves ``as in characteristic $0$,'' and in particular coincides with $\fL^\bk(\lambda)$. With additional information it is sometimes possible to prove this property for a given characteristic.
We assume that $\mathrm{char}(\bk)>0$.

\begin{lem}
\label{lem:criterion-E-L}
For any $\lambda,\mu \in \bX$ and $m \in \Z$ we have
\[
\dim_\bk \bigl( \Hom(\Delta_\bk(\mu), \PEx^\bk_\lambda \langle m \rangle [m] ) \bigr) \geq \dim_\C \bigl( \Hom(\Delta_\C(\mu), \PEx^\C_\lambda \langle m \rangle [m] ) \bigr).
\]
Moreover, given $\lambda \in \bX$, if this inequality is an equality for all $\mu \in \bX$ and $m \in \Z$, then $\PEx^\F_\lambda \cong \fL^\F(\lambda)$.
\end{lem}

\begin{proof}
Using Corollary~\ref{cor:parities-extension-scalars}\eqref{it:parities-extension-scalars-2} we can assume that $\bk=\F_p$ for some prime number $p$.
Then, consider some reduced expression $w_\lambda = \omega s_1 \cdots s_r$. By Proposition~\ref{prop:parity-BS}\eqref{it:Psi-parity-2} $\PEx_\lambda^\bk$ is a direct summand of $\PEx^\bk(\omega,(s_1, \ldots, s_r))$. Using Corollary~\ref{cor:Hom-E-R}, an idempotent-lifting argument (as in the proof of Proposition~\ref{prop:indec-parity-integers}) shows that there exists a direct summand $\PEx_\lambda^{\Z_p}$ of $\PEx^{\Z_p}(\omega,(s_1, \ldots, s_r))$ such that $\varphi_{\Z_p}^\bk(\PEx_\lambda^{\Z_p}) \cong \PEx_\lambda^\bk$. Then, choosing some ring morphism $\Z_p \to \C$, the object $\varphi_{\Z_p}^{\C}(\PEx_\lambda^{\Z_p})$ is a parity object, which by Corollary~\ref{cor:Hom-E-R} satisfies
\[
\dim_\bk \bigl( \Hom(\Delta_\F(\mu), \PEx^\F_\lambda \langle m \rangle [m] ) \bigr) = \dim_\C \bigl( \Hom(\Delta_\C(\mu), \varphi_{\Z_p}^{\C}(\PEx_\lambda^{\Z_p}) \langle m \rangle [m] ) \bigr)
\]
for any $\mu \in \bX$ and $m \in \Z$. From this we see in particular that $\lambda$ is maximal among the weights such that the right-hand side is nonzero for some $m$, and that
\[
\dim_\C \bigl( \Hom(\Delta_\C(\lambda), \varphi_{\Z_p}^{\C}(\PEx_\lambda^{\Z_p})) \bigr) = 1.
\]
Hence $\PEx^\C_\lambda$ is a direct summand in $\varphi_{\Z_p}^\C(\PEx^{\Z_p}_\lambda)$, and we deduce the desired inequality.

If all these inequalities are equalities, then $\varphi_{\Z_p}^\C(\PEx^{\Z_p}_\lambda) = \PEx^\C_\lambda$. In view of Theorem~\ref{thm:parities-char-0}, it follows that $\PEx_\lambda^\bk = \varphi_{\Z_p}^\bk(\PEx_\lambda^{\Z_p})$ satisfies the ``stalks and costalks'' conditions that characterize $\fL^\bk(\lambda)$ (as in the proof of Theorem~\ref{thm:parities-char-0}); therefore we have $\PEx_\lambda^\bk \cong \fL^\bk(\lambda)$.
\end{proof}

\section{Support computations}
\label{sec:support}

\subsection{Support of coherent sheaves on \texorpdfstring{$\cN_\bk$}{Nk}}
\label{ss:support-cN}

In this section we assume that $\Gp_\Z$ is a (split) simply-connected semi-simple group, and let $\fR$ be as in~\S\ref{ss:relation-Gr}. 

Let $\bk$ be an algebraically closed field whose characteristic is very good for $\Gp_\Z$. Then if $\cN_\bk$ is the nilpotent cone of $\Gp_\bk$, the group $\Gp_\bk$ has a finite number of orbits on $\cN_\bk$. We now recall the Bala--Carter parametrization of these orbits, following~\cite{premet}.  For any $I \subset \Sf$, we have an associated parabolic subgroup $\PGp_{I,\Z} \subset \Gp_\Z$ containing $\BGp_\Z$, and a Levi factor $\LGp_{I,\Z}$ of $\PGp_{I,\Z}$ containing $\TGp_\Z$. If $J \subset I \subset \Sf$, then we also have a parabolic subgroup $\PGp_{I,J,\Z} \subset \LGp_{I,\Z}$ containing $\BGp_\Z \cap \LGp_{I,\Z}$ associated with $J$. We will denote by $\mathfrak{P}(\Sf)$ the quotient of the set of pairs $(I,J)$ of subsets of $\Sf$ such that $J \subset I$ and $\PGp_{I,J,\Z}$ is a distinguished parabolic subgroup of $\LGp_{I,\Z}$ by the relation
\[
(I,J) \sim (I',J') \quad \text{if there exists $w \in W$ such that $w(I)=I'$ and $w(J)=J'$.}
\]
We also denote by $\PGp_{I,\bk}$, $\LGp_{I,\bk}$, $\PGp_{I,J,\bk}$ the base change of $\PGp_{I,\Z}$, $\LGp_{I,\Z}$, $\PGp_{I,J,\Z}$ to $\bk$.

To any pair $(I,J)$ as above we associate the $\Gp_\bk$-orbit containing the Richardson orbit associated with $\PGp_{I,J,\bk}$ in $\LGp_{I,\bk}$. Then this assignment factors through $\mathfrak{P}(\Sf)$, and under our assumptions the main result of~\cite{premet} implies that it induces a bijection between $\mathfrak{P}(\Sf)$ and the set of orbits $\cN_\bk/\Gp_\bk$. Moreover, every orbit $\scO$ is separable (i.e., the map $G \to \scO$ given by $g \mapsto g\cdot x$ for some $x \in \scO$ is separable); see e.g.~\cite[Proposition~6]{mcninch}.

It follows in particular from these remarks that there exists a canonical bijection 
\[
\iota_\bk : \cN_\C/\Gp_\C \simto \cN_\bk/\Gp_\bk,
\]
defined by composing the bijection $\mathfrak{P}(\Sf) \simto \cN_\bk/\Gp_\bk$ with the inverse of the bijection $\mathfrak{P}(\Sf) \simto \cN_\C/\Gp_\C$. By~\cite[Th\'eor\`eme~III.5.2]{spaltenstein}, $\iota_\bk$ is an isomorphism of posets, for the order induced by inclusions of closures of orbits.

Recall that the \emph{support} of an object $\cF$ in $\Db \Coh^{\Gp}(\cN)_\bk$, denoted by $\supp(\cF)$, is the set of points $x$ such that the stalk $\cF_x$ is nonzero.  This is a closed $\Gp_\bk$-stable subset (by, say,~\cite[Exercise~II.5.6(c)]{hartshorne}); so it is a finite union of closures of $\Gp_\bk$-orbits.  It is clear that
\[
\supp(\cF) = \supp \left( \bigoplus_{n \in \Z} \cH^n(\cF) \right).
\]
It is sometimes useful to consider the supports of objects in $\Db \Coh^{\Gp}(\fg)_\bk$ as well.  Note that if $i : \cN_\bk \to \fg_\bk$ denotes the inclusion map, then for any $\cF \in \Db \Coh^{\Gp}(\cN)_\bk$, the objects $\cF$ and $i_*\cF$ have the same support.

We also observe that for $\cF$ in $\Db \Coh^{\Gp}(\cN)_\bk$, resp.~$\Db \Coh^{\Gp}(\fg)_\bk$, if $f: \{x \} \hookrightarrow \cN$, resp. $f: \{x\} \hookrightarrow \fg$ denotes the inclusion map of a point $x \in \supp(\cF)$, then
\begin{equation}\label{eqn:support-fiber}
Lf^*\cF \ne 0.
\end{equation}
If $\cF$ is a coherent sheaf, this follows from Nakayama's lemma; otherwise, if $k$ is the largest integer such that $\cH^k(\cF)_x \ne 0$, then~\eqref{eqn:support-fiber} follows from the fact that $\cH^k(Lf^*\cF) \cong f^*\cH^k(\cF)$.

\subsection{Statement}
\label{ss:statement-support}

\newcommand{\reg}{\mathrm{reg}}

We denote by $\pi : \tcN_\bk \to \cN_\bk$ the natural morphism (the ``Springer resolution'').
In this section we are interested in describing
\[
\supp(R\pi_*\PEx^\bk_\lambda) \quad \text{and} \quad \supp \bigl( \Ext^\bullet_{\Db \Coh(\tcN)_\bk}(\PEx^\bk_\lambda, \PEx^\bk_\lambda) \bigr)
\]
for $\lambda \in -\bX^+$ and $\lambda \in \bX$ respectively. (It can be shown that $R\pi_*\PEx^\bk_\lambda=0$ if $\lambda \in \bX \smallsetminus (-\bX^+)$---compare e.g.~Theorem~\ref{thm:parities-char-0} with~\cite[Proposition~2.6] {achar} in the case $\mathrm{char}(\bk)=0$---which justifies the restriction in the first case. For the second case, $\Ext^\bullet_{\Db \Coh(\tcN)_\bk}(\PEx^\bk_\lambda, \PEx^\bk_\lambda)$ is considered as a coherent sheaf on $\cN_\bk$ via the natural action of $\cO(\tcN_\bk)$ and the morphism $\cO(\cN_\bk) \to \cO(\tcN_\bk)$ induced by $\pi$.)

We start with the following easy case.

\begin{lem}
\label{lem:full-support}
For any $\lambda \in \bX$, we have
\[
\supp \bigl( \Ext^\bullet_{\Db \Coh(\tcN_\bk)}(\PEx^\bk_\lambda, \PEx^\bk_\lambda) \bigr) = \cN_\bk \quad \Leftrightarrow \quad \ell(w_\lambda)=0.
\]
If $\lambda \in -\bX^+$, this condition is also equivalent to $\supp(R\pi_*\PEx^\bk_\lambda) = \cN_\bk$.
\end{lem}

\begin{proof}
Let $\scO_\reg^\bk \subset \cN_\bk$ be the unique open orbit. Then $\pi$ restricts to an isomorphism $\pi^{-1}(\scO_\reg^\bk) \simto \scO_\reg^\bk$.

First, assume that $\ell(w_\lambda) \neq 0$. Then if $w_\lambda = \omega s_1 \cdots s_r$ is a reduced expression (with $\omega \in \Omega$ and $s_1, \ldots, s_r \in S$), we have $r \geq 1$, and $\PEx^\bk_\lambda$ is a direct summand in $\PEx^\bk(\omega, (s_1, \ldots, s_r))$ by Proposition~\ref{prop:parity-BS}\eqref{it:Psi-parity-2}. Now as remarked in the proof of Lemma~\ref{lem:filtration-Psi}, any object in the essential image of the functor $\Psi_{s_r}$ has trivial restriction to $\pi^{-1}(\scO_\reg^\bk)$. Hence $\PEx^\bk_\lambda$ has trivial restriction to $\pi^{-1}(\scO_\reg^\bk)$, which proves that $\supp \bigl( \Ext^\bullet_{\Db \Coh(\tcN_\bk)}(\PEx^\bk_\lambda, \PEx^\bk_\lambda) \bigr) \neq \cN_\bk$ and that if $\lambda \in -\bX^+$ we have $\supp(R\pi_*\PEx^\bk_\lambda) \neq \cN_\bk$.

On the other hand, if $\ell(w_\lambda)=0$, the object $\PEx^\bk_\lambda = \mathscr{J}_{T_{w_\lambda}}(\cO_{\tcN_\bk})$ is actually a line bundle on $\tcN$ (cf.~\cite[Lemma~5.1(1)]{achar}); in particular, its restriction to $\pi^{-1}(\scO_\reg^\bk)$ is nonzero.  We therefore have $\supp \bigl( \Ext^\bullet_{\Db \Coh(\tcN_\bk)}(\PEx^\bk_\lambda, \PEx^\bk_\lambda) \bigr) = \cN_\bk$ and $\supp(R\pi_*\PEx^\bk_\lambda) = \cN_\bk$.
\end{proof}

For the general case, our starting point is the description of these supports in the case $\bk=\C$, which is (essentially) due to Bezrukavnikov.

\begin{thm}
\leavevmode
\label{thm:support-char-0}
\begin{enumerate}
\item
\label{it:support-char0-1}
For any $\lambda \in -\bX^+$, there exists $\scO \in \cN_\C/\Gp_\C$ such that we have
\[
\supp \bigl( R\pi_*\PEx^\C_\lambda \bigr)=\overline{\scO}.
\]
\item
\label{it:support-char0-2}
For any $\lambda \in \bX$, there exists $\scO \in \cN_\C/\Gp_\C$ such that
\[
\supp \bigl( \Ext^\bullet_{\Db \Coh(\tcN)_\C}(\PEx^\C_\lambda, \PEx^\C_\lambda) \bigr)=\overline{\scO}.
\]
\end{enumerate}
\end{thm}

\begin{proof}[Proof of~\eqref{it:support-char0-1}]
By Theorem~\ref{thm:parities-char-0} we have $\PEx^\C_\lambda \cong \fL^\C_\lambda$. Now by~\cite[Proposition~8]{bezru} (see also~\cite[Proposition~2.6]{achar}) the object $R\pi_* \fL^\C_\lambda$ is a simple object in the heart of the \emph{perverse coherent} t-structure. By construction, such objects have as support the closure of a nilpotent orbit (see~\cite[\S 4.3]{achar} for details and references).
\end{proof}

Assertion~\eqref{it:support-char0-2} will be justified in~\S\ref{ss:proof-support-bezru} below. In both cases, the orbit ``$\scO$'' can be described more explicitly: see Remark~\ref{rmk:orbit-LV}.

The main result of this section is the following.

\begin{prop}
\label{prop:support-geometry}
Assume that $\mathrm{char}(\bk)>0$.
\begin{enumerate}
\item
\label{it:support-charp-1}
Let $\lambda \in -\bX^+$, and let $\scO$ be as in Theorem~{\rm \ref{thm:support-char-0}\eqref{it:support-char0-1}}. Then we have
\[
\supp \bigl( R\pi_*\PEx^\bk_\lambda \bigr) \supset \overline{\iota_\bk(\scO)}.
\]
Moreover, there exists $N_1 \in \Z$ (depending on $\lambda$) such that if $\mathrm{char}(\bk)>N_1$ we have
\[
\supp \bigl( R\pi_*\PEx^\bk_\lambda \bigr) = \overline{\iota_\bk(\scO)}.
\]
\item
\label{it:support-charp-2}
Let $\lambda \in \bX$, and let $\scO$ be as in Theorem~{\rm \ref{thm:support-char-0}\eqref{it:support-char0-2}}. Then we have 
\[
\supp \bigl( \Ext^\bullet_{\Db \Coh(\tcN)_\bk}(\PEx^\bk_\lambda, \PEx^\bk_\lambda) \bigr) \supset \overline{\iota_\bk(\scO)}.
\]
Moreover, there exists $N_2 \in \Z$ (depending on $\lambda$) such that if $\mathrm{char}(\bk)>N_2$ we have
\[
\supp \bigl( \Ext^\bullet_{\Db \Coh(\tcN)_\bk}(\PEx^\bk_\lambda, \PEx^\bk_\lambda) \bigr) = \overline{\iota_\bk(\scO)}.
\]
\end{enumerate}
\end{prop}

The first statements in both parts of Proposition~\ref{prop:support-geometry} (the ``lower bound'') will be proved in~\S\ref{ss:lower-bound}. The second statements will be proved in~\S\ref{ss:upper-bound}.

\begin{rmk}
\label{rmk:support-geometry}
We will see later (using representation theory) that in fact $N_1$ and $N_2$ can be chosen independently of $\lambda$; see Remark~\ref{rmk:main-thm}.
\end{rmk}

\subsection{\texorpdfstring{$p$}{p}-adic representatives of nilpotent orbits}

To prove the ``lower bound'' parts of Proposition~\ref{prop:support-geometry} we will need the following technical result.

\begin{lem}
\label{lem:padic-nilp-orbits}
Assume that $p=\mathrm{char}(\bk)>0$.
There exists a finite extension $\bO$ of $\Z_p$, ring morphisms $\bO \to \bk$ and $\bO \to \C$, and a subset $\{x_{[I,J]} : [I,J] \in \mathfrak{P}(\Sf)\} \subset \fg_{\bO}$, such that the images of the points $x_{[I,J]}$ in $\fg_\bk$ are representatives for the $\Gp_\bk$-orbits in $\cN_\bk$, and the images of these points in $\fg_\C$ are representatives for the $\Gp_\C$-orbits in $\cN_\C$.
\end{lem}

\begin{proof}
Let $\F$ be an algebraic closure of the prime subfield of $\bk$, and consider an embedding $\F \hookrightarrow \bk$. Then the Bala--Carter classification of nilpotent orbits recalled in~\S\ref{ss:support-cN} shows that the embedding $\fg_\F \hookrightarrow \fg_\bk$ induces a bijection $\cN_\F / \Gp_\F \overset{1:1}{\longleftrightarrow} \cN_\bk/\Gp_\bk$. Therefore, we can assume that $\bk$ is an algebraic closure of a finite field.

Let $\mathfrak{P}'(\Sf)$ be a set of representatives for the equivalences classes of pairs $(I,J)$ as in~\S\ref{ss:support-cN}. For any $(I,J) \in \mathfrak{P}'(\Sf)$, we fix a representative $y_{(I,J)} \in \fn_{I,J,\bk}$ for the Richardson orbit associated with the parabolic subgroup $\PGp_{I,J,\bk} \subset \LGp_{I,\bk}$. (Here, for any $\fS$, $\fn_{I,J,\fS}$ is the Lie algebra of the unipotent radical of $\PGp_{I,J,\fS}$.) Then there exists a finite subfield $\bk_0 \subset \bk$ such that $y_{(I,J)}$ belongs to $\fn_{I,J,\bk_0}$ for any $(I,J) \in \mathfrak{P}'(\Sf)$. Let now $\bO$ be a finite extension of $\Z_p$ with residue field $\bk_0$. For any $(I,J) \in \mathfrak{P}'(\Sf)$, we choose a preimage $x_{[I,J]}$ of $y_{(I,J)}$ under the morphism $\fn_{I,J,\bO} \twoheadrightarrow \fn_{I,J,\bk_0}$.
We also choose an embedding $\bO \hookrightarrow \C$. Then, by the Bala--Carter classification of nilpotent orbits over $\C$, to conclude it suffices to prove that for any $(I,J) \in \mathfrak{P}'(\Sf)$, the image of $x_{[I,J]}$ in $\fn_{I,J,\C}$ belongs to the Richardson orbit associated with the parabolic subgroup $\PGp_{I,J,\C} \subset \LGp_{I,\C}$.

For this, we consider the morphism $f_{(I,J)} : \PGp_{I,J,\bO} \to \fn_{I,J,\bO}$ induced by the adjoint action on $x_{[I,J]}$. By the choice of $x_{[I,J]}$, we know that $\Spec(\bk) \otimes_{\Spec(\bO)} f_{[I,J]}$ is dominant. By~\cite[Proposition~4.2.1]{mcninch2}, this implies that $\Spec(\bk_0) \otimes_{\Spec(\bO)} f_{[I,J]}$ is also dominant. By~\cite[Proposition~4.2.3]{mcninch2}, we deduce that $\Spec(\C) \otimes_{\Spec(\bO)} f_{[I,J]}$ is dominant.  The desired claim follows.
\end{proof}

\subsection{Lower bound}
\label{ss:lower-bound}

Now we are in a position to prove the first statements in both parts of Proposition~\ref{prop:support-geometry}. We fix some $\lambda \in \bX$.

Consider a finite extension $\bO$ of $\Z_p$, ring morphisms $\bO \to \bk$ and $\bO \to \C$, and a subset $\{x_{[I,J]} : [I,J] \in \mathfrak{P}(\Sf)\} \subset \fg_{\bO}$ as in Lemma~\ref{lem:padic-nilp-orbits}. Let also $\F$ be the residue field of $\bO$. For $\mathbb{E} \in \{\C, \bO, \F, \bk\}$, we will denote by $i_{[I,J]}^{\mathbb{E}} : \Spec(\mathbb{E}) \to \fg_{\mathbb{E}}$ the embedding of the image of $x_{[I,J]}$ in $\fg_{\mathbb{E}}$. 

By the same arguments as in the proof of Lemma~\ref{lem:criterion-E-L}, if $w_\lambda=\omega s_1 \cdots s_r$ is a reduced expression, then there exists a direct summand $\PEx_\lambda^\bO$ of $\PEx^\bO(\omega,(s_1, \ldots, s_r))$ such that $\varphi_{\bO}^\F(\PEx_\lambda^\bO) \cong \PEx_\lambda^\F$, and $\PEx^\C_\lambda$ is a direct summand in $\varphi_\bO^\C(\PEx^\bO_\lambda)$. We also set $\tilde{\pi} := i \circ \pi$; note that this morphism also makes sense over $\bO$.

First, assume that $\lambda \in -\bX^+$, and let $\scO$ be as in Theorem~\ref{thm:support-char-0}\eqref{it:support-char0-1}, so that we have $\supp(R\pi_* \PEx^\C_\lambda) = \overline{\scO}$. If $[I,J] \in \mathfrak{P}(\Sf)$ is such that $\scO$ corresponds to $[I,J]$, then by~\eqref{eqn:support-fiber} this implies that we have $L(i_{[I,J]}^\C)^*(R\tilde{\pi}_* \PEx^\C_\lambda) \neq 0$, and hence that
\[
\C \lotimes_\bO L(i_{[I,J]}^\bO)^*(R\tilde{\pi}_* \PEx^\bO_\lambda) \cong L(i_{[I,J]}^\C)^*(R\tilde{\pi}_* \varphi_\bO^\C(\PEx^\bO_\lambda)) \neq 0.
\]
We deduce that $L(i_{[I,J]}^\bO)^*(R\tilde{\pi}_* \PEx^\bO_\lambda) \neq 0$, and then that
\[
\bk \lotimes_\bO L(i_{[I,J]}^\bO)^*(R\tilde{\pi}_* \PEx^\bO_\lambda) \cong L(i_{[I,J]}^\bk)^*(R\tilde{\pi}_* \varphi_\bO^\bk(\PEx^\bO_\lambda)) \cong L(i_{[I,J]}^\bk)^*(R\tilde{\pi}_* \PEx^\bk_\lambda) \neq 0.
\]
Since the image of $x_{[I,J]}$ in $\fg_\bk$ belongs to $\iota_\bk(\scO)$, this implies in turn that
\[
\supp(R\pi_* \PEx^\bk_\lambda) \supset \overline{\iota_\bk(\scO_\lambda^\C)},
\]
finishing the proof of the first statement in Proposition~\ref{prop:support-geometry}\eqref{it:support-charp-1}.

The proof of the first statement in Proposition~\ref{prop:support-geometry}\eqref{it:support-charp-2} is similar. Namely, let us return to the assumption that $\lambda \in \bX$, and let $\scO$ be as in Theorem~\ref{thm:support-char-0}\eqref{it:support-char0-2}. Then if $[I,J] \in \mathfrak{P}(\Sf)$ corresponds to $\scO$, as above we have
\[
L(i_{[I,J]}^\C)^*(R\tilde{\pi}_* R\mathscr{H}\hspace{-2pt}\mathit{om}_{\scO_{\tcN_\C}}(\PEx^\C_\lambda,\PEx^\C_\lambda)) \neq 0,
\]
which implies that
\[
L(i_{[I,J]}^\bk)^*(R\tilde{\pi}_* R\mathscr{H}\hspace{-2pt}\mathit{om}_{\scO_{\tcN_\bk}}(\PEx^\bk_\lambda,\PEx^\bk_\lambda)) \neq 0,
\]
and then that $\supp \bigl( \Ext^\bullet_{\Db \Coh(\tcN)_\bk}(\PEx^\bk_\lambda, \PEx^\bk_\lambda) \bigr) \supset \overline{\iota_\bk(\scO)}$.

\subsection{Integral representatives of nilpotent orbits}
\label{ss:integral-nilpotent}

To prove the second statements in both parts of Proposition~\ref{prop:support-geometry}, we will need some representatives for nilpotent orbits which can be compared in different positive characteristics; in other words some representatives defined over $\Z$.

\begin{lem}
\label{lem:orbits-integers}
There exists $P \in \Z$ and a subset $\{x_{[I,J]} : [I,J] \in \mathfrak{P}(\Sf)\}$ of $\fg_\Z$ such that for any algebraically closed field $\bk$ of characteristic either $0$ or${}> P$, the images of the points $x_{[I,J]}$ in $\fg_\bk = \bk \otimes_\Z \fg_\Z$ are representatives of the $\Gp_\bk$-orbits in $\cN_\bk$.
\end{lem}

\begin{proof}
Fix a set $\mathfrak{P}'(\Sf)$ as in the proof of Lemma~\ref{lem:padic-nilp-orbits}. For any $(I,J) \in \mathfrak{P}'(\Sf)$, let $\fp_{I,J,\Z}$, resp.~$\fn_{I,J,\Z}$, be the Lie algebra of $\PGp_{I,J,\Z}$, resp.~of its unipotent radical. For any algebraically closed field $\bk$, we also denote by $\fp_{I,J,\bk}$ and $\fn_{I,J,\bk}$ their base change to $\bk$, which identify with the Lie algebras of $\PGp_{I,J,\bk}$ and its unipotent radical respectively.

For any $(I,J) \in \mathfrak{P}'(\Sf)$, the subset $\fn_{I,J,\Z} \subset \fn_{I,J,\C}$ is dense; therefore it has to intersect the Richardson orbit, so that there exists $x_{[I,J]} \in \fn_{I,J,\Z}$ whose image in $\fg_\C$ lies in the orbit corresponding to $[I,J]$. To conclude, we only have to prove that if $\mathrm{char}(\bk)$ is either $0$ or large, then the image $x_{[I,J]}^\bk$ of $x_{[I,J]}$ in $\fn_{I,J,\bk}$ is also in the Richardson orbit of the Lie algebra of $\LGp_{I,\bk}$ attached to $\PGp_{I,J,\bk}$. This property is equivalent to the fact that $\dim(\PGp_{I,J,\bk} \cdot x_{[I,J]}^\bk) = \dim(\fn_{I,J,\bk})$, or in other words that $\dim(Z_{\PGp_{I,J,\bk}}(x_{[I,J]}^\bk)) = \dim(\LGp_{J,\bk})$. Now $\dim(\PGp_{I,J,\bk} \cdot x_{[I,J]}^\bk) \leq \dim(\fn_{I,J,\bk})$, so that $\dim(Z_{\PGp_{I,J,\bk}}(x_{[I,J]}^\bk)) \geq \dim(\LGp_{J,\bk})$ in any case. Hence we only have to prove that $\dim(Z_{\PGp_{I,J,\bk}}(x_{[I,J]}^\bk)) \leq \dim(\LGp_{J,\bk})$ if $\mathrm{char}(\bk)$ is $0$ or large. For this, by~\cite[Proposition~1.10]{humphreys-conj} it suffices to prove that $\dim(\mathfrak{z}_{\fp_{I,J,\bk}}(x_{[I,J]}^\bk)) \leq \dim(\LGp_{J,\bk})$. However, by choosing a $\Z$-basis of $\fp_{I,J,\Z}$, the subspace $\mathfrak{z}_{\fp_{I,J,\bk}}(x_{[I,J]}^\bk) \subset \fp_{I,J,\bk}$ can be written as the space of solutions of a system of linear equations with coefficients in $\Z$, and independent of $\bk$. We have
\[
\dim(\mathfrak{z}_{\fp_{I,J,\C}}(x_{[I,J]}^\C)) = \dim(Z_{\PGp_{I,J,\C}}(x_{[I,J]}^\C)) = \dim(\LGp_{J,\C})
\] 
again by~\cite[Proposition~1.10]{humphreys-conj}  (and because, by construction, $x_{[I,J]}^\C$ is in the Richardson orbit); therefore, $\dim(\mathfrak{z}_{\fp_{I,J,\bk}}(x_{[I,J]}^\bk)) = \dim(\LGp_{J,\bk})$ if $\mathrm{char}(\bk)$ is $0$ or large, which concludes the proof.
\end{proof}

\begin{rmk}
There exist explicit descriptions of nilpotent orbit representatives over arbitrary fields; see e.g.~\cite{ls}. Using this information one can obtain an explicit value for the integer $P$ in Lemma~\ref{lem:orbits-integers}. Since this information does not lead to an improvement in our results, for simplicity we do not go into these details.
\end{rmk}

\subsection{Asymptotic support}

If $\bk$ is as in~\S\ref{ss:support-cN}, then as explained in~\cite[\S 7.14]{jantzen-nilp} (see also~\cite[\S 3.14]{slodowy}), under our assumptions the nilpotent cone $\cN_\bk$ is the scheme-theoretic fiber of the adjoint quotient $\fg_\bk \to \fg_\bk/\Gp_\bk$ over the image of $0$. We define $\cN_\fR$ as the scheme-theoretic fiber of the adjoint quotient $\fg_\fR \to \fg_\fR/\Gp_\fR$ over the image of $0$. Then it follows from~\cite[Corollary~4.2.2]{riche-kostant} that for any $\bk$ as above we have $\cN_\bk = \Spec(\bk) \times_{\Spec(\fR)} \cN_\fR$. Hence it makes sense to define, for any commutative ring $\fS$ admitting a morphism $\fR \to \fS$, the nilpotent cone over $\fS$ as $\cN_\fS := \Spec(\fS) \times_{\Spec(\fR)} \cN_\fR$. As in~\S\ref{ss:support-cN} (in the case of an algebraically closed field), we will denote by $i : \cN_\fS \to \fg_\fS$ the embedding. Note that the morphism $\tilde{\pi} : \tcN_\fS \to \fg_\fS$ factors through a morphism $\pi : \tcN_\fS \to \cN_\fS$ (which specializes to the morphism considered in~\S\ref{ss:statement-support} if $\fS$ is an algebraically closed field).

We claim that $\cN_\fR$ is flat over $\fR$. In fact this follows from the fact that the adjoint quotient $\fg_\fR \to \fg_\fR/\Gp_\fR$ is a flat morphism, which itself follows from the case of fields $\bk$ as above (which is known, see~\cite[\S 3.14]{slodowy}) using~\cite[Corollary~4.1.2 and Proposition~4.2.1]{riche-kostant}. Once this is known, for any commutative Noetherian rings $\fS$, $\mathfrak{T}$ of finite global dimension and any ring morphism $\fS \to \mathfrak{T}$ we can consider the ``extension of scalars'' functor
\[
\psi_\fS^{\mathfrak{T}} : \Db \Coh^{\Gp}(\cN)_\fS \to \Db \Coh^{\Gp}(\cN)_{\mathfrak{T}}
\]
induced by the functor $\mathfrak{T} \lotimes_{\fS} (-)$.

\begin{prop}
\label{prop:support-change-scalars}
Let $\fS$ be a finite localization of $\Z$ containing $\fR$, and let $\cF$ be an object in $\Db \Coh^{\Gp}(\cN)_\fS$. Assume that $\supp(\psi_{\fS}^\C(\cF))=\overline{\scO}$ for some orbit $\scO \in \cN_\C/\Gp_\C$. Then there exists $Q \in \Z$ such that $\fS \subset \Z[\frac{1}{Q!}]$ and for any algebraically closed field $\bk$ of characteristic either $0$ or $>Q$, we have $\supp(\psi_\fS^\bk(\cF)) \subset \overline{\iota_\bk(\scO)}$.
\end{prop}

\begin{proof}
We fix a set $\{x_\alpha : \alpha \in \mathfrak{P}(\Sf)\}$ as in Lemma~\ref{lem:orbits-integers}, and for any $\alpha \in \mathfrak{P}(\Sf)$ and any commutative ring $\mathfrak{T}$ we denote by $i_\alpha^{\mathfrak{T}} : \Spec(\mathfrak{T}) \to \fg_{\mathfrak{T}}$ the embedding of the image of $x_\alpha$ in $\fg_{\mathfrak{T}}$. 

If $\sigma \in \mathfrak{P}(\Sf)$ and the orbit in $\cN_\C$ corresponding to $\sigma$ is not included in $\overline{\scO}$, since $i_\alpha^\C$ factors through $\fg_\C \smallsetminus \supp(\psi_\fS^\C(\cF))$ we have
\[
\C \lotimes_{\fS} L(i_\sigma^\fS)^*(i_* \cF)=L(i_\sigma^\C)^*(i_*\psi_\fS^\C(\cF))=0.
\]
Hence all the cohomology objects of the bounded complex of finitely generated $\fS$-module $L(i^\fS_\sigma)^*(i_* \cF)$ are torsion. We deduce that there exists $Q \in \Z$ with $\fS \subset \Z[\frac{1}{Q!}]$ such that
\begin{equation}
\label{eqn:vanishing-stalks}
\Z \left[ \frac{1}{Q!} \right] \lotimes_{\fS} L(i_\sigma^\fS)^*(i_* \cF)=0
\end{equation}
for any $\sigma \in \mathfrak{P}(\Sf)$ such that the corresponding orbit in $\cN_\C$ is not contained in $\overline{\scO}$. Of course, we can assume that $Q \geq P$, where $P$ is as in Lemma~\ref{lem:orbits-integers}.

We claim that with this choice of $Q$, for any algebraically closed field $\bk$ such that $\mathrm{char}(\bk) > Q$ or $\mathrm{char}(\bk)=0$, we have $\supp(\psi_\fS^\bk(\cF)) \subset \overline{\iota_\bk(\scO)}$. Assume for a contradiction that this is not the case for some $\bk$, and let $\sigma \in \mathfrak{P}(\Sf)$ be such that the corresponding orbit in $\cN_\C$ is not contained in $\overline{\scO}$, and such that the corresponding orbit in $\cN_\bk$ is maximal (i.e.~open) in $\supp(\varphi_\fS^\bk(\cF))$. 
By~\eqref{eqn:support-fiber}, it follows that
\[
\bk \lotimes_{\fS} L(i_\sigma^\fS)^*(i_* \cF) \cong
L(i_\sigma^\bk)^* (i_* \psi_\fS^\bk(\cF)) \neq 0,
\]
which contradicts~\eqref{eqn:vanishing-stalks} since the morphism $\fS \to \bk$ factors through $\Z[\frac{1}{Q!}]$.
\end{proof}

\subsection{Upper bound}
\label{ss:upper-bound}

We can now finish the proof of Proposition~\ref{prop:support-geometry}.

Let $\lambda \in -\bX^+$, 
let $M$ be as in Proposition~\ref{prop:indec-parity-integers}, and consider the corresponding object $\PEx_\lambda^{\Z[\frac{1}{M!}]} \in \Db \Coh^{\Gp \times \Gm}(\tcN)_{\Z[\frac{1}{M!}]}$. By Theorem~\ref{thm:support-char-0}\eqref{it:support-char0-1} we have
\[
\supp \bigl( \psi_{\Z[\frac{1}{M!}]}^\C(R\pi_*\PEx^{\Z[\frac{1}{M!}]}_\lambda ) \bigr)=\overline{\scO_\lambda^\C}.
\]
Of course we can assume that $M \geq N$ (where $N$ is as in~\S\ref{ss:relation-Gr}); then
Proposition~\ref{prop:support-change-scalars} provides a bound $N_1$ such that
\[
\supp \bigl( R\pi_*\PEx^{\bk}_\lambda \bigr) =
\supp \bigl( \psi_{\Z[\frac{1}{M!}]}^\bk(R\pi_*\PEx^{\Z[\frac{1}{M!}]}_\lambda ) \bigr) \subset \overline{\iota_\bk(\scO_\lambda^\C)}
\]
if $\mathrm{char}(\bk) > N_1$. Since the reverse inclusion has already been proved (for any $\bk$), this finishes the proof of Proposition~\ref{prop:support-geometry}\eqref{it:support-charp-1}.

The proof of the second statement in Proposition~\ref{prop:support-geometry}\eqref{it:support-charp-2} is similar, using the object $R\mathscr{H}\hspace{-2pt}\mathit{om}_{\scO_{\tcN_{\Z[\frac{1}{M!}]}}}(\PEx^{\Z[\frac{1}{M!}]}_\lambda,\PEx^{\Z[\frac{1}{M!}]}_\lambda)$ instead of $\PEx_\lambda^{\Z[\frac{1}{M!}]}$.

\section{Antispherical cells and \texorpdfstring{$p$}{p}-cells}
\label{sec:antispherical-cells}

We let $\Gp_\Z$ (and all the related data) be as in~\S\ref{ss:def-exotic}.

\subsection{The affine Hecke algebra and its canonical bases}

Let $\Haff$ be the affine Hecke algebra of $(W,S)$, i.e.~the $\Z[v,v^{-1}]$-algebra with basis $\{H_w : w \in W\}$ and multiplication determined by the rules:
\begin{itemize}
\item
$H_s^2 = 1 + (v^{-1}-v)H_s$ for $s \in S$;
\item
$H_x H_y = H_{xy}$ for $x,y \in W$ such that $\ell(xy)=\ell(x)+\ell(y)$.
\end{itemize}
(Note that we follow the conventions of~\cite{soergel}.)

This algebra admits a family of ``canonical bases'' which can be constructed using geometry. Namely, let $\Gp^\wedge$ be the simply-connected cover of the derived subgroup of the group $\Gp^\vee$ of~\S\ref{ss:relation-Gr} (but without the further technical conditions), and let $\TGp^\wedge$ be the inverse image of the maximal torus $\TGp^\vee \subset \Gp^\vee$ (so that $\TGp^\wedge$ is a maximal torus of $\Gp^\wedge$). We let $\BGp^\wedge \subset \Gp^\wedge$ be the Borel subgroup whose roots are the negative coroots of $(\Gp,\TGp)$, let $\Iwa \subset \Gp^\wedge(\scO)$ be the Iwahori subgroup determined by $\BGp^\wedge$, and consider the affine flag variety
\[
\Fl := \Gp^\wedge(\mathscr{K}) / \Iwa.
\]
For any $w \in W$ we have a corresponding $\Iwa$-orbit $\Fl_w \subset \Fl$, and then a Bruhat decomposition
\[
\Fl = \bigsqcup_{w \in W} \Fl_w,
\]
with $\Fl_w$ isomorphic to $\C^{\ell(w)}$.

Let $\bk$ be a field, and consider the derived category $\Db_{(\Iwa)}(\Fl,\bk)$ of $\Iwa$-constructible complexes of $\bk$-sheaves on $\Fl$. Following~\cite{jmw}, we say that a complex $\cF$ in $\Db_{(\Iwa)}(\Fl,\bk)$ is \emph{even} if
\[
\cH^n(\cF) = \cH^n(\mathbb{D}_\Fl(\cF))=0 \quad \text{unless $n$ is even.}
\]
A complex is called \emph{parity} if it is isomorphic to $\cF_0 \oplus \cF_1$ with $\cF_0$ and $\cF_1[1]$ even complexes. The results of~\cite{jmw} show that the indecomposable parity complexes are parametrized, up to cohomological shift, by $W$; more precisely for any $w \in W$ there exists a unique indecomposable parity complex $\cE_w^\bk$ which is supported on $\overline{\Fl_w}$ and whose restriction to $\Fl_w$ is $\underline{\bk}_{\Fl_w}[\ell(w)]$. Then any indecomposable parity complex is isomorphic to some $\cE_x^\bk [i]$ for a unique pair $(x,i) \in W \times \Z$.

For any $\cF$ in $\Db_{(\Iwa)}(\Fl,\bk)$, we can define the \emph{character} of $\cF$ as
\[
\mathrm{ch}(\cF) := \sum_{\substack{w \in W \\ n \in \Z}} \dim_\bk \mathsf{H}^{-n-\ell(w)}(\Fl_w, \cF_{|\Fl_w}) \cdot v^n H_w \quad \in \cH.
\]
If $p=\mathrm{char}(\bk)$, we define the \emph{$p$-canonical basis} $(\puH_w : w \in W)$ by
\[
\puH_w := \mathrm{ch}(\cE_w^\bk).
\]
An obvious generalization of~\cite[Corollary~3.9]{williamson} shows that $\puH_w$ depends on $\bk$ only through $p$, which justifies the notation.

\begin{rmk}
\leavevmode
\begin{enumerate}
\item
It follows from work of Kazhdan--Lusztig~\cite{kl} and Sprin\-ger~\cite{springer} (see also~\cite[Proposition~3.6]{williamson}) that if $p=0$ then $\cE_w^\bk$ is the intersection cohomology complex of $\overline{\Fl_w}$ (for the constant local system), and that the $0$-canonical basis coincides with the ``usual'' Kazhdan--Lusztig basis (with the conventions of~\cite[Theorem~2.1]{soergel}). In this case we will sometimes write $\uH_w$ instead of ${}^0 \hspace{-1pt} \uH_w$.
\item
Using the results of~\cite[Part~III]{rw} one can check that the $p$-canonical basis as defined above coincides with the basis considered in~\cite{jw}.
\end{enumerate}
\end{rmk}

The $p$-canonical basis has the following positivity property, which will be crucial to us:
\begin{equation}
\label{eqn:positivity-puH}
 \puH_x \cdot \puH_y \in \bigoplus_{w \in W} \Z_{\geq 0}[v,v^{-1}] \cdot \puH_w.
\end{equation}
(One way to prove this is to extend the previous construction to the equivariant setting, which yields the same $p$-canonical basis by~\cite[Lemma~2.4]{mr2}; then to observe that the same arguments as in~\cite{springer}---see also~\cite[\S 4.1]{jmw}---show that the convolution product of parity complexes is still parity, and that its character is the product of the characters of the factors; then~\eqref{eqn:positivity-puH} follows from the definition of the character map.) Note also that if $s \in S$, then for any $p$, using the fact that $\overline{\Fl_s} \cong \mathbb{P}^1$, we have
\begin{equation}\label{eqn:puHs}
\puH_s = \uH_s = H_s + v.
\end{equation}

If $\underline{w}=(s_1, \ldots, s_n)$ is an expression, we also set
\[
\uH_{\underline{w}} = \underline{H}_{s_1} \cdots \underline{H}_{s_n}.
\]
If $\uw$ is a reduced expression for some $w \in W$, then, using the Bott--Samelson resolution of $\overline{\Fl_w}$ determined by $\uw$, one checks that
\begin{equation}
\label{eqn:puH-uw-w}
\uH_{\underline{w}} \in \puH_w \oplus \bigoplus_{y<w} \Z_{\geq 0}[v,v^{-1}] \cdot \puH_y.
\end{equation}

Finally, we have an analogous notion of parity complexes on partial affine flag varieties. Using the fact that the inverse image under the projection from $\Fl$ to a partial affine flag variety of an indecomposable parity complex remains indecomposable (see~\cite[Proposition~3.5]{williamson}) and the left--right symmetry of our constructions, one checks that if $\uw$ is an expression starting with some $s \in S$, then
\begin{equation}
\label{eqn:decomposition-1st-reflection}
\uH_{\underline{w}} \in \bigoplus_{\substack{y \in W \\ sy<y}} \Z_{\ge 0}[v,v^{-1}] \cdot \puH_y.
\end{equation}

\subsection{Right \texorpdfstring{$p$}{p}-cells}
\label{ss:right-p-cells}

We fix $p$ to be either $0$ or a prime number.

If $H \in \Haff$, we will say that $\puH_w$ \emph{appears with nonzero coefficient in $H$} if the coefficient of $\puH_w$ in the decomposition of $H$ in the basis $(\puH_w : w \in W)$ is nonzero. 

The following definition (due to Williamson, and studied in general in~\cite{jensen}) is an obvious generalization of a notion studied by Kazhdan--Lusztig~\cite{kl1}. (Their setting corresponds to our special case $p=0$. In this case we will sometimes omit the superscript ``$0$.'')

\begin{defn}
\label{defn:order-cells}
 We define the preorder $\pRleq$ on $W$ by declaring that
 \[
  w \pRleq y \quad \text{iff $\puH_w$ appears with nonzero coefficient in $\puH_y \cdot \puH_x$ for some $x \in W$.}
 \]
 We denote by $\pRsim$ the equivalence relation on $W$ defined by
 \[
  w \pRsim y \quad \text{iff $w \pRleq y$ and $y \pRleq w$.}
 \]
 The equivalence classes for this relation will be called the \emph{right $p$-cells}.
\end{defn}

The fact that $\pRleq$ is a preorder follows from the observation that we have $w \pRleq y$ iff there exists $H \in \Haff$ such that $\puH_w$ appears with nonzero coefficient in $\puH_y \cdot H$ (which itself follows simply from the fact that $(\puH_x : x \in W)$ is a basis of $\Haff$). In case $p=0$, the right $p$-cells will be called the \emph{right cells}. If $w \in W$, we will denote by $\bc(w) \subset W$ the right cell containing $w$.

Below we will need some elementary properties of right $p$-cells, stated in the following two lemmas.

\begin{lem}
\label{lem:order-generated}
\leavevmode
\begin{enumerate}
\item
\label{it:order-generated-1}
Let $w,y \in W$, and $s_1, \ldots, s_j \in S$. Assume that $\puH_w$ appears with nonzero coefficient in $\puH_y \cdot \uH_{s_1} \cdots \uH_{s_j}$. Then there exist $v_1, \ldots, v_{j+1} \in W$ with $v_1=y$, $v_{j+1}=w$, and for each $i \in \{1, \ldots, j\}$, $\puH_{v_{i+1}}$ appears with nonzero coefficient in $\puH_{v_{i}} \cdot \uH_{s_i}$.
\item
\label{it:order-generated-2}
Let $w,y \in W$. Then $w \pRleq y$ iff there exist $v_1, \ldots, v_k \in W$ with $v_1=y$, $v_k=w$ and $s_1, \ldots, s_{k-1} \in S$ such that for any $i \in \{1, \ldots, k-1\}$, $\puH_{v_{i+1}}$ appears with nonzero coefficient in $\puH_{v_{i}} \cdot \uH_{s_i}$.
\end{enumerate}
\end{lem}

\begin{proof}
\eqref{it:order-generated-1}
We proceed by induction on $j$. If $j=0$ there is nothing to prove. Now assume that $j>0$, and that $\puH_w$ appears with nonzero coefficient in $\puH_y \cdot \uH_{s_1} \cdots \uH_{s_j}$. Then~\eqref{eqn:positivity-puH} implies that there exists $z \in W$ such that $\puH_z$ appears with nonzero coefficient in $\puH_y \cdot \uH_{s_1} \cdots \uH_{s_{j-1}}$ and $\puH_w$ appears with nonzero coefficient in $\puH_z \cdot \uH_{s_j}$. Then we apply induction to the pair $(z,y)$, and deduce the claim for the pair $(w,y)$.

\eqref{it:order-generated-2}
If $w,y$ satisfy the second condition, then we have
\[
w=v_k \pRleq v_{k-1} \pRleq \cdots \pRleq v_1=y,
\]
and hence $w \pRleq y$ since $\pRleq$ is a preorder.

On the other hand, assume that $w \pRleq y$. Let $z \in W$ be such that $\puH_w$ appears with nonzero coefficient in $\puH_y \cdot \puH_z$, and let $\underline{z}=(s_1, \cdots, s_k)$ be a reduced expression for $z$. Then $\puH_w$ appears with nonzero coefficient in $\puH_y \cdot \uH_{\underline{z}}$ by~\eqref{eqn:positivity-puH} and~\eqref{eqn:puH-uw-w}. Hence using~\eqref{it:order-generated-1} we deduce the existence of $v_1, \ldots, v_k \in W$ and $s_1, \ldots, s_{k-1} \in S$ satisfying the desired property.
\end{proof}

\begin{lem}
\label{lem:pRleq-Bruhat}
Let $w,y \in W$ and $s \in S$. Assume that $sy<y$ and $w \pRleq y$. Then $sw<w$.
\end{lem}

\begin{proof}
Choose a reduced expression $\underline{y}$ for $y$ which starts with $s$. Then $\puH_y$ appears with nonzero coefficient in $\uH_{\underline{y}}$ by~\eqref{eqn:puH-uw-w}. Next, let $z \in W$ be such that $\puH_w$ appears with nonzero coefficient in $\puH_y \cdot \puH_z$, and let $\underline{z}$ be a reduced expression for $z$. The same reasoning shows that $\puH_z$ appears with nonzero coefficient in $\uH_{\underline{z}}$. Using~\eqref{eqn:positivity-puH} we obtain that $\puH_w$ appears with nonzero coefficients in $\uH_{\underline{yz}}$. By~\eqref{eqn:decomposition-1st-reflection}, this implies that $sw<w$.
\end{proof}

\begin{ex}
\label{ex:cell-1}
 It is clear from the definition that $w \pRleq 1$ for any $w \in W$. On the other hand, if $w \neq 1$ then Lemma~\ref{lem:pRleq-Bruhat} implies that we have $1 \not\pRleq w$. 
 In particular, $\{1\}$ is a right $p$-cell.
\end{ex}

\subsection{Antispherical right \texorpdfstring{$p$}{p}-cells}
\label{ss:antispherical-cells}

Again, we fix $p$ to be either $0$ or a prime number.
Recall that $\fW \subset W$ is the subset of elements $w$ which are minimal in $\Wf \cdot w$.

\begin{lem}
Let $\mathbf{c}$ be a right $p$-cell. If $\mathbf{c} \cap \fW \neq \varnothing$ then $\mathbf{c} \subset \fW$.
\end{lem}

\begin{proof}
Assume for a contradiction that $\mathbf{c} \cap \fW \neq \varnothing$ but $\mathbf{c} \not\subset \fW$. Let $x \in \mathbf{c} \cap \fW$ and $y \in \mathbf{c} \cap (W \smallsetminus \fW)$. Then $x \pRleq y$. And since $y \notin \fW$, there exists $s \in \Sf$ such that $sy<y$. By Lemma~\ref{lem:pRleq-Bruhat} this implies that $sx<x$, contradicting the assumption that $x \in \fW$.
\end{proof}

The right $p$-cells which intersect $\fW$ will be called \emph{antispherical}. Following the terminology introduced in~\cite{lx}, we might have called these $p$-cells \emph{canonical}. We prefer the term antispherical in view of the following interpretation.

We denote by $\Hf$ the Hecke algebra of $(\Wf,\Sf)$.
Then we can consider the antispherical right $\Haff$-module
\[
 \Masph := \mathrm{sgn} \otimes_{\Hf} \Haff,
\]
where $\mathrm{sgn}$ is the set $\Z[v,v^{-1}]$, made into a right $\Hf$-module by having $H_s$ act by multiplication by $-v$ for all $s \in \Sf$. The standard, resp.~$p$-canonical, basis of $\Masph$ is defined by
\[
 N_w := 1 \otimes H_w, \quad \text{resp.} \quad \puN_w := 1 \otimes \puH_w,
\]
where $w \in \fW$, and $p$ is either $0$ or a prime number.

\begin{rmk}
In the case $p=0$, it follows from~\cite[Proposition~3.4 and its proof]{soergel} that the basis $(\puN_w : w \in W)$  coincides with the basis characterized in~\cite[Theorem~3.1(2)]{soergel}.
\end{rmk}

\begin{lem}
\label{lem:p-can-basis-Masph}
If $w \notin \fW$, then $1 \otimes \puH_w = 0$ in $\Masph$.
\end{lem}

\begin{proof}
We prove the claim by induction on $\ell(w)$. If $\ell(w)=0$ there is nothing to prove. Otherwise, let $\underline{w}$ be a reduced expression for $w$ starting with an element of $\Sf$. Then it follows from~\eqref{eqn:puHs} that $1 \otimes \uH_{\underline{w}}=0$. Using induction together with~\eqref{eqn:puH-uw-w} and~\eqref{eqn:decomposition-1st-reflection} we deduce that $1 \otimes \puH_w = 0$ as expected.
\end{proof}

As in the case of $\Haff$, for $N \in \Masph$ we will say that $\puN_w$ \emph{appears with nonzero coefficient in $N$} if the coefficient of $\puN_w$ in the decomposition of $N$ in the basis $(\puN_w : w \in W)$ is nonzero. Then from Lemma~\ref{lem:p-can-basis-Masph} we deduce that if $w,y \in \fW$, then
 $w \pRleq y$ iff $\puN_w$ appears with nonzero coefficient in $\puN_y \cdot H$ for some $H \in \Haff$. In other words, the restriction of the preorder $\pRleq$ to $\fW$ can be described in a way completely similar to the preorder on $\Haff$, simply replacing the regular right module $\Haff$ by the antispherical right module $\Masph$. Similar remarks apply to the equivalence relation $\pRsim$; the antispherical right $p$-cells are the equivalence classes for this relation on $\fW$.

\subsection{Description in terms of \texorpdfstring{$W$}{W}-modules}
\label{ss:cells-W-mod}

Let
\[
 \Masph^\circ := \Z \otimes_{\Z[v,v^{-1}]} \Masph \cong \Z_\varepsilon \otimes_{\Z[\Wf]} \Z[W],
\]
where $v$ is specialized to $1$. (Here $\Z_\varepsilon$ is $\Z$, equipped with the $\Wf$-action in which each $s \in \Sf$ acts by multiplication by $-1$.)
This is a right module over $1 \otimes_{\Z[v,v^{-1}]} \Haff \cong \Z[W]$. We also have a $p$-canonical ($\Z$-)basis in $\Z[W]$ defined by
\[
 \puH_w^\circ := 1 \otimes \puH_w
\]
and a standard and a $p$-canonical basis in $\Masph^\circ$ defined by
\[
 N_w^\circ := 1 \otimes N_w, \qquad \puN_w^\circ := 1 \otimes \puN_w
\]
for $w \in \fW$. (In case $p=0$, we will drop the superscript $p$.)

As in the case of $\Masph$, for $N \in \Masph^\circ$, we will say that $\puN^\circ_w$ \emph{appears with nonzero coefficient in $N$} if the coefficient of $\puN^\circ_w$ in the decomposition of $N$ in the $\Z$-basis $(\puN^\circ_w : w \in W)$ is nonzero.

\begin{lem}
 \label{lem:pRleq-M0}
 If $w,y \in \fW$, then
 $w \pRleq y$ 
 iff $\puN_w^\circ$ appears with nonzero coefficient in $\puN_y^\circ \cdot h$ for some $h \in \Z[W]$.
\end{lem}

\begin{proof}
If $\puN_w^\circ$ appears with nonzero coefficient in $\puN_y^\circ \cdot h$ for some $h \in \Z[W]$, then $\puN_w$ appears with nonzero coefficient in $\puN_y \cdot H$ for some $H \in \Haff$, which implies that $w \pRleq y$. 

On the other hand, assume that $w \pRleq y$, and let $z \in W$ be such that $\puH_w$ appears with nonzero coefficient in $\puH_y \cdot \puH_z$. Then $\puN_w$ appears with nonzero coefficient in $\puN_y \cdot \puH_z$ by the considerations in~\S\ref{ss:antispherical-cells}. Moreover,~\eqref{eqn:positivity-puH} and Lemma~\ref{lem:p-can-basis-Masph} imply that the coefficients of $\puN_y \cdot \puH_z$ in the basis $(\puN_x : x \in \fW)$ all belong to $\Z_{\geq 0}[v,v^{-1}]$. Hence 
$\puN_w^\circ$ appears with nonzero coefficient in $\puN_y^\circ \cdot \puH_z^\circ$, proving the desired claim.
\end{proof}

\section{Finite generation of antispherical right cells}
\label{sec:finite-generation}

In this section, we prove that each antispherical right cell is generated by a finite number of elements under the operation of left multiplication by elements of the form $t_\lambda$ with $\lambda \in \bY^+$; see Proposition~\ref{prop:finite-generation}. (Recall that \emph{right cell} is a synonym for \emph{$0$-right cell}; thus, we are in the realm of classical Kazhdan--Lusztig theory.)  This statement is a slight variant of~\cite[Corollary~11]{andersen}. This result is stated without proof by Andersen but, as indicated in~\cite{andersen},
a proof can be obtained by copying some ideas in~\cite{xi}.

The proof will use the following well-known properties (see e.g.~\cite[\S 2]{xi}):
\begin{gather}
\label{eqn:xi1}
\ell(t_{\lambda+\mu}) = \ell(t_{\lambda})+\ell(t_{\mu}) \quad \text{if $\lambda, \mu \in \bY^+$;} \\
\label{eqn:xi2}
\ell(wt_\lambda w^{-1}) = \ell(t_\lambda) \quad \text{if $\lambda \in \bY^+$ and $w \in W$;}\\
\label{eqn:xi3}
wv \Rleq w \quad \text{if $w,v \in W$ and $\ell(wv) = \ell(w) + \ell(v)$.}
\end{gather}

\subsection{Preliminaries on minimal length representatives}

\begin{lem}
\label{lem:char-fW}
Let $\lambda \in \bY$ and $v \in \Wf$, and set $w = t_\lambda v$. The following conditions are equivalent:
\begin{enumerate}
\item
\label{it:char-fW-1}
$w \in \fW$;
\item
\label{it:char-fW-2}
$\lambda \in \bY^+$ and $\ell(t_\lambda v) = \ell(t_\lambda) - \ell(v)$;
\item
\label{it:char-fW-3}
$\lambda \in \bY^+$ and for all $\alpha \in \Phi^+$ such that $v^{-1}(\alpha) \in -\Phi^+$, we have $\langle \lambda, \alpha^\vee \rangle \geq 1$.
\end{enumerate}
\end{lem}

\begin{proof}
From~\cite[Lemma~2.4]{mr} we see that~\eqref{it:char-fW-1} implies~\eqref{it:char-fW-2}. On the other hand, if~\eqref{it:char-fW-2} holds and $s \in \Sf$ we have
\[
\ell(sw) = \ell(st_\lambda v) \geq \ell(st_\lambda) - \ell(v) = \ell(t_\lambda) + 1 -\ell(v) = \ell(w)+1,
\]
proving that $sw>w$, and hence that~\eqref{it:char-fW-1} holds.

By the Iwahori--Matsumoto formula for the length in $W$ (see, for in\-stance,~\cite[(2.2)]{mr}), we have
\begin{align*}
\ell(w) &= \sum_{\alpha \in \Phi^+ \cap v^{-1}(\Phi^+)} |\langle v^{-1}(\lambda), \alpha^\vee \rangle| + \sum_{\alpha \in \Phi^+ \cap v^{-1}(-\Phi^+)} |1 + \langle v^{-1}(\lambda), \alpha^\vee \rangle| \\
&= \sum_{\alpha \in v(\Phi^+) \cap \Phi^+} |\langle \lambda, \alpha^\vee \rangle| + \sum_{\alpha \in v(\Phi^+) \cap (-\Phi^+)} |1 + \langle \lambda, \alpha^\vee \rangle|.
\end{align*}
If $\lambda \in \bY^+$, we deduce that
\begin{align*}
\ell(w) &= \sum_{\alpha \in v(\Phi^+) \cap \Phi^+} \langle \lambda, \alpha^\vee \rangle + \sum_{\substack{\beta \in v(-\Phi^+) \cap \Phi^+\\ \langle \lambda, \beta^\vee \rangle \geq 1}} (\langle \lambda, \beta^\vee \rangle-1) 
+ \sum_{\substack{\beta \in v(-\Phi^+) \cap \Phi^+\\ \langle \lambda, \beta^\vee \rangle =0}} 1
\\
&= \sum_{\alpha \in \Phi^+} \langle \lambda, \alpha^\vee \rangle
- \#(v(-\Phi^+) \cap \Phi^+) + 2\#\{\beta \in v(-\Phi^+) \cap \Phi^+\mid \langle \lambda, \beta^\vee \rangle=0\}.
\end{align*}
Since $\ell(t_\lambda) = \sum_{\alpha \in \Phi^+} \langle \lambda, \alpha^\vee \rangle$ and $\ell(v) = \#(v(-\Phi^+) \cap \Phi^+)$, we deduce the equivalence of~\eqref{it:char-fW-2} and~\eqref{it:char-fW-3}.
\end{proof}

\begin{lem}
\label{lem:length-fW-y+}
Let $w \in \fW$ and $\lambda \in \bY^+$. Then 
\begin{enumerate}
\item
\label{it:length-fW-y+1}
$\ell(t_\lambda w) = \ell(t_\lambda) + \ell(w)$;
\item
\label{it:length-fW-y+2}
$t_\lambda w \in \fW$;
\item
\label{it:length-fW-y+3}
$t_\lambda w \Rleq w$.
\end{enumerate}
\end{lem}

\begin{proof}
Write $w=t_\mu v$ with $\mu \in \bY^+$ and $v \in \Wf$ (see Lemma~\ref{lem:char-fW}\eqref{it:char-fW-1}). 
Then $t_\lambda w = t_{\lambda + \mu} v$, so that
\[
\ell(t_\lambda w) \geq \ell(t_{\lambda + \mu}) - \ell(v) \overset{\eqref{eqn:xi1}}{=} \ell(t_\lambda) + \ell(t_{\mu})- \ell(v) = \ell(t_\lambda) + \ell(w).
\]
Since $\ell(t_\lambda w) \leq \ell(t_\lambda) + \ell(w)$, we deduce~\eqref{it:length-fW-y+1}. Then~\eqref{it:length-fW-y+2} follows from Lemma~\ref{lem:char-fW} and~\eqref{eqn:xi1}.
Finally, we remark that $t_\lambda w= w \cdot (w^{-1} t_\lambda w)$ and that
\[
\ell(t_\lambda w) = \ell(w) + \ell(w^{-1} t_\lambda w)
\]
by~\eqref{it:length-fW-y+1} and~\eqref{eqn:xi2}, so that~\eqref{it:length-fW-y+3} follows from~\eqref{eqn:xi3}.
\end{proof}

By Lemma~\ref{lem:length-fW-y+}\eqref{it:length-fW-y+2} we see in particular that the semigroup $\bY^+$ acts on $\fW$ via left multiplication in $W$.

\subsection{Stabilization}

\begin{cor}
\label{cor:stabilization}
If $w \in \fW$ and $\lambda \in \bY^+$, there exists a unique $n(w,\lambda) \in \Z_{\geq 0}$ such that
\[
n,m \geq n(w,\lambda) \quad \Rightarrow \quad t_{n\lambda} w \Rsim t_{m\lambda} w
\]
and
\[
n \geq n(w,\lambda) > m \quad \Rightarrow \quad t_{n\lambda} w \not\Rsim t_{m\lambda} w.
\]
\end{cor}

\begin{proof}
If $n \geq m$, then by Lemma~\ref{lem:length-fW-y+} we have
\[
t_{n\lambda} w \Rleq t_{m\lambda} w.
\]
Since $W$ only has a finite number of right cells by~\cite[Theorem~2.2]{lusztig-cells2}, the sequence $(t_{n\lambda} w : n \in \Z_{\geq 0})$ must be stationary (with respect to the preorder $\Rleq$), and the existence and uniqueness of $n(w,\lambda)$ follow.
\end{proof}

For any $\alpha \in \Phis$, we fix once and for all a weight $\varpi_\alpha \in \bY^+$ such that
$\langle \varpi_\alpha, \beta^\vee \rangle =0$
if $\beta \in \Phis \smallsetminus \{\alpha\}$ and $\langle \varpi_\alpha, \alpha^\vee \rangle >0$, and we set $k_\alpha := \langle \varpi_\alpha, \alpha^\vee \rangle$. We also set $k_\Phi:=\max\{k_\alpha : \alpha \in \Phis\}$. For any $\Psi \subset \Phis$, we set
\[
x_\Psi := \sum_{\alpha \in \Psi} \varpi_{\alpha}.
\]
For $\lambda \in \bY^+$, we set
\[
\Psi(\lambda):=\{\alpha \in \Phis \mid \langle \lambda, \alpha^\vee \rangle >0\}.
\]

The following lemma (which will not be used below) justifies why later we will only consider the numbers $n(w,\lambda)$ when $\lambda$ is of the form $x_\Psi$.

\begin{lem}
\label{lem:n-xPsi}
Let $w \in \fW$ and $\lambda \in \bY^+$. Then we have $n(w,\lambda) \leq k_\Phi \cdot n(w,x_{\Psi(\lambda)})$.
\end{lem}

\begin{proof}
Let $m>n\geq k_\Phi \cdot n(w,x_{\Psi(\lambda)})$, and let $a:= \lfloor \frac{n}{k_\Phi} \rfloor$, so that $a \geq n(w,x_{\Psi(\lambda)})$. For any $\beta \in \Psi(\lambda)$ we have
\[
\langle n \lambda - a x_{\Psi(\lambda)}, \beta^\vee \rangle = n \langle \lambda, \beta^\vee \rangle - a \langle x_{\Psi(\lambda)}, \beta^\vee \rangle 
\geq n-ak_\Phi \geq 0.
\]
If $\beta \in \Phis \smallsetminus \Psi(\lambda)$ this quantity vanishes. Hence
$n \lambda - a x_{\Psi(\lambda)} \in \bY^+$. By Lemma~\ref{lem:length-fW-y+}, we deduce that
\[
t_{n\lambda} w = t_{n\lambda- a x_{\Psi(\lambda)}} t_{ax_{\Psi(\lambda)}} w  \Rleq t_{ax_{\Psi(\lambda)}} w.
\]

On the other hand, let $b$ be an integer such that $b > \max \{ \langle \lambda,\alpha^\vee\rangle : \alpha \in \Phis \}$.  This implies that $bx_{\Psi(\lambda)} - \lambda \in \bY^+$. Then similarly we have
\[
t_{bm \cdot x_{\Psi(\lambda)}} w = t_{m(b x_{\Psi(\lambda)} - \lambda)} t_{m\lambda} w \Rleq t_{m\lambda} w.
\]
Summarizing, we have
\[
t_{bm \cdot x_{\Psi(\lambda)}} w \Rleq t_{m\lambda} w \Rleq t_{n\lambda} w \Rleq t_{ax_{\Psi(\lambda)}} w.
\]
But $t_{bm \cdot x_{\Psi(\lambda)}} w \Rsim t_{ax_{\Psi(\lambda)}} w$ because $bm \geq n(w,x_{\Psi(\lambda)})$ and $a \geq n(w,x_{\Psi(\lambda)})$, so we deduce that $t_{m\lambda} w \Rsim t_{n\lambda} w$.
\end{proof}

\subsection{Reduction to a finite subset of \texorpdfstring{$\fW$}{fW}}

We define
\[
\bY_0 := \{\lambda \in \bY^+ \mid \forall \alpha \in \Phis, \, \langle \lambda, \alpha^\vee \rangle \leq k_\alpha\}.
\]
Of course, $\bY_0$ is a finite set. We then define
\[
Z := \{t_\lambda v : v \in \Wf, \, \lambda \in \bY_0\} \cap \fW,
\]
which is again a finite set.

\begin{prop}
\label{prop:dec-fW}
For any $w \in \fW$, there exists $\lambda \in \bY^+$ and $z \in Z$ such that $w=t_\lambda z$.
\end{prop}

\begin{proof}
By Lemma~\ref{lem:char-fW}, we can
write (uniquely) $w=t_{\mu} v$ with $\mu \in \bY^+$ and $v \in \Wf$ such that $\langle \mu, \alpha^\vee \rangle \geq 1$ for any $\alpha \in \Phi^+$ such that $v^{-1}(\alpha) \in -\Phi^+$. 
We proceed by induction on
\[
s(w):=\sum_{\alpha \in \Phis} \langle \mu, \alpha^\vee \rangle.
\]

If $\mu \in \bY_0$, then $w \in Z$, and there is nothing to prove. (This covers in particular the base case when $s(w)=0$, so that $\mu=0$.) Otherwise, there exists $\alpha \in \Phi_s$ such that $\langle \mu, \alpha^\vee \rangle > k_\alpha$. We fix such an $\alpha$; then we have
\[
w=t_{\varpi_\alpha} \cdot t_{\mu - \varpi_\alpha} v,
\]
and $\mu-\varpi_\alpha \in \bY^+$.

We claim that $w':=t_{\mu - \varpi_\alpha} v \in \fW$. First, we remark that if $\beta \in \Phi^+$ satisfies $v^{-1}(\beta) \in -\Phi^+$, and if $m_{\alpha,\beta}$ is the coefficent of the simple coroot $\alpha^\vee$ in $\beta^\vee$, then
\[
\langle \mu-\varpi_\alpha, \beta^\vee \rangle = \langle \mu, \beta^\vee \rangle - k_\alpha \cdot m_{\alpha,\beta}.
\]
If $m_{\alpha,\beta} = 0$ then $\langle \mu-\varpi_\alpha, \beta^\vee \rangle = \langle \mu, \beta^\vee \rangle \geq 1$. Otherwise we have $\langle \mu, \beta^\vee \rangle \ge \langle \mu, m_{\alpha,\beta}\alpha^\vee\rangle > k_\alpha \cdot m_{\alpha,\beta}$, so again we have $\langle \mu-\varpi_\alpha, \beta^\vee \rangle \ge 1$. By Lemma~\ref{lem:char-fW}, these observations imply our claim.

Clearly, we have $s(w')<s(w)$,
so that by induction there exists $z \in Z$ and $\lambda' \in \bY^+$ such that $w' = t_{\lambda'} z$. Then $w=t_{\lambda'+\varpi_\alpha} z$, and the proof is complete.
\end{proof}

\subsection{Uniform boundedness}

\begin{prop}
\label{prop:bounded}
The set $\{n(w,x_\Psi) : w \in \fW, \, \Psi \subset \Phis\}$ is bounded.
\end{prop}

\begin{proof}
In this proof, for $\lambda \in \bY^+$ we set
\[
\Psi'(\lambda):=\{\alpha \in \Phis \mid \langle \lambda, \alpha^\vee \rangle \geq k_\alpha\}.
\]

We set
\[
A_0:=\max\{n(t_\lambda z,x_\Psi) : z \in Z, \, \lambda \in \bY^+ \text{ with $\Psi'(\lambda)=\varnothing$},\, \Psi \subset \Phis\}.
\]
Then for $i \in \{1, \cdots, |\Phis|\}$ we define $A_i$ by induction as
\begin{multline*}
A_i = \max \bigl( \{A_{i-1}\} \cup \{n(t_\lambda z, x_\Psi) : z \in Z, \, \lambda \in \bY^+ \text{ with $\#\Psi'(\lambda)=i$ and} \\
\text{$\langle \lambda, \alpha^\vee \rangle < A_{i-1} \cdot k_\alpha$ for all $\alpha \in \Psi'(\lambda)$}, \, \Psi \subset \Phis\} \bigr).
\end{multline*}
(Here we take the maximum over a finite set, so that this number is well defined.)
It is clear that
\[
A_0 \leq A_1 \leq \cdots \leq A_{|\Phis|}.
\]

We will prove by induction on $i$ that for any $w=t_\lambda z$ with $z \in Z$ and $\#\Psi'(\lambda)=i$, and for any $\Psi \subset \Phis$, we have
\begin{equation}
\label{eqn:bounded-induction}
n(w,x_{\Psi}) \leq A_i.
\end{equation}
It follows from this that $n(w, x_\Psi) \leq A_{|\Phis|}$ for any $w \in \fW$ and $\Psi \subset \Phis$ by Proposition~\ref{prop:dec-fW}, as desired.

If $i=0$,~\eqref{eqn:bounded-induction} is clear from the definition. Now, assume this claim is known for $i-1$, and let $w=t_\lambda z$ with $z \in Z$ and $\#\Psi'(\lambda)=i$, and $\Psi \subset \Phis$. If $\langle \lambda, \alpha^\vee \rangle < A_{i-1} \cdot k_\alpha$ for all $\alpha \in \Psi'(\lambda)$, then $n(w,x_\Psi) \leq A_i$ by definition. Otherwise, choose $\alpha \in \Psi'(\lambda)$ such that $\langle \lambda, \alpha^\vee \rangle \geq A_{i-1} \cdot k_\alpha$. 
In this case, we will actually prove that
\begin{equation}
\label{eqn:ineq-n}
n(w,x_\Psi) \leq A_{i-1}.
\end{equation}

Indeed, let $k \in \Z$ be maximal such that $\mu:=\lambda-k\varpi_\alpha \in \bY^+$. Then $k \geq A_{i-1}$, and $\#\Psi'(\mu)=i-1$. First, assume that $\alpha \notin \Psi$. Then if $m>n\geq A_{i-1}$ we have
\[
t_{nx_\Psi} w = t_{nx_\Psi} t_{k\varpi_\alpha} t_\mu z \Rleq t_{\min(n,k) \cdot x_{\Psi \cup \{\alpha\}}} t_{\mu} z
\]
and
\[
t_{mx_\Psi} w = t_{mx_\Psi} t_{k\varpi_\alpha} t_\mu z \Rgeq t_{\max(m,k)\cdot x_{\Psi \cup \{\alpha\}}} t_\mu z
\]
by Lemma~\ref{lem:length-fW-y+}. Since $\min(n,k) \geq A_{i-1} \geq n(t_\mu z, x_{\Psi \cup \{\alpha\}})$ by induction, we have
\[
t_{\min(n,k) \cdot x_{\Psi \cup \{\alpha\}}} t_{\mu} z \Rsim t_{\max(m,k)\cdot x_{\Psi \cup \{\alpha\}}} t_\mu z,
\]
which implies that also $t_{nx_\Psi} w \Rsim t_{mx_\Psi} w$. The case $\alpha \in \Psi$ is similar, using the inequalities
\[
t_{nx_\Psi} w \Rleq t_{nx_\Psi} t_\mu z \quad \text{and} \quad t_{mx_\Psi} w \Rgeq t_{(m+k) x_\Psi} t_\mu z,
\]
which finishes the proof.
\end{proof}

\subsection{Finite generation of cells}

\begin{prop}
\label{prop:finite-generation}
Let $\mathbf{c}$ be a right antispherical cell. Then there exists a finite subset $K \subset \mathbf{c}$ such that for any $w \in \mathbf{c}$ there exists $v \in K$ and $\mu \in \bY^+$ such that $w=t_\mu v$.
\end{prop}

\begin{proof}
We set
\[
A=\max \{n(w,x_\Psi) : w \in \fW, \, \Psi \subset \Phis\}
\]
(which is well defined by Proposition~\ref{prop:bounded}). We then set
\[
K:=\mathbf{c} \cap \{t_\lambda z : z \in Z, \, \lambda \in \bY^+ \text{ with $\left\lfloor \frac{\langle \lambda, \alpha^\vee \rangle}{k_\alpha} \right\rfloor \leq A$ for all $\alpha \in \Phis$}\}.
\]
Clearly, $K$ is finite. We will prove that this subset satisfies the property of the proposition.

Let $w \in \mathbf{c}$, and choose $z \in Z$ and $\lambda \in \bY^+$ such that $w=t_\lambda z$ (see Proposition~\ref{prop:dec-fW}). We proceed by induction on
\[
a(\lambda):=\max \left\{ \left\lfloor \frac{\langle \lambda, \alpha^\vee \rangle}{k_\alpha} \right\rfloor : \alpha \in \Phis \right\}.
\]
If $a(\lambda) \leq A$, then $w \in K$, and there is nothing to prove. Otherwise, 
let
\[
\Psi=\left\{\alpha \in \Phis \mid \left\lfloor \frac{\langle \lambda, \alpha^\vee \rangle}{k_\alpha} \right\rfloor = a(\lambda)\right\}.
\]
Then $\lambda - a(\lambda)  x_\Psi \in \bY^+$, and
\[
w=t_{a(\lambda) x_\Psi} t_{\lambda-a(\lambda) x_\Psi} z.
\]
Since $a(\lambda) > A \geq n(t_{\lambda-a(\lambda) x_\Psi} z, x_\Psi)$, if we set $w':=t_{\lambda-x_\Psi} z$ we have
$w \Rsim w'$, or in other words $w' \in \mathbf{c}$.
Since $a(\lambda-x_\Psi) < a(\lambda)$, by induction there exist $v \in K$ and 
$\nu \in \bY^+$ such that $w'=t_\nu v$. Then if we set $\mu=\nu + x_\Psi$, we have $\mu \in \bY^+$ and $w=t_\mu v$.
\end{proof}

\section{Cells and weight cells}
\label{sec:cells}

From now on we assume (as in Section~\ref{sec:support}) that $\Gp_\Z$ is semisimple and simply connected. Recall that in~\S\ref{ss:def-exotic}, we chose a certain weight $\varsigma$ (this played a role in the definition of the functors in~\S\ref{ss:wall-crossing}).  We assume from now on that $\varsigma = \rho$, where $\rho$ is the half-sum of the positive roots.  Finally, we denote by $h$ the Coxeter number of $\Gp_\Z$.

\subsection{Reminder on a result by Ostrik}
\label{ss:reminder-ostrik}

Let $\ell>h$ be an odd integer, which is prime to $3$ if $\Phi$ has a component of type $\mathbf{G}_2$, and let $\UQ$ be Lusztig's quantum group at a primitive $\ell$-th root of unity associated with $\Gp_\Z$. (Here $\UQ$ is obtained by specialization from Lusztig's $\Z[v,v^{-1}]$-form of the $q$-deformed enveloping algebra of $\fg_\C$.) Let $\Rep(\UQ)$ be the category of finite-dimensional $\UQ$-modules of type $\mathbf{1}$; the simple objects in this category are parametrized in a natural way by $\bX^+$, and we denote by $\LQ(\lambda)$ the simple module associated with $\lambda$.

We consider the ``dot-action'' of $W$ on $\bX$ defined by
\[
(w t_\lambda) \cdot_\ell \mu = w(\mu + \ell \lambda + \rho) - \rho
\]
for $w \in \Wf$ and $\lambda, \mu \in \bX$, and the associated \emph{fundamental alcove}
\[
\cC_\ell := \{\lambda \in \bX \mid \forall \alpha \in \Phi^+, \ 0 < \langle \lambda + \rho, \alpha^\vee \rangle < \ell\}.
\]
An \emph{alcove} is a subset of $\bX$ of the form $w \cdot_\ell \cC_\ell$ for some $w \in W$. Such a subset is defined by a family of inequalities $n_\alpha \ell < \langle \lambda+\rho, \alpha^\vee \rangle < (n_\alpha+1)\ell$ for some family of integers $(n_\alpha : \alpha \in \Phi^+)$. The \emph{lower closure} of such an alcove is then the subset defined by the inequalities $n_\alpha \ell\leq \langle \lambda+\rho, \alpha^\vee \rangle < (n_\alpha+1)\ell$ for all $\alpha \in \Phi^+$. Note also that an alcove $w \cdot_\ell \cC_\ell$ intersects $\bX^+$ iff $w \in \fW$.

The category $\Rep(\UQ)$ has a natural structure of highest-weight category with weight poset $(\bX^+,\leq)$ (where $\leq$ is the standard order on $\bX^+$), so that we can consider the indecomposable tilting module $\TQ(\lambda)$ with highest weight $\lambda$. 
Following Ostrik~\cite{ostrik} we define the ``quantum weight preorder'' on $\bX^+$ as follows:
$\lambda \qTleq \mu$ iff there exists a tilting module $M$ in $\Rep(\UQ)$ such that $\TQ(\lambda)$ is a direct summand of $\TQ(\mu) \otimes M$. (Transitivity of $\qTleq$ follows from the fact that a tensor product of tilting modules is again tilting~\cite{paradowski}.)  We denote by $\qTsim$ the associated equivalence relation: $\lambda \qTsim \mu$ iff $\lambda \qTleq \mu$ and $\mu \qTleq \lambda$.

By~\cite[Proposition~8]{andersen}, two dominant weights belonging to the lower closure of a given alcove are equivalent for the relation $\qTsim$. Therefore one can transfer the preorder $\qTleq$ and the equivalence relation $\qTsim$ to $\fW$ as follows: we set $w \qTleq y$ (resp.~$w \qTsim y$) iff there exists a weight $\lambda \in \bX^+$ in the lower closure of $w \cdot_\ell \cC_\ell$ and a weight $\mu \in \bX^+$ in the lower closure of $y \cdot_\ell \cC_\ell$ such that $\lambda \qTleq \mu$ (resp.~$\lambda \qTsim \mu$). The following result is due to Ostrik~\cite{ostrik}.

\begin{thm}
\label{thm:ostrik-cells}
The preorders $\qTleq$ and $\Rleq$ on $\fW$ coincide. In particular, the equivalence classes for the relation $\qTsim$ on $\fW$ are the antispherical right cells.
\end{thm}

Our goal in this section is to prove a counterpart of Theorem~\ref{thm:ostrik-cells} in the setting of modular representations of reductive algebraic groups. Our proof will be essentially identical to that of Ostrik, replacing Soergel's character formula for quantum tilting modules~\cite{soergel-char-tilt} by its modular analogue obtained in~\cite{mkdkm}.

\subsection{Characters of tilting \texorpdfstring{$\Gp_\bk$}{Gk}-modules}
\label{ss:char-tilting}

From now on we let $\bk$ be an algebraically closed field of characteristic $p$, and we
assume that $p>h$. We denote by $\Rep(\Gp_\bk)$ the category of finite-dimensional algebraic $\Gp_\bk$-modules. For any $\lambda \in \bX^+$ we have a costandard (or co-Weyl) module $\coweyl(\lambda) := \Ind_{\BGp_\bk}^{\Gp_\bk}(\lambda)$ and a standard (or Weyl) module $\weyl(\lambda) = (\coweyl(-w_0\lambda))^*$, where $w_0 \in \Wf$ is the longest element. There exists a unique (up to scalar) nonzero morphism $\weyl(\lambda) \to \coweyl(\lambda)$, and its image $\irr(\lambda)$ is simple. It is well known that the category $\Rep(\Gp_\bk)$ is a highest weight category with weight poset $(\bX^+, \leq)$, so that we can consider the indecomposable tilting module $\tilt(\lambda)$ with highest weight $\lambda$. 

We consider the dot-action of $W$ on $\bX$ defined as in~\S\ref{ss:reminder-ostrik} (but with $\ell$ replaced by $p$), and
the \emph{principal block} $\Rep_0(\Gp_\bk)$, i.e. the Serre subcategory of $\Rep(\Gp_\bk)$ generated by the simple objects of the form $\irr(w \cdot_p 0)$ for $w \in \fW$. The linkage principle shows that the subcategory $\Rep_0(\Gp_\bk)$ is a direct summand of $\Rep(\Gp_\bk)$; we denote by
 $\mathrm{pr}_0 : \Rep(\Gp_\bk) \to \Rep_0(\Gp_\bk)$ the corresponding projection functor.

Recall the right $W$-module $\Masph^\circ$ defined in~\S\ref{ss:cells-W-mod}. Then we have a canonical isomorphism
\[
 \vartheta : \mathsf{K}(\Rep_0(\Gp_\bk)) \simto \Masph^\circ
\]
where $[\weyl(w \cdot_p 0)] = [\coweyl(w \cdot_p 0)]$ corresponds to $N_w^\circ$ for any $w \in \fW$. 
For $s \in S$ we choose a weight $\mu_s$ on the $s$-wall of the fundamental alcove $\cC_p$, and set
\[
\Theta_s := T_{\mu_s}^0 T_0^{\mu_s} : \Rep_0(\Gp_\bk) \to \Rep_0(\Gp_\bk),
\]
where $T_{\mu_s}^0$ and $T_0^{\mu_s}$ are the translation functors as in~\cite[\S II.7.6]{jantzen}.
It is well known that for any $s \in S$ and $V$ in $\Rep_0(\Gp_\bk)$ we have
\begin{equation}
\label{eqn:vartheta-wallcrossing}
 \vartheta([\Theta_s(V)]) = \vartheta([V]) \cdot \uH^\circ_s.
\end{equation}

The following result was conjectured in~\cite{rw}, and proved in~\cite{mkdkm}.

\begin{thm}
\label{thm:characters}
 For $w \in \fW$ we have $\vartheta(\tilt(w \cdot_p 0)) = \puN_w^\circ$.
\end{thm}

\begin{rmk}
\label{rmk:char-tilting}
Any tilting $\Gp_\bk$-module $M$ can be ``quantized'' to a tilting module $M_q$ in $\Rep(\UQ)$, where $\UQ$ is constructed using a root of unity of order $\ell=p$; see~\cite[\S 2]{andersen} for details and references. (In fact $M_q$ is characterized up to isomorphism as the unique tilting $\UQ$-module with the same character as $M$.) In particular, if $w \in \fW$, comparing Theorem~\ref{thm:characters} with its quantum counterpart due to Soergel (see~\S\ref{ss:reminder-ostrik}), we see that for $x \in \fW$, the multiplicity of $\TQ(x \cdot_p 0)$ as a direct summand of $\tilt(w \cdot_p 0)_q$ is the coefficient of $\uN_x^\circ$ in the expansion of $\puN_w^\circ$ in the basis $(\uN^\circ_y : y \in \fW)$ of $\Masph^\circ$.
\end{rmk}

We will also need the following technical result. (A very similar claim appears in~\cite{ostrik}.)

\begin{lem}
\label{lem:vartheta-tensor}
 Let $M \in \Rep(\Gp_\bk)$, and
 set $\Lambda_M:=\{\lambda \in W \cdot_p 0 \mid \dim(M_\lambda) \neq 0\}$. For any $\lambda \in \Lambda_M$, we denote by $x^\lambda$ the unique element of $W$ such that $x^\lambda \cdot_p 0 = \lambda$, and then
define
 \[
  c(M):=\sum_{\lambda \in \Lambda_M} \dim(M_\lambda) \cdot x^\lambda \quad \in \Z[W].
 \]
 For any $V$ in $\Rep_0(\Gp_\bk)$, we have
 \[
  \vartheta([\pr_0(V \otimes M)]) = \vartheta([V]) \cdot c(M).
 \]
\end{lem}

\begin{proof}
 It suffices to prove the equality when $V=\weyl(w \cdot_p 0)$ for some $w \in \fW$. 
 For any $\lambda \in \bX$ we set
 \[
  \chi(\lambda) = \frac{\sum_{w \in \Wf} (-1)^{\ell(w)} e^{w \cdot_p \lambda}}{\sum_{w \in \Wf} (-1)^{\ell(w)} e^{w \cdot_p 0}} \quad \in \Z[\bX].
 \]
 Then by Weyl's character formula (see~\cite[\S II.5.11]{jantzen}) and standard arguments, the character of $\weyl(w \cdot_p 0) \otimes M$ is
 \[
  \sum_{\lambda \in \mathrm{wt}(M)} \dim(M_\lambda) \cdot \chi(w \cdot_p 0+\lambda) = \sum_{\lambda \in \mathrm{wt}(M)} \dim(M_\lambda) \cdot \chi(w \cdot_p \lambda),
 \]
where $\mathrm{wt}(M)$ is the set of weights of $M$. (The second equality uses the fact that $\Wf$ permutes $\mathrm{wt}(M)$ without changing the multiplicities.) Therefore, we have
\[
 \vartheta([\pr_0(\weyl(w \cdot_p 0) \otimes M)]) = \sum_{\lambda \in \Lambda_M} \dim(M_\lambda) \cdot N^\circ_{wx^\lambda}= N^\circ_w \cdot c(M),
\]
which finishes the proof.
\end{proof}

\subsection{Weight order and weight cells}

Recall that the tensor product of two tilting $\Gp_\bk$-modules is again tilting~\cite{mathieu}. Using this property and copying the constructions in~\S\ref{ss:reminder-ostrik} one can define the following preorder and associated equivalence relation.

\begin{defn}
 We define the preorder $\Tleq$ on $\bX^+$ by
 \[
  \lambda \Tleq \mu \quad \text{iff $\tilt(\lambda)$ is a direct summand of $\tilt(\mu) \otimes M$ for a tilting $\Gp_\bk$-module $M$.}
 \]
We denote by $\Tsim$ the equivalence relation on $\bX^+$ defined by
\[
 \lambda \Tsim \mu \quad \text{iff $\lambda \Tleq \mu$ and $\mu \Tleq \lambda$.}
\]
The equivalence classes for this relation will be called \emph{weight cells}.
\end{defn}

As in the quantum case we have the following property (also stated in~\cite[\S 4]{andersen}).

\begin{lem}
\label{lem:Tleq-alcoves}
 If $\lambda$ and $\mu$ belong to the lower closure of the same alcove, then $\lambda \Tsim \mu$.
\end{lem}

\begin{proof}
 We can assume that $\lambda$ belongs to $W \cdot_p 0$. Then by~\cite[Proposition~3.1.3]{hardesty} we have $T_\mu^\lambda(\tilt(\mu)) \cong \tilt(\lambda)$, and $T_\lambda^\mu(\tilt(\lambda))$ is a (non-empty) sum of copies of $\tilt(\mu)$.
 Since translation functors can be described as tensoring with a tilting module and then taking a direct summand (see~\cite[Remark~7.6(1)]{jantzen}), the claim follows.
\end{proof}

Since the set of alcoves meeting $\bX^+$ is in a natural bijection with $\fW$ (through $w \mapsto w \cdot_p \mathcal{C}$), as in the quantum case
we use this lemma to transfer the preorder $\Tleq$ and the equivalence relation $\Tsim$ to $\fW$: we set $w \Tleq y$, resp.~$w \Tsim y$, iff $w \cdot_p 0 \Tleq y \cdot_p 0$, resp.~$w \cdot_p 0 \Tsim y \cdot_p 0$. Then given $\lambda, \mu \in \bX^+$, there exist unique $w,y\in W$ such that $\lambda$ belongs to the lower closure of $w \cdot_p \cC_p$ and $\mu$ belongs to the lower closure of $y \cdot_p \cC_p$; with this notation $\lambda \Tleq \mu$ iff $w \Tleq y$.

The following theorem is the ``modular analogue" of Theorem~\ref{thm:ostrik-cells} promised in~\S\ref{ss:reminder-ostrik}.

\begin{thm}
\label{thm:ostrik}
 For $w,y \in \fW$, we have
 \[
  w \Tleq y \quad \text{iff} \quad w \pRleq y.
 \]
 In particular, the equivalence classes for the relation $\Tleq$ on $\fW$ are the antispherical right $p$-cells.
\end{thm}

\begin{proof}
 To prove that
 \[
  w \pRleq y \quad \Rightarrow \quad w \Tleq y,
 \]
by Lemma~\ref{lem:order-generated}\eqref{it:order-generated-2}, it suffices to prove that if $s \in S$ and $\puN_w$ appears in $\puN_y \cdot \uH_s$, then $w \Tleq y$. However, if $\puN_w$ appears in $\puN_y \cdot \uH_s$, then as in the proof of Lemma~\ref{lem:pRleq-M0} we see that $\puN^\circ_w$ appears in $\puN^\circ_y \cdot \puH^\circ_s$. Using~\eqref{eqn:vartheta-wallcrossing} and Theorem~\ref{thm:characters}, we deduce that $\tilt(w \cdot_p 0)$ is a direct summand of $\Theta_s(\tilt(y \cdot_p 0))$. As in the proof of Lemma~\ref{lem:Tleq-alcoves}, this implies that $w \Tleq y$.

On the other hand, assume that $w \Tleq y$, i.e.~that $\tilt(w \cdot_p 0)$ is a direct summand of $\tilt(y \cdot_p 0) \otimes M$ for some tilting $\Gp_\bk$-module $M$. Then by Lemma~\ref{lem:vartheta-tensor} and Theorem~\ref{thm:characters} $\puN_w^\circ$ appears with nonzero coefficient in $\puN_y^\circ \cdot c(M)$. By Lemma~\ref{lem:pRleq-M0}, this implies that $w \pRleq y$, as desired.
\end{proof}

\subsection{Thick tensor ideals}

Let $\Tilt(\Gp_\bk) \subset \Rep(\Gp_\bk)$ be the category of tilting $\Gp_\bk$-modules. For any $M$ in $\Tilt(\Gp_\bk)$, we denote by $\langle M \rangle_{\mathsf{T}}$ the thick tensor ideal of $\Tilt(\Gp_\bk)$ generated by $M$, i.e.~the smallest strictly full subcategory of $\Tilt(\Gp_\bk)$ containing $M$ and which is stable under direct summands and under tensoring with any tilting $\Gp_\bk$-module.

It is clear from definition that
\[
w \Tleq y \quad \Leftrightarrow \quad \tilt(w \cdot_p 0) \in \langle \tilt(y \cdot_p 0) \rangle_{\mathsf{T}} \quad \Leftrightarrow  \quad \langle \tilt(w \cdot_p 0) \rangle_{\mathsf{T}} \subset \langle \tilt(y \cdot_p 0) \rangle_{\mathsf{T}}.
\]
From this remark we deduce the following reformulation of Theorem~\ref{thm:ostrik}.

\begin{cor}
\label{cor:tensor-ideals}
For $w,y \in \fW$, we have
\[
w \pRsim y \quad \Leftrightarrow \quad \langle \tilt(w \cdot_p 0) \rangle_{\mathsf{T}} = \langle \tilt(y \cdot_p 0) \rangle_{\mathsf{T}}.
\]
\end{cor}

\subsection{An application}

We now observe that the following result (originally due to Georgiev--Mathieu~\cite{georgiev-mathieu}; see also~\cite{mathieu-tilting}) is an immediate consequence of Theorem~\ref{thm:ostrik}. (Here, for a module $M$ which admits a standard filtration and $\mu \in \bX^+$, we denote by $(M : \weyl(\mu))_\Delta$ the multiplicity of $\weyl(\mu)$ in a standard filtration of $M$.)

\begin{prop}
\label{prop:georgiev-mathieu}
Let $\lambda \in \bX^+$. The following conditions are equivalent:
\begin{enumerate}
\item
$\lambda \in \mathcal{C}_p$;
\item
$p \nmid \dim(\tilt(\lambda))$;
\item
there exists $\mu \in \cC_p$ such that
$\sum_{w \in \fW} (-1)^{\ell(w)} (\tilt(\lambda) : \weyl(w \cdot_p \mu))_\Delta \neq 0$.
\end{enumerate}
\end{prop}

\begin{proof}
If $\lambda \in \mathcal{C}_p$, then $\tilt(\lambda) = \coweyl(\lambda)=\weyl(\lambda)$, and Weyl's dimension formula shows that $p \nmid \dim(\tilt(\lambda))$. It is also clear that
\[
\sum_{w \in \fW} (-1)^{\ell(w)} (\tilt(\lambda) : \weyl(w \cdot_p \lambda))_\Delta = 1 \neq 0.
\]

Now, assume that $p \nmid \dim(\tilt(\lambda))$, and let $v \in \fW$ be such that $\lambda$ belongs to the lower closure of $v \cdot_p \mathcal{C}_p$. Under our assumption the $\Gp_\bk$-module $\bk=\tilt(0)$ is a direct summand in $\tilt(\lambda) \otimes \tilt(\lambda)^*$, so $0 \Tleq \lambda$, which implies that $0 \Tleq v \cdot_p 0$ by Lemma~\ref{lem:Tleq-alcoves}. Using Theorem~\ref{thm:ostrik}, we deduce that $1 \pRleq v$. As explained in Example~\ref{ex:cell-1}, this implies that $v=1$, i.e.~that $\lambda \in \mathcal{C}_p$.

Finally, we prove that if $\lambda \notin \mathcal{C}_p$, for any $\mu \in \mathcal{C}_p$ we have
\begin{equation}
\label{eqn:formula-mult-C}
\sum_{w \in \fW} (-1)^{\ell(w)} (\tilt(\lambda) : \weyl(w \cdot_p \mu))_\Delta=0.
\end{equation}
In fact, if $\lambda \notin W \cdot_p \mu$ 
this equality is clear since $(\tilt(\lambda) : \weyl(w \cdot_p \mu))_\Delta=0$ for any $w$ by the linkage principle. We prove the case $\lambda = v \cdot_p \mu$ for some $v \in \fW$ by induction on $\ell(v)$.
Write $v=us$ with $s \in S$ and $\ell(v)=\ell(u)+1$ (so that $u \in \fW$),
and choose $\nu \in \bX$ on the $s$-wall of $\mathcal{C}_p$. 
We observe that $T_{\nu}^{\mu} \tilt(u \cdot_p \nu)$ is the direct sum of $\tilt(\lambda)$ and some modules of the form $\tilt(x \cdot_p \mu)$ with $\ell(x) < \ell(v)$. Now we have
\[
\sum_{w \in \fW} (-1)^{\ell(w)} (T_{\nu}^{\mu} \tilt(u \cdot_p \nu) : \weyl(w \cdot_p \mu))_\Delta=0
\]
(because this equality holds if $\tilt(u \cdot_p \nu)$ is replaced by any standard module, in which case it follows from~\cite[Proposition~II.7.12]{jantzen}), and $\tilt(\mu)$ is not a direct summand of $T_{\nu}^{\mu} \tilt(u \cdot_p \nu)$ (because otherwise we would have $\mu \Tleq u \cdot_p \nu$, which is excluded by Theorem~\ref{thm:ostrik}). Hence we obtain~\eqref{eqn:formula-mult-C} using the induction hypothesis.
\end{proof}

The following consequence is also stated in~\cite{georgiev-mathieu, mathieu-tilting} (where the second formula is called the \emph{modular Verlinde formula}); we recall the proof for completeness. (Here, for $\lambda, \mu, \nu \in \bX^+$, we denote by $K_{\lambda,\mu}^\nu$ the multiplicity of the simple module for the complex group $\Gp_\C$
with highest weight $\nu$ in the tensor product of the simple modules of highest weights $\lambda$ and $\mu$.)

\begin{cor}
\label{cor:Verlinde}
\leavevmode
\begin{enumerate}
\item
\label{it:Verlinde-1}
For any $\lambda \in \mathcal{C}_p$ and any tilting module $M$, the multiplicity of $\tilt(\lambda)$ as a direct summand of $M$ is equal to $\sum_{w \in \fW} (-1)^{\ell(w)} (M : \weyl(w \cdot_p \lambda))_\Delta$.
\item
\label{it:Verlinde-2}
For any $\lambda, \mu, \nu \in \mathcal{C}_p$, the multiplicity of $\tilt(\nu)$ as a direct summand of $\tilt(\lambda) \otimes \tilt(\mu)$ is equal to $\sum_{w \in \fW} (-1)^{\ell(w)} K_{\lambda,\mu}^{w \cdot_p \nu}$.
\end{enumerate}
\end{cor}

\begin{proof}
\eqref{it:Verlinde-1}
Of course it is sufficient to prove the formula when $M$ is indecomposable, i.e.~of the form $\tilt(\mu)$ for some $\mu \in \bX^+$. If $\mu \notin \mathcal{C}_p$, the claim follows from Proposition~\ref{prop:georgiev-mathieu}. If $\mu \in \mathcal{C}_p$ we have $\tilt(\mu)=\weyl(\mu)$, so that the claim is obvious.

\eqref{it:Verlinde-2}
By~\eqref{it:Verlinde-1}, the multiplicity we want to compute is equal to
\begin{multline*}
\sum_{w \in \fW} (-1)^{\ell(w)} (\tilt(\lambda) \otimes \tilt(\mu) : \weyl(w \cdot_p \nu))_\Delta \\
= \sum_{w \in \fW} (-1)^{\ell(w)} (\weyl(\lambda) \otimes \weyl(\mu) : \weyl(w \cdot_p \nu))_\Delta.
\end{multline*}
The integer $(\weyl(\lambda) \otimes \weyl(\mu) : \weyl(w \cdot_p \nu))_\Delta$ is the coefficient of $\chi(w \cdot_p \nu)$ in the product $\chi(\lambda) \cdot \chi(\mu)$, where $\chi(\eta)$ is as in the proof of Lemma~\ref{lem:vartheta-tensor}.
This coefficient can be interpreted as stated in terms of characteristic-$0$ representation theory.
\end{proof}

\begin{rmk}
\label{rmk:multplicities-mod-quantum}
In the setting of~\S\ref{ss:reminder-ostrik}, the same arguments as for Corollary~\ref{cor:Verlinde} show that if $V$ is a tilting module in $\Rep(\UQ)$ and if $\lambda \in \mathcal{C}_\ell$, the multiplicity of $\TQ(\lambda)$ as a direct summand of $V$ is equal to $\sum_{w \in \fW} (-1)^{\ell(w)} (V : \weyl_q(w \cdot_\ell \lambda))_\Delta$, where $\weyl_q(\nu)$ is the quantum Weyl module associated with $\nu$, and as above $(- : -)_\Delta$ means the multiplicity in a standard filtration. (In this setting, this claim is due to Andersen--Paradowski~\cite{andersen-paradowski}.) Since the modular and quantum Weyl modules have the same character, in the setting of Remark~\ref{rmk:char-tilting} we deduce that for any tilting module $M$ in $\Rep(\Gp_\bk)$ and any $\lambda \in \mathcal{C}_p$, the multiplicity of $\tilt(\lambda)$ as a direct summand of $M$ is equal to the multiplicity of $\TQ(\lambda)$ as a direct summand of $M_q$.
\end{rmk}

\section{The Humphreys conjecture}
\label{sec:humphreys-conj}

We continue with the assumptions of~\S\ref{ss:char-tilting}; but to simplify notation (and since $\bk$ is fixed in this section) we write $G=\Gp_\bk$. Note that since $G$ is obtained by base change to $\bk$ of a $\Z$-group scheme, hence also by base change of an $\F_p$-group scheme, its Frobenius twist identifies canonically with $G$. We also set $\cN=\cN_\bk$.

\subsection{Support varieties}

We denote 
by $G_1$ the first Frobenius kernel of $G$, i.e.~the kernel of the Frobenius morphism $\Fr : G \to G$. Since the subgroup $G_1 \subset G$ is distinguished, the algebra $\Ext^\bullet_{G_1}(\bk,\bk)$ admits a natural action of $G$, which factors through an action of $G/G_1 \cong G$. It is also clearly $\Z$-graded, hence can be considered as a $G \times \Gm$-equivariant algebra.

The following lemma will be proved in~\S\ref{ss:relation-tilt}. (This statement is well known, see e.g.~\cite[\S 3]{andersen-jantzen} for a proof under the same assumptions as ours; later we will use our explicit construction of this isomorphism, however.)

\begin{lem}
\label{lem:cohom-G1}
There exists a canonical isomorphism of graded $G \times \Gm$-equivariant algebras
\[
\Ext^\bullet_{G_1}(\bk,\bk) \cong \cO(\cN).
\]
\end{lem}

If $M,N$ are $G$-modules, the graded vector space $\Ext^\bullet_{G_1}(M,N)$ has a natural $G \times \Gm$-action. It also has a natural structure of module over $\Ext_{G_1}^\bullet(\bk,\bk)$, see~\cite[\S 2.2]{npv}.
Using the isomorphism of Lemma~\ref{lem:cohom-G1}, this object can then be considered as a $G \times \Gm$-equivariant quasi-coherent sheaf on $\cN$. Therefore it makes sense to consider the support $\supp(\Ext^\bullet_{G_1}(M,N))$; see~\S\ref{ss:support-cN}.

Of course, the $G \times \Gm$-equivariant quasi-coherent sheaf on $\cN$ associated with $\Ext^\bullet_{G_1}(M,N)$
depends on the choice of isomorphism in Lemma~\ref{lem:cohom-G1}.
Note however that the results of~\cite{serre} imply that any $G$-equivariant automorphism of $\cN$ induces the identity on the set of orbits $\cN/ G$. Therefore, the support 
$\supp(\Ext^\bullet_{G_1}(M,N))$ is well defined, i.e.~does not depend on any choice.

Following~\cite{npv} we set $V_{G_1}(M,N) := \supp(\Ext^\bullet_{G_1}(M,N))$.
We will also use the special notation
\[
V_{G_1}(M) := V_{G_1}(M,M), \qquad \oVG(M) := V_{G_1}(\bk,M).
\]
The subvariety $V_{G_1}(M)$ is called the \emph{support variety} of $M$.

It is clear that $V_{G_1}(M,N) \subset V_{G_1}(M) \cap V_{G_1}(N)$, see~\cite[(2.2.1)]{npv}; in particular we have
\begin{equation}
\label{eqn:V-oV}
\oVG(M) \subset V_{G_1}(M).
\end{equation}
On the other hand, by~\cite[(2.2.4)--(2.2.5)]{npv}, for $M,N$ in $\Rep(G)$ we have
\begin{gather}
\label{eqn:support-sum}
V_{G_1}(M \oplus N) = V_{G_1}(M) \cup V_{G_1}(N);\\
\label{eqn:support-tensor}
V_{G_1}(M \otimes N) = V_{G_1}(M) \cap V_{G_1}(N).
\end{gather}

The following property is also stated in~\cite[\S 3.1]{hardesty}.

\begin{lem}
\label{lem:V-Tleq}
If $\lambda, \mu \in \bX^+$, we have
\[
\mu \Tleq \lambda \quad \Rightarrow \quad V_{G_1}(\tilt(\mu)) \subset V_{G_1}(\tilt(\lambda)).
\]
In particular, if $\mu \Tsim \lambda$ then $V_{G_1}(\tilt(\mu)) = V_{G_1}(\tilt(\lambda))$.
\end{lem}

\begin{proof}
If $\mu \Tleq \lambda$, then there exist tilting $G$-modules $M,N$ such that
\[
\tilt(\lambda) \otimes M \cong \tilt(\mu) \oplus N.
\]
Then we have
\[
V_{G_1}(\tilt(\mu)) \subset V_{G_1}(\tilt(\mu) \oplus N) = V_{G_1}(\tilt(\lambda) \otimes M) \subset V_{G_1}(\tilt(\lambda))
\]
by~\eqref{eqn:support-sum}--\eqref{eqn:support-tensor}, proving the desired claim.
\end{proof}

\subsection{The Humphreys conjecture}
\label{ss:humphreys-conj}

Recall that the main result of~\cite{lusztig} provides a canonical bijection between the set $\cN_\C/\Gp_\C$ of $\Gp_\C$-orbits in $\cN_\C$ and the set of Kazhdan--Lusztig two-sided cells in $W$. (The two-sided cells considered here are defined as in the special case $p=0$ of Definition~\ref{defn:order-cells}, but allowing multiplication both on the left and on the right by elements of $\cH$.)
Then, the main result of~\cite{lx} implies that taking the intersection with $\fW$ induces a bijection between the set of Kazhdan--Lusztig two-sided cells in $W$ and the set of antispherical right cells in $\fW$. 
Combining these bijections,
we obtain a canonical bijection between the set of orbits $\cN_\C/\Gp_\C$ and the set of antispherical right cells in $\fW$. 
For $\bc \subset \fW$ an antispherical right cell, we will denote by $\scO_{\bc}^\C \subset \cN_\C$ the corresponding orbit.

\begin{rmk}
\label{rmk:orbit-LV}
Now that this notation is introduced, we can make the description of the orbit $\scO$ appearing in Theorem~\ref{thm:support-char-0}\eqref{it:support-char0-1} slightly more explicit: by definition this is the orbit appearing in the pair corresponding to $\lambda$ under the \emph{Lusztig--Vogan bijection} considered in~\cite[\S 4.3]{achar}.
It follows from~\cite[Remark~6]{bezru-perverse} (see also the remarks surrounding~\cite[Conjecture~3]{ostrik-Kth}) that $\scO = \scO_{\bc(w)}^\C$, where $w \in \fW$ is the unique element such that there exists $\omega \in \Omega$ such that $w_\lambda = w\omega$. It will follow from the proof in~\S\ref{ss:proof-support-bezru} that the orbit appearing in Theorem~\ref{thm:support-char-0}\eqref{it:support-char0-2} can be described in the same way.
\end{rmk}

\begin{lem}
\label{lem:order-orbits}
For $y,w \in \fW$, we have
\[
y \Rleq w \quad \Rightarrow \quad \overline{\iota_\bk(\scO^\C_{\bc(y)})} \subset \overline{\iota_\bk(\scO^\C_{\bc(w)})}.
\]
\end{lem}

\begin{proof}
If $y \Rleq w$, then $y \leq_{\mathrm{LR}} w$ (where $\leq_{\mathrm{LR}}$ is the two-sided Kazhdan--Lusztig order). By~\cite[Theorem~4(b)]{bezru-perverse}, this implies that $\overline{\scO^\C_{\bc(y)}} \subset \overline{\scO^\C_{\bc(w)}}$. Since the bijection $\iota_\bk$ is an isomorphism of posets (see~\S\ref{ss:support-cN}), we deduce that $\overline{\iota_\bk(\scO^\C_y)} \subset \overline{\iota_\bk(\scO_w)}$.
\end{proof}

The following conjecture is due to Humphreys, see~\cite{humphreys}.

\begin{conj}
\label{conj:Humphreys}
For any $w \in \fW$, if $\lambda \in \bX^+$ belongs to the lower closure of $w \cdot_p \mathcal{C}_p$ then
\[
V_{G_1}(\tilt(\lambda)) = \overline{\iota_\bk(\scO^\C_{\bc(w)})}.
\]
\end{conj}

By Lemma~\ref{lem:Tleq-alcoves} and Lemma~\ref{lem:V-Tleq}, $V_{G_1}(\tilt(\lambda))$ only depends on the alcove whose lower closure contains $\lambda$. In particular, Conjecture~\ref{conj:Humphreys} is equivalent to the following.

\begin{conj}
\label{conj:Humphreys-regular}
For any $w \in \fW$,
\[
V_{G_1}(\tilt(w \cdot_p 0)) = \overline{\iota_\bk(\scO^\C_{\bc(w)})}.
\]
\end{conj}

\subsection{A relative version of Conjecture~\ref{conj:Humphreys}}

Recall the subset $\fWf \subset W$ introduced in~\S\ref{ss:def-exotic}.

\begin{lem}
\label{lem:cohom-G1-fWf}
Let $w \in \fW$, and assume that $w \notin \fWf$. Then $\Ext^\bullet_{G_1}(\bk,\tilt(w \cdot_p 0)) = 0$.
\end{lem}

\begin{proof}
Let $s \in \Sf$ be such that $ws < w$. Then $ws \in \fW$, and by 
\cite[E.11(1)]{jantzen} we have $\Theta_s \tilt(w\cdot_p 0) \cong \tilt(w\cdot_p 0) \oplus \tilt(w\cdot_p 0)$. Hence 
\[
\Ext^\bullet_{G_1}(\bk,\tilt(w\cdot_p 0)) = 0 \quad \text{iff} \quad \Ext^\bullet_{G_1}(\bk,\Theta_s \tilt(w\cdot_p 0)) = 0.
\]

According to~\cite[\S II.9.22]{jantzen}, $\Theta_s$ can be extended to a self-adjoint endofunctor of the category of rational
$G_1T$-modules (where $T=\TGp_\bk$) with the property that 
\[
\Theta_s(M \otimes \bk_{G_1T}(p\lambda)) \cong \Theta_s(M)\otimes \bk_{G_1 T}(p\lambda).
\]
for any $\lambda \in \bX$ and any $G_1T$-module $M$. By comparing $G_1$ and $G_1T$ cohomology, we deduce that 
\[
\Ext^\bullet_{G_1}(\bk,\Theta_s \tilt(w\cdot_p 0)) \cong \Ext^\bullet_{G_1}(\Theta_s \bk,\tilt(w\cdot_p 0)). 
\]
However, we have
$\Theta_s \bk =0$, which implies the desired vanishing.
\end{proof}

\begin{lem}
\label{lem:cells-fWf}
For any $\lambda \in \bX^+$, there exists $w \in \fWf$ such that 
$\lambda \Tsim w \cdot_p 0$.
\end{lem}

\begin{proof}
Let us choose a decomposition
\[
\tilt(\lambda) \otimes \tilt(\lambda)^* = \bigoplus_{i \in I} M_i
\]
where each $M_i$ is an indecomposable tilting module. We also set $J:=\{i \in I \mid \Hom_G(\bk, M_i) \neq 0\}$. (Note that since there exists a nonzero morphism $\bk \to \tilt(\lambda) \otimes \tilt(\lambda)^*$, we have $J \neq \varnothing$.) We also set $M:=\bigoplus_{j \in J} M_j$. For any $j \in J$, since $M_j$ is indecomposable and contains $\bk$ as a submodule, it belongs to $\Rep_0(G)$, hence is isomorphic to $\tilt(w_j \cdot_p 0)$ for some $w_j \in \fW$. Moreover $\Ext^\bullet_{G_1}(\bk, M_j) \neq \{0\}$; by Lemma~\ref{lem:cohom-G1-fWf} this implies that $w_j \in \fWf$. 

It is clear that for any $j \in J$ we have $w_j \cdot_p 0 \Tleq \lambda$. Hence to conclude the proof it suffices to prove that $\lambda \Tleq w_j \cdot_p 0$ for some $j$, i.e.~that $\tilt(\lambda)$ is a direct summand of $M \otimes N$ for some tilting $G$-module $N$.
In fact, we claim that $\tilt(\lambda)$ is a direct summand of $M \otimes \tilt(\lambda)$. Indeed, consider the morphisms of $G$-modules
\[
\tilt(\lambda) \xrightarrow{\phi} \tilt(\lambda) \otimes \tilt(\lambda)^* \otimes \tilt(\lambda) \xrightarrow{\psi} \tilt(\lambda)
\]
defined by $x \mapsto \id \otimes x$ and $f \otimes y \mapsto f(y)$ (where we identify $\tilt(\lambda) \otimes \tilt(\lambda)^*$ with $\End_\bk(\tilt(\lambda))$). It is clear that $\psi \circ \phi=\id$, and that $\phi$ factors through $M \otimes \tilt(\lambda)$. Hence indeed $\tilt(\lambda)$ is a direct summand of $M \otimes \tilt(\lambda)$, and the proof is complete.
\end{proof}

\begin{rmk}
\label{rmk:intersection-fWf}
Comparing Theorem~\ref{thm:ostrik} and Lemma~\ref{lem:cells-fWf} we see that for any antispherical right $p$-cell $\mathbf{c}$ we have $\mathbf{c} \cap \fWf \neq \varnothing$. This does not seem to be obvious from the definition.
\end{rmk}

The following can be considered as a ``relative'' version of Conjecture~\ref{conj:Humphreys-regular}. 

\begin{conj}
\label{conj:Humphreys-relative}
For any 
$w \in \fWf$, we have
\[
\oVG(\tilt(w \cdot_p 0)) = \overline{\iota_\bk(\scO^\C_{\bc(w)})}.
\]
\end{conj}

The following lemma explains the relation between this ``relative'' version and the original Humphreys conjecture.

\begin{lem}
\label{lem:Humphreys-classical-relative}
Assume that
Conjecture~{\rm \ref{conj:Humphreys-relative}} holds, and moreover that if $x,y \in \fW$ we have
\begin{equation}
\label{eqn:pRleq-Rleq}
x \pRleq y \quad \Rightarrow \quad x \Rleq y.
\end{equation}
Then Conjecture~{\rm \ref{conj:Humphreys}} holds.
\end{lem}

\begin{proof}
We assume that Conjecture~\ref{conj:Humphreys-relative} and~\eqref{eqn:pRleq-Rleq} hold. Property~\eqref{eqn:pRleq-Rleq} implies that the function $w \mapsto \scO^\C_{\bc(w)}$ is constant on antispherical $p$-cells. On the other hand, by Lemma~\ref{lem:V-Tleq} the function $\lambda \mapsto V_{G_1}(\tilt(\lambda))$ is constant on weight cells.
Hence, using Theorem~\ref{thm:ostrik}, to prove Conjecture~\ref{conj:Humphreys}, it suffices to prove that for each weight cell, there exists $w \in \fW$ and $\lambda$ in the lower closure of $w \cdot_p \cC_p$ which belongs to the given weight cell and such that $V_{G_1}(\tilt(\lambda))=\overline{\iota_\bk(\scO^\C_{\bc(w)})}$.
Hence by Lemma~\ref{lem:cells-fWf} we can assume that $\lambda = w \cdot_p 0$ for some $w \in \fWf$.
In this case, by assumption we know that $\oVG(\tilt(\lambda)) = \overline{\iota_\bk(\scO^\C_{\bc(w)})}$. Using~\eqref{eqn:V-oV}, this implies that $\overline{\iota_\bk(\scO^\C_{\bc(w)})} \subset V_{G_1}(\tilt(\lambda))$. 

On the other hand, we have
\[
\Ext^\bullet_{G_1}(\tilt(\lambda), \tilt(\lambda)) \cong \Ext^\bullet_{G_1}(\bk, \tilt(\lambda) \otimes \tilt(\lambda)^*).
\]
All the indecomposable direct summands of $\tilt(\lambda) \otimes \tilt(\lambda)^*$ are of the form $\tilt(\mu)$ with $\mu \in \bY^+$ and $\mu \Tleq \lambda$, and those which contribute to $\Ext^\bullet_{G_1}(\bk, \tilt(\lambda) \otimes \tilt(\lambda)^*)$ satisfy in addition $\mu=v \cdot_p 0$ for some $v \in \fWf$ by the linkage principle for $G_1$-modules (see~\cite[Lemma~II.9.19]{jantzen}) and Lemma~\ref{lem:cohom-G1-fWf}.
To conclude, we only have to prove that for all such $\mu$ we have $\oVG(\tilt(\mu)) \subset \overline{\iota_\bk(\scO^\C_{\bc(w)})}$. However, by Conjecture~\ref{conj:Humphreys-relative} (which we assume to hold) we have $\oVG(\tilt(\mu)) = \overline{\iota_\bk(\scO^\C_{\bc(v)})}$. 
Since $v \cdot_p 0 \Tleq w \cdot_p 0$, we have
$v \Tleq w$. By Theorem~\ref{thm:ostrik} and~\eqref{eqn:pRleq-Rleq}, this implies that $v \Rleq w$. Finally, using Lemma~\ref{lem:order-orbits} we deduce that $\overline{\iota_\bk(\scO^\C_{\bc(v)})} \subset \overline{\iota_\bk(\scO^\C_{\bc(w)})}$, and the proof is complete.
\end{proof}

\subsection{The quantum case}
\label{ss:quantum-Humphreys}

Let us come back to the setting of~\S\ref{ss:reminder-ostrik}. All the constructions considered in the present section have obvious analogues for $\UQ$-modules, replacing $p$ by $\ell$, the Frobenius kernel $G_1$ by Lusztig's small quantum group $\uQ$, and omitting the bijection $\iota_\bk$. (In this case, the analogue of Lemma~\ref{lem:cohom-G1} is due to Ginzburg--Kumar~\cite{gk}.) In particular, we can consider the quantum analogues of Conjecture~\ref{conj:Humphreys} and Conjecture~\ref{conj:Humphreys-relative}.

The quantum analogue of Conjecture~\ref{conj:Humphreys} was proved by Ostrik in~\cite{ostrik-supp} in the special case $\Gp_\Z=\mathrm{SL}_{n,\Z}$. Then the general case of the quantum analogue of Conjecture~\ref{conj:Humphreys-relative} was proved by Bezrukavnikov in~\cite{bezru}. The same arguments as in the proof of Lemma~\ref{lem:Humphreys-classical-relative} show that this implies the quantum version of Conjecture~\ref{conj:Humphreys} in full generality. (In this setting we do not need any analogue of~\eqref{eqn:pRleq-Rleq}.)

\subsection{Proof of Theorem~\texorpdfstring{\ref{thm:support-char-0}\eqref{it:support-char0-2}}{}}
\label{ss:proof-support-bezru}

We consider the setting of Section~\ref{sec:support} (without assuming that $\bk=\C$ for now). Recall the group $\Gp^\vee$ considered in~\S\ref{ss:relation-Gr}, and let $\Gr':= \Gp^\vee(\mathscr{O}) \backslash \Gp^\vee(\mathscr{K})$ be its ``opposite affine Grassmannian'' (where $\mathscr{K}:=\C(\hspace{-1pt}(t)\hspace{-1pt})$). Let also $\mathbf{I} \subset \Gp^\vee(\mathscr{O})$ be the Iwahori subgroup associated with the \emph{negative} Borel subgroup of $\Gp^\vee$. Then, as in~\S\ref{ss:relation-Gr}, there exists an equivalence of triangulated categories
\[
\Upsilon' : \Dmix_{(\mathbf{I})}(\Gr',\bk) \simto \Db \Coh^{\Gp \times \Gm}(\tcN)_\bk
\]
which satisfies $\Upsilon' \circ \langle 1 \rangle \cong \langle 1 \rangle [1] \circ \Upsilon'$ and which sends the normalized standard object $\cJ'_!(\lambda)$ associated with $\lambda$ to $\Delta_\bk(\lambda)$, for any $\lambda \in \bX$ (see in particular~\cite[Remark~11.3(2)]{prinblock}). Again as in~\S\ref{ss:relation-Gr}, this equivalence sends the tilting mixed perverse sheaf $\cT'(\lambda)$ associated with $\lambda$ to $\PEx^\bk_\lambda$.

We consider the Grothendieck group $[\Dmix_{(\mathbf{I})}(\Gr',\bk)]$, and the embedding
\[
\Masph \hookrightarrow [\Dmix_{(\mathbf{I})}(\Gr',\bk)]
\]
sending $v^m N_{w_\lambda}$ to $[\cJ'_!(\lambda) \langle m \rangle]$ for any $m$ in $\Z$ and $\lambda \in \bY$.
Then the results of~\cite[\S 7.3]{mkdkm} show that, under this identification, we have
\[
[\cT'(\lambda)] = \puN_{w_\lambda} \quad \text{for any $\lambda \in \bY$}
\]
(where $p=\mathrm{char}(\bk)$).
Composing this embedding with the isomorphism
\[
[\Dmix_{(\mathbf{I})}(\Gr',\bk)] \simto [\Db \Coh^{\Gp \times \Gm}(\tcN)_\bk]
\]
induced by $\Upsilon'$, we deduce an embedding
\begin{equation}
\label{eqn:Masp-Groth}
\Masph \hookrightarrow [\Db \Coh^{\Gp \times \Gm}(\tcN)_\bk]
\end{equation}
sending $v^m N_{w_\lambda}$ to $[\Delta(\lambda) \langle m \rangle[m]]$ for any $m$ in $\Z$ and $\lambda \in \bY$, and which also sends $\puN_{w_\lambda}$ to $[\PEx^\bk_\lambda]$. The proof of Lemma~\ref{lem:filtration-Psi} shows that, for any $s \in S$, the embedding~\eqref{eqn:Masp-Groth} intertwines the action of $\uH_s$ on the left-hand side with the morphism induced by $\Psi_s$ on the right-hand side. These remarks show that if $\lambda \in \bY$ and if $\uw=(s_1, \cdots, s_r)$ is an expression, then the direct summands of
\[
\Psi_{s_r} \circ \cdots \circ \Psi_{s_1}(\PEx^\bk_\lambda)
\]
can be determined by expanding the element $\puN_{w_\lambda} \cdot \uH_{\uw}$ in the $p$-canonical basis of $\Masph$. In particular, these direct summands are of the form $\PEx^\bk_\nu$ with $\nu \in \bY$ and $w_\nu \pRleq w_\lambda$.

\begin{lem}
\label{lem:support-geom-pcells}
Let $\lambda, \mu \in \bY$, and assume that $w_\lambda \pRsim w_\mu$. Then
\[
\supp \bigl( \Ext^\bullet_{\Db \Coh(\tcN)_\bk}(\PEx^\bk_\lambda, \PEx^\bk_\lambda) \bigr) = \supp \bigl( \Ext^\bullet_{\Db \Coh(\tcN)_\bk}(\PEx^\bk_\mu, \PEx^\bk_\mu) \bigr).
\]
\end{lem}

\begin{proof}
Of course, it is enough to prove that if $w_\lambda \pRleq w_\mu$, then
\[
\supp \bigl( \Ext^\bullet_{\Db \Coh(\tcN)_\bk}(\PEx^\bk_\lambda, \PEx^\bk_\lambda) \bigr) \subset \supp \bigl( \Ext^\bullet_{\Db \Coh(\tcN)_\bk}(\PEx^\bk_\mu, \PEx^\bk_\mu) \bigr).
\]
However, if $w_\lambda \pRleq w_\mu$ then there exists an expression $\uw=(s_1, \cdots, s_r)$ such that $\puN_{w_\lambda}$ appears with nonzero coefficient in $\puN_{w_\mu} \cdot \uH_{\uw}$ (see~\S\S\ref{ss:right-p-cells}--\ref{ss:antispherical-cells}). Then the remarks above show that $\PEx^\bk_\lambda$ is a direct summand of $\Psi_{s_r} \circ \cdots \circ \Psi_{s_1}(\PEx^\bk_\mu)$, so that
\begin{multline*}
\supp \bigl( \Ext^\bullet_{\Db \Coh(\tcN)_\bk}(\PEx^\bk_\lambda, \PEx^\bk_\lambda) \bigr) \subset \\
\supp \bigl( \Ext^\bullet_{\Db \Coh(\tcN)_\bk}(\Psi_{s_r} \circ \cdots \circ \Psi_{s_1}(\PEx^\bk_\mu), \Psi_{s_r} \circ \cdots \circ \Psi_{s_1}(\PEx^\bk_\mu)) \bigr).
\end{multline*}
Now, by adjunction (see Remark~\ref{rmk:Psi-adjoint}) we have
\begin{multline*}
\Ext^\bullet_{\Db \Coh(\tcN)_\bk}(\Psi_{s_r} \circ \cdots \circ \Psi_{s_1}(\PEx^\bk_\mu), \Psi_{s_r} \circ \cdots \circ \Psi_{s_1}(\PEx^\bk_\mu)) \cong \\
\Ext^\bullet_{\Db \Coh(\tcN)_\bk}(\PEx^\bk_\mu, \Psi_{s_1} \circ \cdots \circ \Psi_{s_r} \circ \Psi_{s_r} \circ \cdots \circ \Psi_{s_1}(\PEx^\bk_\mu)),
\end{multline*}
and the support of the right-hand side is included in $\supp \bigl( \Ext^\bullet_{\Db \Coh(\tcN)_\bk}(\PEx^\bk_\mu, \PEx^\bk_\mu) \bigr)$; this finishes the proof.
\end{proof}

Now, we are in a position to give the proof of Theorem~\ref{thm:support-char-0}\eqref{it:support-char0-2}. More precisely, as explained in Remark~\ref{rmk:orbit-LV}, we will prove that for any $\lambda \in \bX$ we have
\[
\supp \bigl( \Ext^\bullet_{\Db \Coh(\tcN)_\C}(\PEx^\C_\lambda, \PEx^\C_\lambda) \bigr) = \overline{\scO^\C_{\bc(w)}},
\]
where $w \in \fW$ is the unique element such that $w_\lambda = w\omega$ for some $\omega \in \Omega$. Write $w=w_\mu$ for some $\mu \in \bY$; then
it is easy to check that
\[
\PEx^\C_\lambda = \mathscr{J}_{T_\omega}(\PEx^\C_\mu).
\]
Since $\mathscr{J}_{T_\omega}$ induces an autoequivalence of $\Db \Coh(\tcN)_\C$, we can assume that $\lambda=\mu$, i.e.~that $\omega=1$. And using Lemma~\ref{lem:support-geom-pcells} and Remark~\ref{rmk:intersection-fWf} (or rather the analogous characteristic-$0$ observation, which can be checked similarly using quantum groups) we can further assume that $w \in \fWf$ (or in other words that $\lambda \in -\bX^+$). In this case we have
\[
\supp \bigl( \Ext^\bullet_{\Db \Coh(\tcN)_\C}(\PEx^\C_\lambda, \PEx^\C_\lambda) \bigr) \supset \supp \bigl( R\pi_* \PEx^\C_\lambda \bigr) = \supp \bigl( \Ext^\bullet_{\Db \Coh(\tcN)_\C}(\cO_{\tcN_\C}, \PEx^\C_\lambda) \bigr),
\]
and by Theorem~\ref{thm:support-char-0}\eqref{it:support-char0-1} the right-hand side is equal to $\overline{\scO^\C_{\bc(w_\lambda)}}$; hence to conclude it suffice to prove that
\[
\supp \bigl( \Ext^\bullet_{\Db \Coh(\tcN)_\C}(\PEx^\C_\lambda, \PEx^\C_\lambda) \bigr) \subset \overline{\scO^\C_{\bc(w_\lambda)}}.
\]

Now, choose a reduced decomposition $w_\lambda = s_1 \cdots s_r$ (with $s_i \in S$ for $i \in \{1, \cdots, r\}$). Then, by Proposition~\ref{prop:parity-BS}\eqref{it:Psi-parity-2}, $\PEx^\C_\lambda$ is a direct summand of the object $\PEx^\C(1,(s_1, \cdots, s_r))$; it follows that
\[
\supp \bigl( \Ext^\bullet_{\Db \Coh(\tcN)_\C}(\PEx^\C_\lambda, \PEx^\C_\lambda) \bigr) = \supp \bigl( \Ext^\bullet_{\Db \Coh(\tcN)_\C}(\PEx^\C(1,(s_1, \cdots, s_r)), \PEx^\C_\lambda) \bigr).
\]
Now, by definition and adjunction (see again Remark~\ref{rmk:Psi-adjoint}), we have
\[
\Ext^\bullet_{\Db \Coh(\tcN)_\C}(\PEx^\C(1,(s_1, \cdots, s_r)), \PEx^\C_\lambda) = \Ext^\bullet_{\Db \Coh(\tcN)_\C}(\cO_{\tcN_\C}, \Psi_{s_1} \circ \cdots \circ \Psi_{s_r} (\PEx^\C_\lambda)).
\]
It follows that
\[
\supp \bigl( \Ext^\bullet_{\Db \Coh(\tcN)_\C}(\PEx^\C_\lambda, \PEx^\C_\lambda) \bigr) = \supp \bigl( R\pi_*(\Psi_{s_1} \circ \cdots \circ \Psi_{s_r} (\PEx^\C_\lambda)) \bigr).
\]
As explained before the proof of Lemma~\ref{lem:support-geom-pcells}, $\Psi_{s_1} \circ \cdots \circ \Psi_{s_r} (\PEx^\C_\lambda)$ is a direct sum of objects of the form $\PEx^\C_\nu$ with $\nu \in \bY$ such that $w_\nu \Rleq w_\lambda$. For such an object, by Theorem~\ref{thm:support-char-0}\eqref{it:support-char0-1} we have either $R\pi_* \PEx^\C_\nu = 0$ (if $\nu \notin -\bX^+$) or
\[
\supp \bigl( R\pi_* \PEx^\C_\nu \bigr) = \overline{\scO^\C_{\bc(w_\nu)}}.
\]
By Lemma~\ref{lem:order-orbits} we have $\overline{\scO^\C_{\bc(w_\nu)}} \subset \overline{\scO^\C_{\bc(w_\lambda)}}$ since $w_\nu \Rleq w_\lambda$, hence these remarks conclude the proof.

\section{Exotic parity sheaves and tilting modules}
\label{sec:exotic-tilting}

In this section we continue with the assumptions of~\S\ref{ss:char-tilting}. As in Section~\ref{sec:humphreys-conj} we identify $\Gp_\bk$ with its Frobenius twist in the natural way.

\subsection{Relation with tilting modules for reductive groups}
\label{ss:relation-tilt}

The dot-action of $W$ on $\bX$ extends in a natural way to an action of $\Wext$. We
let $\Rep_\varnothing(\Gp_\bk)$ be the Serre subcategory of $\Rep(\Gp_\bk)$ generated by the simple modules of the form $\irr(w \cdot_p 0)$ where $w \in \fWext$. (Note in particular that the ``extended principal block'' $\Rep_\varnothing(\Gp_\bk)$ contains the principal block $\Rep_0(\Gp_\bk)$.)
One of the main results of~\cite{prinblock} (see in particular~\cite[Theorem~10.7 and its proof]{prinblock}) is the construction of a functor
\[
\Xi : \Db \Coh^{\Gp \times \Gm}(\tcN)_\bk \to \Db\Rep_\varnothing(\Gp_\bk)
\]
and an isomorphism of functors $\varepsilon : \Xi \circ \langle 1 \rangle[1] \simto \Xi$ such that:
\begin{enumerate}
\item
for any $\cF,\cG$ in $\Db \Coh^{\Gp \times \Gm}(\tcN)_\bk$, $\Xi$ and $\varepsilon$ induce an isomorphism
\[
\bigoplus_{n \in \Z} \Hom_{\Db \Coh^{\Gp \times \Gm}(\tcN)_\bk}(\cF,\cG \langle n \rangle[n]) \simto \Hom_{\Db \Rep_\varnothing(\Gp_\bk)}(\Xi(\cF), \Xi(\cG));
\]
\item
for any $V \in \Rep(\Gp_\bk)$ and $\cF \in \Db \Coh^{\Gp \times \Gm}(\tcN)_\bk$, there exists a canonical and functorial isomorphism
\[
\Xi(\cF \otimes V) \cong \Xi(\cF) \otimes \Fr^*(V);
\]
\item
for any $\lambda \in \bX$ we have
\[
\Xi(\Delta_\bk(\lambda)) \cong \weyl(w_\lambda \cdot_p 0), \quad \Xi(\nabla_\bk(\lambda)) \cong \coweyl(w_\lambda \cdot_p 0).
\]
\end{enumerate}

Using this result we can finally give the proof of Lemma~\ref{lem:cohom-G1}.

\begin{proof}[Proof of Lemma~{\rm \ref{lem:cohom-G1}}]
By Frobenius reciprocity and the fact that $\Ind_{(\Gp_\bk)_1}^{\Gp_\bk}$ is exact (see~\cite[Corollary~I.5.13(1)]{jantzen}), we have
\begin{equation}\label{eqn:cohom-G1}
\Ext^\bullet_{(\Gp_\bk)_1}(\bk,\bk) \cong \Ext^\bullet_{\Gp_\bk}(\bk, \Ind_{(\Gp_\bk)_1}^{\Gp_\bk}(\bk)) \cong \Ext^\bullet_{\Gp_\bk}(\bk, \Fr^*(\cO(\Gp_\bk))).
\end{equation}
Since $\bk = \Xi(\cO_{\tcN})$, the functor $\Xi$ induces an isomorphism
\[
\bigoplus_{n,m \in \Z} \Hom_{\Db \Coh^{\Gp \times \Gm}(\tcN)_\bk}(\cO_{\tcN_\bk}, \cO_{\tcN_\bk} \otimes \cO(\Gp_\bk) \langle n \rangle [m]) \simto \Ext^\bullet_{\Gp_\bk}(\bk, \Fr^*(\cO(\Gp_\bk))).
\]
Now the left-hand side identifies canonically with $R^\bullet\Gamma(\tcN_\bk, \cO_{\tcN_\bk}) \cong \cO(\cN_\bk)$. (The latter isomorphism follows from~\cite[Theorem~2]{klt}, the normality of $\cN_\bk$---see~\cite[\S 5]{klt}---and Zariski's Main Theorem.) The desired isomorphism
\[
\cO(\cN_\bk) \simto \Ext^\bullet_{(\Gp_\bk)_1}(\bk,\bk)
\]
follows. It is left to the reader to check that this morphism is $(\Gp \times \Gm)_\bk$-equivariant, and an algebra morphism.  (The latter involves unwinding the definitions of each of the steps in~\eqref{eqn:cohom-G1} to express the ring structure of $\Ext^\bullet_{(\Gp_\bk)_1}(\bk,\bk)$ in terms of that of $\cO(\Gp_\bk)$.)
\end{proof}

The application of our results of Section~\ref{sec:support} to the Humphreys conjecture(s) will be based on the following result.

\begin{prop}
\label{prop:supp-tilt-cN}
For any $\lambda \in \bX$, there exist $\Gp_\bk$-equivariant isomorphisms
\begin{align*}
\Ext^\bullet_{(\Gp_\bk)_1}(\bk, \tilt(w_\lambda \cdot_p 0)) &\cong R^\bullet\Gamma(\tcN_\bk, \PEx^\bk_\lambda), \\
\Ext^\bullet_{(\Gp_\bk)_1}(\tilt(w_\lambda \cdot_p 0), \tilt(w_\lambda \cdot_p 0)) &\cong \Ext^\bullet_{\Db \Coh(\tcN)_\bk}(\PEx^\bk_\lambda, \PEx^\bk_\lambda)
\end{align*}
which intertwine the actions of $\Ext^\bullet_{(\Gp_\bk)_1}(\bk,\bk)$ and $\cO(\cN_\bk)$ under the isomorphism of Lemma~{\rm \ref{lem:cohom-G1}} constructed above.
\end{prop}

\begin{proof}
The proof is the same as that of Lemma~\ref{lem:cohom-G1} once we have shown that $\Xi(\PEx^\bk_\lambda) \cong \tilt(w_\lambda \cdot_p 0)$. For this we first show that $\Xi(\PEx^\bk_\lambda)$ is a tilting $\Gp_\bk$-module. In fact, for any $\mu \in \bX$ we have
\[
\Ext^n_{\Gp_\bk}(\weyl(w_\mu \cdot_p 0), \Xi(\PEx^\bk_\lambda)) \cong \Ext^n_{\Gp_\bk}(\Xi(\Delta_\bk(\mu)), \Xi(\PEx^\bk_\lambda)).
\]
We deduce an isomorphism
\[
\Ext^n_{\Gp_\bk}(\weyl(w_\mu \cdot_p 0), \Xi(\PEx^\bk_\lambda)) \cong \bigoplus_{m \in \Z} \Hom_{\Db \Coh^{\Gp \times \Gm}(\tcN)_\bk}(\Delta_\bk(\mu), \PEx^\bk_\lambda \langle m \rangle [n+m]).
\]
Using Corollary~\ref{cor:Hom-E-R}, it follows that $\Ext^n_{\Gp_\bk}(\weyl(w_\mu \cdot_p 0), \Xi(\PEx^\bk_\lambda))=0$ unless $n=0$. One shows similarly that $\Ext^n_{\Gp_\bk}(\Xi(\PEx^\bk_\lambda), \coweyl(w_\mu \cdot_p 0))=0$ unless $n=0$. Combining these facts, we see that $\Xi(\PEx^\bk_\lambda)$ indeed is a tilting $\Gp_\bk$-module.

Using the properties of $\Xi$ and~\cite[Theorem~3.1]{gordon-green} we obtain that $\Xi(\PEx^\bk_\lambda)$ is indecomposable, hence isomorphic to $\tilt(w \cdot_p 0)$ for some $w \in \fWext$. It is not difficult to check that $w=w_\lambda$, and the proof is complete.
\end{proof}

\begin{rmk}
\label{rmk:proof-Humphreys-quantum}
The functor $\Xi$ has a ``quantum analogue'' constructed (long before its modular counterpart) by Arkhipov--Bezrukavnikov--Ginzburg. More precisely, if $\Gp'_\C=\Gp_\C/Z(\Gp_\C)$, then the main result of the first part of~\cite{abg} is the construction of a functor
\[
\Xi_\C : \Db \Coh^{\Gp' \times \Gm}(\tcN)_\C \to \Db \Rep_0(\UQ)
\]
with the same properties as the functor $\Xi$ of~\S\ref{ss:relation-tilt}, where $\Rep_0(\UQ)$ is the principal block of $\Rep(\UQ)$.
As in the case of $\Xi$, the formal properties of $\Xi_\C$ imply that this functor sends $\PEx^\C_\lambda$ to $\TQ(w_\lambda \cdot_\ell 0)$ for any $\lambda \in \bY$. Combining this observation with
Theorem~\ref{thm:parities-char-0} provides an alternative proof of~\cite[Theorem~3]{bezru} which does not rely on~\cite{ab}. (This proof replaces the weight arguments of~\cite{bezru} by parity considerations.)
\end{rmk}

\subsection{A proof of the Humphreys conjectures in large characteristic}

The following theorem is the main result of this paper. It provides a proof of the ``lower bound'' part of Conjecture~\ref{conj:Humphreys} and Conjecture~\ref{conj:Humphreys-relative}, and a full proof in large characteristic.

\begin{thm}
\label{thm:Humphreys-large}
\leavevmode
\begin{enumerate}
\item
\label{it:Humphreys-large-1}
For any $w \in \fW$ and $\lambda \in \bX^+$ which belongs to the lower closure of $w \cdot_p \cC_p$, we have
\[
V_{(\Gp_\bk)_1}(\tilt(\lambda)) \supset \overline{\iota_\bk(\scO_{\bc(w)}^\C)}.
\]
Moreover, if $w \in \fWf$ then
\[
\overline{V} \hspace{-2pt}{}_{(\Gp_\bk)_1}(\tilt(w \cdot_p 0)) \supset \overline{\iota_\bk(\scO^\C_{\bc(w)})}.
\]
\item
\label{it:Humphreys-large-2}
There exists $N \in \Z_{\geq 0}$ (depending on $\Gp_\Z$) such that if $p>N$, then for any $w \in \fW$ and $\lambda \in \bX^+$ which belongs to the lower closure of $w \cdot_p \cC_p$, we have
\[
V_{(\Gp_\bk)_1}(\tilt(\lambda)) = \overline{\iota_\bk(\scO_{\bc(w)}^\C)},
\]
and if $w \in \fWf$ then
\[
\overline{V}\hspace{-2pt}{}_{(\Gp_\bk)_1}(\tilt(w \cdot_p 0)) = \overline{\iota_\bk(\scO^\C_{\bc(w)})}.
\]
\end{enumerate}
\end{thm}

\begin{proof}
As explained in~\S\ref{ss:humphreys-conj}, in both cases it suffices to treat the case $\lambda = w \cdot_p 0$. In this case, if $\mu \in \bY$ is such that $w=w_\mu$, by Proposition~\ref{prop:supp-tilt-cN} we have
\[
V_{(\Gp_\bk)_1}(\tilt(w \cdot_p 0)) = \supp \bigl( \Ext^\bullet_{\Db \Coh(\tcN)_\bk}(\PEx^\bk_\mu, \PEx^\bk_\mu) \bigr)
\]
and if $w \in \fWf$ (i.e.~$\mu \in -\bX^+$) we have
\[
\overline{V}\hspace{-2pt}{}_{(\Gp_\bk)_1}(\tilt(w \cdot_p 0)) = \supp(R\pi_* \PEx^\bk_\mu).
\]
Then~\eqref{it:Humphreys-large-1} follows from Proposition~\ref{prop:support-geometry} and Remark~\ref{rmk:orbit-LV}.

To prove~\eqref{it:Humphreys-large-2}, recall that $\fW$ has only a finite number of antispherical right cells (see the proof of Corollary~\ref{cor:stabilization}). Hence it suffices to fix such a cell $\mathbf{c}$ 
and prove that there exists $N_{\mathbf{c}}$ such that if $p>N_{\mathbf{c}}$, for any $w \in \mathbf{c}$ and $\lambda \in \bX^+$ which belongs to the lower closure of $w \cdot_p \cC_p$, we have
\[
V_{(\Gp_\bk)_1}(\tilt(\lambda)) = \overline{\iota_\bk(\scO_{\mathbf{c}}^\C)},
\]
and if moreover $w \in \fWf$ then
\[
\overline{V}\hspace{-2pt}{}_{(\Gp_\bk)_1}(\tilt(w \cdot_p 0)) = \overline{\iota_\bk(\scO^\C_{\mathbf{c}})}.
\]

Fix a finite subset $K \subset \mathbf{c}$ as in Proposition~\ref{prop:finite-generation}. By Proposition~\ref{prop:support-geometry}, there exists $N_{\mathbf{c}}$ such that for any $p > N_{\mathbf{c}}$ and any $\mu \in \bY$ such that $w_\mu \in K$ we have
\[
\supp \bigl( \Ext^\bullet_{\Db \Coh(\tcN)_\bk}(\PEx^\bk_\mu, \PEx^\bk_\mu) \bigr) = \overline{\iota_\bk(\scO_{\mathbf{c}}^\C)}.
\]
Then as above, for any $w \in K$ we have
\[
V_{(\Gp_\bk)_1}(\tilt(w \cdot_p 0)) = \overline{\iota_\bk(\scO_{\mathbf{c}}^\C)}.
\]
If $p>N_{\mathbf{c}}$ and $w \in \mathbf{c}$ is arbitrary, then by Proposition~\ref{prop:finite-generation} there exist $v \in K$ and $\nu \in \bY^+$ such that $w=t_\nu v$. Then $w \cdot_p 0 = v \cdot_p 0 + p\nu$, so that $\tilt(w \cdot_p 0)$ is a direct summand of $\tilt(v \cdot_p 0) \otimes \tilt(p\nu)$, and hence
$w \Tleq v$. 
By Lemma~\ref{lem:V-Tleq}, this implies that $V_{(\Gp_\bk)_1}(\tilt(w \cdot_p 0)) \subset V_{(\Gp_\bk)_1}(\tilt(v \cdot_p 0)) = \overline{\iota_\bk(\scO_{\mathbf{c}}^\C)}$. Since the reverse inclusion was already proved in~\eqref{it:Humphreys-large-1}, this implies that $V_{(\Gp_\bk)_1}(\tilt(w \cdot_p 0))=\overline{\iota_\bk(\scO_{\mathbf{c}}^\C)}$.

If we assume in addition that $w \in \fWf$, then by~\eqref{eqn:V-oV} we have
\[
\overline{V}\hspace{-2pt}{}_{(\Gp_\bk)_1}(\tilt(w \cdot_p 0)) \subset V_{(\Gp_\bk)_1}(\tilt(w \cdot_p 0)).
\]
The right-hand side is now known to be equal to $\overline{\iota_\bk(\scO_{\mathbf{c}}^\C)}$, and the left-hand side contains $\overline{\iota_\bk(\scO_{\mathbf{c}}^\C)}$ by~\eqref{it:Humphreys-large-1}. We deduce that $\overline{V}\hspace{-2pt}{}_{(\Gp_\bk)_1}(\tilt(w \cdot_p 0)) = \overline{\iota_\bk(\scO_{\mathbf{c}}^\C)}$.
\end{proof}

\begin{rmk}
\label{rmk:main-thm}
\leavevmode
\begin{enumerate}
\item
From the proof of Theorem~\ref{thm:Humphreys-large} we see that if Conjecture~\ref{conj:Humphreys} is true, then Conjecture~\ref{conj:Humphreys-relative} also holds (for any fixed $\bk$).
\item
Running the arguments of the proof of Theorem~\ref{thm:Humphreys-large} backwards, we see that the bounds $N_1$ and $N_2$ in Proposition~\ref{prop:support-geometry} can be chosen to be independent of $\lambda$.
\end{enumerate}
\end{rmk}

\subsection{A simplicity criterion for parity objects}

We conclude this section with a criterion that guarantees that certain parity objects in $\Db \Coh^{\Gp \times \Gm}(\tcN)_\bk$ are simple given the information that a corresponding tilting module remains indecomposable upon quantization. This criterion will be used below in the course of the proof of the Humphreys conjecture in type $\mathbf{G}_2$.

\begin{cor}
\label{cor:E-simple}
Let $\lambda \in \bY$, and assume that $\tilt(w_\lambda \cdot_p 0)_q \cong \TQ(w_\lambda \cdot_p 0)$. Then $\PEx^\bk_\lambda \cong \fL^\bk(\lambda)$.
\end{cor}

\begin{proof}
Our assumption implies that for any $\mu \in \bX$ we have
\[
\dim_\bk \bigl( \Hom_{\Gp_\bk}( \weyl(w_\mu \cdot_p 0), \tilt(w_\lambda \cdot_p 0)) \bigr) = \dim_\C \bigl( \Hom_{\UQ}( \weyl_q(w_\mu \cdot_p 0), \TQ(w_\lambda \cdot_p 0)) \bigr).
\]
Using the properties of the functor $\Xi$ considered in the course of the proof of Proposition~\ref{prop:supp-tilt-cN} and the analogous properties in the quantum setting (see~\S\ref{ss:proof-support-bezru}), this implies that
\[
\sum_{m \in \Z} \dim_\F \bigl( \Hom_\F(\Delta_\F(\mu), \PEx^\F_\lambda \langle m \rangle [m] ) \bigr) = \sum_{m \in \Z} \dim_\C \bigl( \Hom_\C(\Delta_\C(\mu), \PEx^\C_\lambda \langle m \rangle [m] ) \bigr).
\]
Hence the inequalities in
Lemma~\ref{lem:criterion-E-L} must be equalities, and this lemma ensures that $\PEx^\bk_\lambda \cong \fL^\bk(\lambda)$.
\end{proof}

\section{Examples}
\label{sec:examples}

\newcommand{\sreg}{\mathrm{sreg}}

\subsection{Easy cases}
\label{ss:examples-easy}

In this subsection, we explain how to check the main conjectures for the regular, subregular, and zero nilpotent orbits.

\subsubsection{Regular orbit}

As explained in Example~\ref{ex:cell-1}, $\{1\}$ is a right cell, which is clearly antispherical. The corresponding orbit in $\cN_\C$ is the unique open orbit $\scO^\C_\reg$, and for any $\bk$ the orbit $\iota_\bk(\scO_\reg^\C)$ is the unique open orbit $\scO_\reg^\bk \subset \cN_\bk$.
Conjecture~\ref{conj:Humphreys} and Conjecture~\ref{conj:Humphreys-relative} are clearly true if $w=1$ since $\tilt(0)=\bk$.

\subsubsection{Subregular orbit}
\label{sss:Humphreys-sreg}

Now, consider the subregular orbit $\scO^\C_\sreg \subset \cN_\C$, i.e.~the unique orbit which is open in $\cN_\C \smallsetminus \scO^\C_\reg$. Then for any $\bk$ the orbit $\iota_\bk(\scO^\C_\sreg)$ is the subregular orbit $\scO^\bk_\sreg$, i.e.~the unique orbit which is open in $\cN_\bk \smallsetminus \scO^\bk_\reg$. Let $\bc_\sreg \subset \fW$ be the corresponding antispherical cell. Then by Lemma~\ref{lem:full-support} and Proposition~\ref{prop:supp-tilt-cN}, if $w \in \bc_\sreg$ we have $V_{(\Gp_\bk)_1}(\tilt(w \cdot_p 0)) \subset \cN_\bk \smallsetminus \scO^\bk_\reg = \overline{\scO^\bk_\sreg}$. On the other hand, by Theorem~\ref{thm:Humphreys-large} we have $V_{(\Gp_\bk)_1}(\tilt(w \cdot_p 0)) \supset \overline{\scO^\bk_\sreg}$. Hence Conjecture~\ref{conj:Humphreys} holds for these values of $w$. Similar arguments show that Conjecture~\ref{conj:Humphreys-relative} also holds for these values of $w$.

\subsubsection{Zero orbit}
\label{sss:Humphreys-0}

Finally, consider the unique minimal antispherical cell $\bc_0 \subset \fW$, corresponding to the orbit $\{0\} \subset \cN_\C$. The union of the lower closures of the alcoves of the form $w \cdot_p \cC_p$ with $w \in \bc_0$ is $(p-1)\rho + \bX^+$ (see in particular~\cite[Proposition~6]{andersen}). If $\lambda = (p-1)\rho + \mu$ with $\mu \in \bX^+$ then $\tilt(\lambda)$ is a direct summand in $\tilt((p-1)\rho) \otimes \tilt(\mu)$. Since $\tilt((p-1)\rho)$ is projective and injective as a $(\Gp_\bk)_1$-module (see~\cite[Proposition~II.10.2]{jantzen}), so is $\tilt((p-1)\rho) \otimes \tilt(\mu)$, and finally so is $\tilt(\lambda)$. Hence $V_{(\Gp_\bk)_1}(\tilt(\lambda))=\{0\} = \iota_\bk(\{0\})$, so that Conjecture~\ref{conj:Humphreys} holds if $w \in \bc_0$. Similar arguments show that Conjecture~\ref{conj:Humphreys-relative} also holds if $w \in \bc_0$.

\subsection{Type \texorpdfstring{$\mathbf{C}_2$}{C2}}
\label{ss:type-C2}

In this subsection we assume that $\Gp_\Z$ is quasi-simple of type $\mathbf{C}_2$, so that $\Gp_\bk=\mathrm{Sp}_4(\bk)$. We will prove Conjecture~\ref{conj:Humphreys} and Conjecture~\ref{conj:Humphreys-relative} in this case, under the assumption that $p>5$. By Remark~\ref{rmk:main-thm}, it suffices to consider Conjecture~\ref{conj:Humphreys}.

\begin{figure}
\includegraphics[scale=0.55]{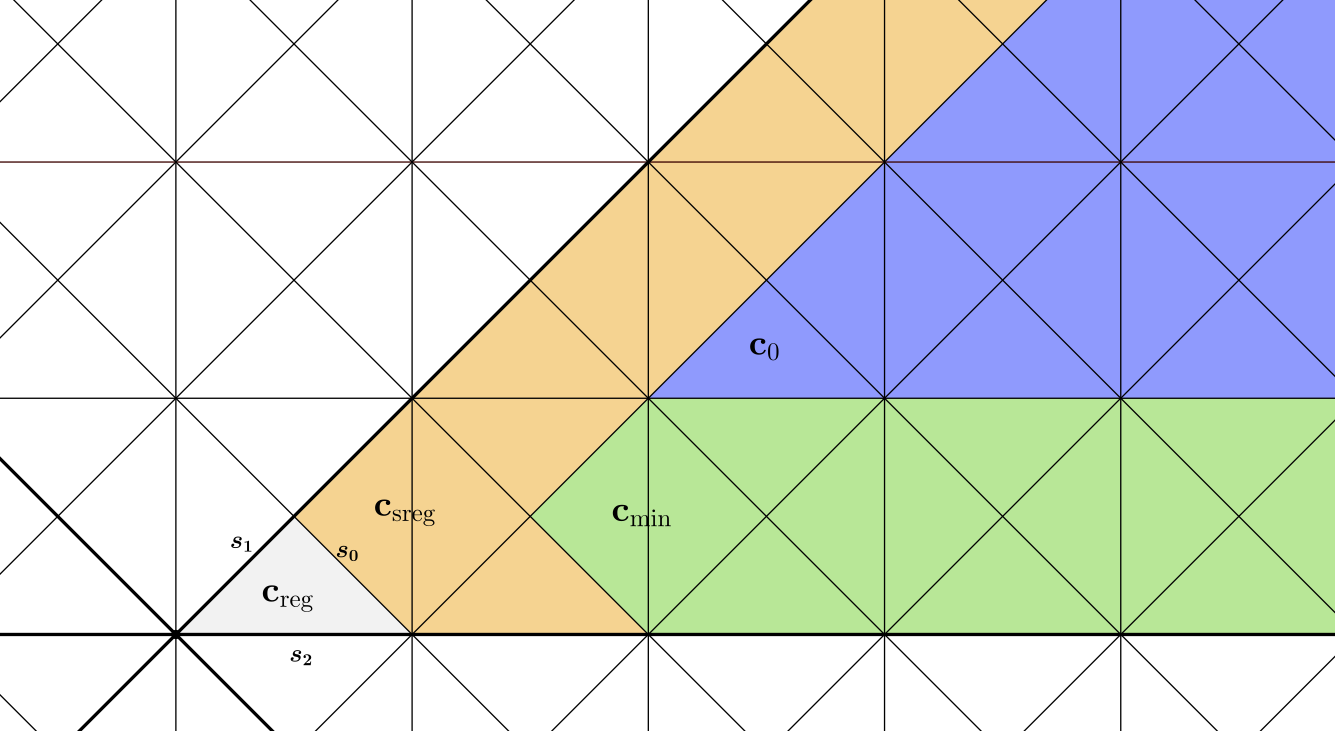}
\caption{Antispherical right cells for $\Gp_\Z$ of type $\mathbf{C}_2$.}\label{fig:c2}
\end{figure}

We denote by $\alpha_1$ and $\alpha_2$ the simple roots, with $\alpha_1$ short and $\alpha_2$ long. In this case $\Gp_\bk$ has 4 orbits in $\cN_\bk$ (for any $\bk$): the orbits $\scO_\reg^\bk$, $\scO^\bk_\sreg$ and $\{0\}$ considered in~\S\ref{ss:examples-easy}, and the minimal orbit $\scO_{\min}^\bk$. The antispherical right cells have been determined in~\cite{lusztig-cells1} and are depicted in Figure~\ref{fig:c2}.  Let $\bc_{\min} \subset \fW$ be the antispherical cell corresponding to $\scO^\C_{\min}$; then in view of the cases considered in~\S\ref{ss:examples-easy}, to settle Conjecture~\ref{conj:Humphreys} in this case it suffices to consider the case $w \in \bc_{\min}$. In fact, using Theorem~\ref{thm:Humphreys-large}, it suffices to prove that if $w \in \bc_{\min}$ we have $V_{(\Gp_\bk)_1}(\tilt(w \cdot_p 0)) \subset \overline{\scO^\bk_{\min}}$, or in other words $V_{(\Gp_\bk)_1}(\tilt(w \cdot_p 0)) \not\supset \overline{\scO^\bk_\sreg}$.

Let $H \subset \Gp_\bk$ be the derived subgroup of the Levi factor of $\Gp_\bk$ associated with $\alpha_1$ (so $H \cong \mathrm{SL}_2(\bk)$). Then if $\cN_\bk(H) \subset \cN_\bk$ is the nilpotent cone of $H$, it is not difficult to check that $\overline{\scO_\sreg^\bk} = \overline{\Gp_\bk \cdot \cN_\bk(H)}$. Now if $M$ is an $H$-module we can consider the support variety $V_{H_1}(M) \subset \cN_\bk(H)$, and using~\cite[(2.2.10)]{npv}, the considerations above show that for any $\Gp_\bk$-module $M$ we have
\begin{equation}
\label{eqn:equiv-support-B2-modular}
V_{(\Gp_\bk)_1}(M) \subset \overline{\scO^\bk_{\min}} \quad \text{iff} \quad V_{H_1}(M_{|H}) \neq \cN_\bk(H).
\end{equation}

Now, consider the quantum group $\UQ$ (associated with a root of unity of order $\ell=p$), the subalgebra $\UQ(H) \subset \UQ$ which quantizes $H$, and the corresponding small quantum group $\uQ(H) \subset \uQ$. Then by the same arguments as above (replacing the reference to~\cite{npv} to a reference to~\cite[Lemma~3.4]{ostrik-supp}) show that for any $M$ in $\Rep(\UQ)$ we have
\begin{equation}
\label{eqn:equiv-support-B2-quantum}
\text{if } V_{\uQ}(M) \subset \overline{\scO^\C_{\min}} \quad \text{then} \quad V_{\uQ(H)}(M_{|\UQ(H)}) \neq \cN_\C(H),
\end{equation}
where $\cN_\C(H) \subset \cN_\C$ is the nilpotent cone of the complex counterpart of $H$.

Fix $w \in \bc_{\min}$ and set $M:=\tilt(w \cdot_p 0)$. As in Remark~\ref{rmk:char-tilting} we consider the corresponding (possibly decomposable) tilting module $M_q$ in $\Rep(\UQ)$. By~\cite[Theorem~3.9]{rasmussen}, this quantum tilting module is a direct sum of indecomposable tilting modules of the form $\TQ(x \cdot_p 0)$ with $x \in \bc_{\min} \sqcup \bc_0$ (where $\bc_0$ is the antispherical right cell corresponding to the orbit $\{0\}$). In particular, by the quantum Humphreys conjecture (see~\S\ref{ss:quantum-Humphreys}), we have $V_{\uQ}(M_q) \subset \overline{\scO^\C_{\min}}$, i.e.
\begin{equation}
\label{eqn:support-M-quantum}
V_{\uQ(H)} \bigl( (M_q)_{|\UQ(H)} \bigr) \neq \cN_\C(H)
\end{equation}
by~\eqref{eqn:equiv-support-B2-quantum}. Now by~\cite[Proposition~6.2.1]{hardesty}\footnote{The same argument can be adapted to all types.}, under our assumptions $(M_q)_{|\UQ(H)}$ is a tilting object in $\Rep(\UQ(H))$; hence this object is the ``quantization'' of the tilting module $M_{|H}$ in $\Rep(H)$ by the process considered in Remark~\ref{rmk:char-tilting}. 

For $\nu$ a dominant weight for $H$, we denote by $\tilt^H(\nu)$, resp.~$\TQ^H(\nu)$, the indecomposable tilting module in $\Rep(H)$, resp.~in $\Rep(\UQ(H))$, of highest weight $\nu$. If $\mu$ belongs to the fundamental alcove $\mathcal{C}_p(H)$ for $H$, we have $V_{\uQ(H)}(\TQ^H(\mu)) = \cN_\C(H)$; hence
\eqref{eqn:support-M-quantum} implies that $(M_q)_{|\UQ(H)}$ has no direct summand of the form $\TQ^H(\mu)$ with $\mu \in \cC_p(H)$. 
By Remark~\ref{rmk:multplicities-mod-quantum}, this in turn implies that $M_{|H}$ has no direct summand of the form $\tilt^H(\mu)$ with $\mu \in \cC_p(H)$. By Lemma~\ref{lem:full-support} and Proposition~\ref{prop:supp-tilt-cN}, this means that $V_{H_1}(M_{|H}) \neq \cN_\bk(H)$, hence concludes the proof by~\eqref{eqn:equiv-support-B2-modular}.

\subsection{Type \texorpdfstring{$\mathbf{G}_2$}{G2}}
\label{ss:type-G2}
In this subsection we assume that $\Gp_\Z$ is simply-connected and quasi-simple of type $\mathbf{G}_2$. Unfortunately, we were unable to prove the main conjectures for this group.  Below, we will describe partial progress towards Conjectures~\ref{conj:Humphreys} and~\ref{conj:Humphreys-relative} under the assumption that $p>7$. Specifically, these conjectures will be shown to hold for every orbit except the \emph{middle orbit}.  For the middle orbit, we check Conjecture~\ref{conj:Humphreys-relative} when $w$ is the smallest element of $\fWf$ in the corresponding cell.

\begin{figure}
\includegraphics[scale=0.7]{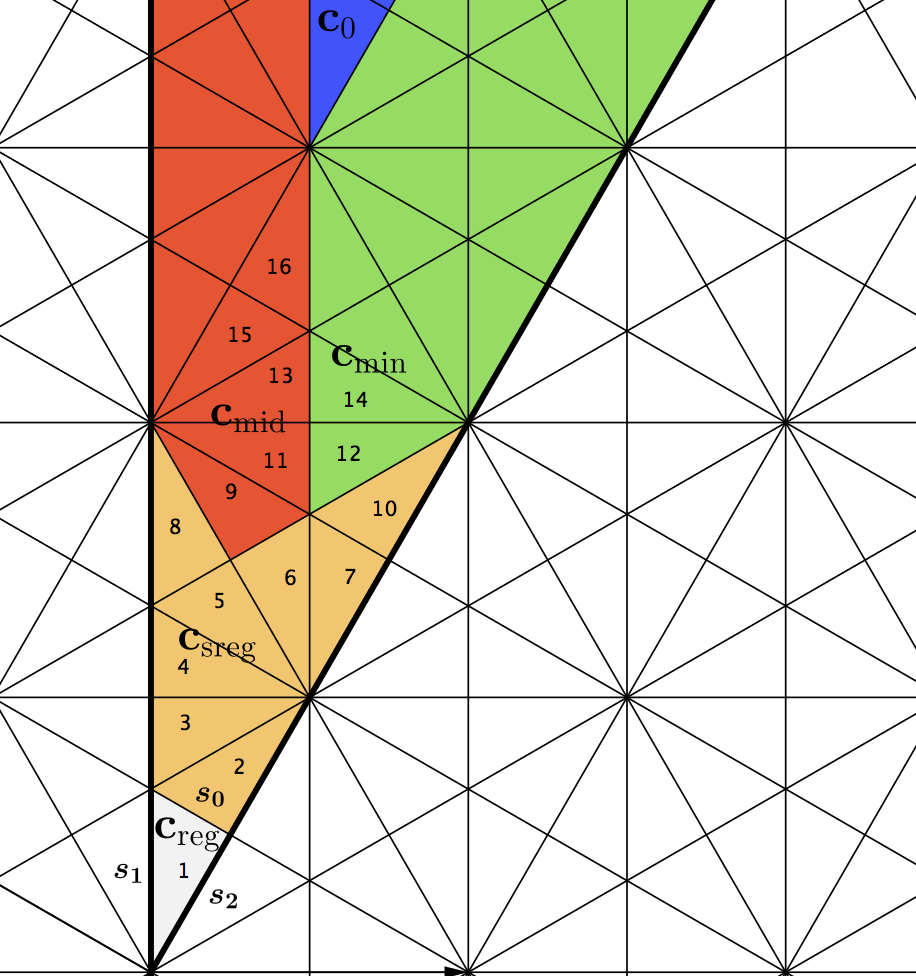}
\caption{Antispherical right cells for $\Gp_\Z$ of type $\mathbf{G}_2$.}
\label{fig:g2}
\end{figure}

\subsubsection{Preliminaries}

We let $\alpha_1$ and $\alpha_2$ be the simple roots with $\alpha_1$ short and $\alpha_2$ long, and 
denote by $s_1, s_2$ the corresponding reflections.  Let $s_0 \in S \smallsetminus \Sf$ be the remaining simple reflection in $W$.
The fundamental weights are given by $\varpi_1 = 2\alpha_1 + \alpha_2$ and $\varpi_2 = 3\alpha_1 + 2\alpha_2$, so that $3\alpha_1 + \alpha_2$, 
$\varpi_2$ and $\alpha_2$ form the set of long positive roots, and $\alpha_1$, $\varpi_1$ and $\alpha_1+\alpha_2$ form the set of short positive roots. 
In this case, $\Gp_\bk$ has 5 orbits in $\cN_\bk$ (for any $\bk$), whose closures are linearly ordered: the orbits $\scO_\reg^\bk$, $\scO^\bk_\sreg$ and $\{0\}$ considered in~\S\ref{ss:examples-easy}, 
and two additional orbits: the middle orbit $\scO_{\mathrm{mid}}^\bk$ and the minimal orbit $\scO_{\min}^\bk$. Let $\bc_{\min}$ and $\bc_{\mathrm{mid}}$ denote the antispherical right cells corresponding to $\scO_{\min}^\C$ and $\scO_{\mathrm{mid}}^\C$ respectively. 

The antispherical right cells have been determined in~\cite{lusztig-cells1} and are depicted in Figure~\ref{fig:g2}. This figure also shows certain alcoves that will play a role in the arguments below. For any $k \in \{1,\ldots,16\}$, we will denote by $w_k \in \fW$ the unique element such that $w_k\cdot_p 0$ belongs to the alcove labeled ``$k$.''

We will denote by $\UQ$ the complex quantum group of type $\mathbf{G}_2$ (associated with a root of unity of order $\ell=p$), and by $\uQ$ the corresponding small quantum group. In the following statement we use (once again) the notation introduced in Remark~\ref{rmk:char-tilting}.

\begin{lem}\label{lem:G2-tilting-char}
For any $k \in \{1,\ldots,16\}$, we have
$\tilt(w_k\cdot_p 0)_q = \TQ(w_k\cdot_p 0)$.
\end{lem}

\begin{proof}
The $w_1$ case is obvious, the $w_2,\ldots, w_8$ cases are handled in \cite[Remark 3.7]{rasmussen} and the $w_{10}$ case follows by observing that
$\TQ(w_{10}\cdot_p 0) = \Theta_{s_0}(\TQ(w_{7}\cdot_p 0))$.\footnote{There appears to have been a slight mistake in \cite[Remark 3.7]{rasmussen} regarding the 
character of $\TQ(w_{10}\cdot_p 0)$.}
This gives the characters of all the tilting modules in $\bc_{\reg}\sqcup \bc_{\sreg}$.  

For the labeled alcoves in 
$\bc_{\mathrm{mid}}$, the $w_{9}$ case immediately follows from \cite[Theorem 3.9]{rasmussen}, the $w_{11}$ case follows from the fact that 
$\TQ(w_{11}\cdot_p 0) = \Theta_{s_1}(\TQ(w_{9}\cdot_p 0))$, and the $w_{13}$, $w_{15}$ and $w_{16}$ cases can be verified by using the sum formula
given in \cite{andersen-sum} (see in particular~\cite[\S 2.13]{andersen-sum}).  Finally, the $w_{12}$ case follows by observing that
\[
\Theta_{s_1}(\TQ(w_{10}\cdot_p 0)) = \TQ(w_{12} \cdot_p 0) \oplus \TQ(w_7 \cdot_p 0)
\]
and then applying \cite[Theorem 3.9]{rasmussen}, and the $w_{14}$ case follows from the fact that $\TQ(w_{14}\cdot_p 0) = \Theta_{s_2}(\TQ(w_{12}\cdot_p 0))$. 
\end{proof}

\subsubsection{The minimal orbit}

We will completely verify Conjecture~\ref{conj:Humphreys} (and hence Conjecture~\ref{conj:Humphreys-relative}) for $w \in \bc_{\min}$. We begin with a lemma similar to \cite[Theorem 3.9]{rasmussen}. (Here, $\bc_0$ is the antispherical right cell corresponding to the orbit $\{0\}$.)

\begin{lem}\label{lem:G2-min-cell}
For any $w \in \bc_{\min}$, the quantum tilting module $\tilt(w\cdot_p 0)_q$ decomposes as a direct sum of modules of the form $\TQ(x\cdot_p 0)$ for $x \in \bc_{\min} \sqcup \bc_0$. 
\end{lem}

\begin{proof}
By Lemma~\ref{lem:G2-tilting-char}, $\TQ(w_{12}\cdot_p 0) = \tilt(w_{12}\cdot_p 0)_q$ and thus by \cite{bezru} we have
$V_{\uQ}(\tilt(w_{12}\cdot_p 0)_q) = \overline{\scO^{\C}_{\min}}$ (see~\S\ref{ss:quantum-Humphreys}). From Figure~\ref{fig:g2}, we can see that for any $w \in \bc_{\min}$ we have $w_{12} \preceq w$, where $\preceq$ is the \emph{weak order} defined in~\cite[\S 2.2]{hardesty}.
By~\cite[Lemma~3.2.1 and its proof]{hardesty} this implies that $\tilt(w \cdot_p 0)$ is a direct summand in a tilting object of the form $\Theta_{r_1} \circ \cdots \circ \Theta_{r_k}(\tilt(w_{12} \cdot_p 0))$ for some $r_1, \ldots, r_k \in S$, and then that $\tilt(w \cdot_p 0)_q$ is a direct summand in $\Theta_{r_1} \circ \cdots \circ \Theta_{r_q}(\TQ(w_{12} \cdot_p 0))$. As above this implies that
$V_{\uQ}(\tilt(w\cdot_p 0)_q) \subset \overline{\scO^{\C}_{\min}}$, hence
that every indecomposable direct summand $\TQ(x\cdot_p 0)$
of $\tilt(w\cdot_p 0)_q$ satisfies $V_{\uQ}(\tilt(x\cdot_p 0)_q) \subset \overline{\scO^{\C}_{\min}}$, hence satisfies $x \in \bc_{\min} \sqcup \bc_0$. 
\end{proof}

Now we observe that $\scO_{\mathrm{mid}}^\bk$ is regular in the Levi factor of $\Gp_\bk$ corresponding to $\alpha_1$. If we let $H \subset \Gp_\bk$ be the derived subgroup
of this Levi factor then, as in \S\ref{ss:type-C2}, to conclude it suffices to show that
\[
V_{H_1}(\tilt(w\cdot_p 0)_{|H}) \neq \cN_\bk(H)
\]
for all $w \in \bc_{\min}$. The proof of Conjecture~\ref{conj:Humphreys} then proceeds identically to the one in \S\ref{ss:type-C2}, where we replace the reference to \cite[Theorem 3.9]{rasmussen} with Lemma~\ref{lem:G2-min-cell}. 

\subsubsection{The middle orbit}

We will now verify Conjecture~\ref{conj:Humphreys-relative} for the tilting module $\tilt(w_{16}\cdot_p 0)$.
It is important to note that $w_{16} \in \fWf$ and that this element is actually the 
\emph{Duflo involution} for $\scO_{\mathrm{mid}}^\bk$. (If $z := s_2s_1s_2s_1s_0$, then it can be checked that 
$\bc_{\mathrm{mid}}\cap \fWf = \{w_{16}z^{r} : r\geq 0\}$.)

\begin{prop}\label{prop:G2-duflo}
Let $\lambda$ be the element corresponding to $w_{16}$ under the bijection $\fWf \leftrightarrow -\bX^+$. Then we have $\PEx_{\lambda}^{\bk} \cong \fL^{\bk}(\lambda)$ and 
\[
\oVG{(\Gp_\bk)_1}(\tilt(w_{16}\cdot_p 0)) = \overline{\scO^{\bk}_{\mathrm{mid}}}.
\]
\end{prop}

\begin{proof}
We have proved in Lemma~\ref{lem:G2-tilting-char} that $\TQ(w_{16}\cdot_p 0) = \tilt(w_{16}\cdot_p 0)_q$. By
Corollary~\ref{cor:E-simple}, this implies that
$\PEx_{\lambda}^{\bk} \cong \fL^{\bk}(\lambda)$. 
By~\cite[Proposition~2.6 and~\S 4.3]{achar}, it follows
that $R\pi_* \PEx_{\lambda}^\bk$ is the coherent intersection cohomology complex (in the sense of~\cite{arinkin-bezru}) attached to a pair $(\scO,\cV)$, where $\scO \subset \cN_\bk$ is a $\Gp_\bk$-orbit and $\cV$ is an irreducible $\Gp_\bk$-equivariant vector bundle on $\scO$. Then by construction we have $\supp(R\pi_* \PEx_{\lambda}^\bk) =  \overline{\scO}$. 

We will prove that $\scO = \scO^\bk_{\mathrm{mid}}$.
Proposition~\ref{prop:support-geometry}\eqref{it:support-charp-1} and the results of~\cite{bezru} (see~\S\ref{ss:quantum-Humphreys}) already ensure that
$\overline{\scO} \supset \scO^\bk_{\mathrm{mid}}$.
On the other hand, $\bc_{\reg}\sqcup \bc_{\sreg}$ is finite, hence so is its intersection with $\fWf$. Since there are as many irreducible $\Gp_\bk$-equivariant vector bundles on $\scO_\reg^\bk$ and $\scO_\sreg^\bk$ as irreducible $\Gp_\C$-equivariant vector bundles on $\scO_\reg^\C$ and $\scO_\sreg^\C$,
applying Lemma~\ref{lem:G2-tilting-char} and the preceding argument to each 
$\mu \in (\bc_{\reg}\sqcup \bc_{\sreg})\cap \fWf$, we can see that the 
push-forwards $R\pi_*\PEx_{\mu}^{\bk}$ for $\mu \in (\bc_{\reg}\sqcup \bc_{\sreg})\cap \fWf$ exhaust all of the irreducible vector bundles on $\scO_{\reg}^\bk$ and $\scO_{\sreg}^{\bk}$. The Lusztig--Vogan bijection for coefficients $\bk$ (see again~\cite[\S 4.3]{achar}) then forces $\scO \subset \overline{\scO^{\bk}_{\mathrm{mid}}}$, and finally $\scO = \scO^{\bk}_{\mathrm{mid}}$.
\end{proof}


\begin{thebibliography}{AMRW}

\bibitem[Ac]{achar}
P.~Achar, \emph{On exotic and perverse-coherent sheaves}, in \emph{Representations of reductive groups}, 11--49,
Progr. Math. 312, Birkh\"auser/Springer, 2015. 

\bibitem[AMRW]{mkdkm}
P.~Achar, S.~Makisumi, S.~Riche, and G.~Williamson, \emph{Koszul duality for Kac--Moody groups and characters of tilting modules}, preprint arXiv:1706.00183.

\bibitem[AR1]{ar-koszul}
P.~Achar and S.~Riche, {\em Koszul duality and semisimplicity of Frobenius},
Ann. Inst. Fourier {\bf 63} (2013), 1511--1612.

\bibitem[AR2]{modrap2}
P.~Achar and S.~Riche, \emph{Modular perverse sheaves on flag varieties II: Koszul duality and formality}, Duke Math. J. \textbf{165} (2016), 161--215.

\bibitem[AR3]{modrap3}
P.~Achar and S.~Riche, \emph{Modular perverse sheaves on flag varieties III: positivity conditions}, preprint arXiv:1408.4189, to appear in Trans.~Amer.~Math.~Soc.

\bibitem[AR4]{prinblock}
P.~Achar and S.~Riche, \emph{Reductive groups, the loop Grassmannian, and the Springer resolution}, preprint arXiv:1602.04412.

\bibitem[ARd]{arider}
P.~Achar and L.~Rider, {\em The affine Grassmannian and the Springer resolution
  in positive characteristic}, Compos. Math. {\bf 152} (2016), 2627--2677, with
  an appendix joint with S.~Riche.

\bibitem[An1]{andersen}
H.~H.~Andersen, \emph{Cells in affine Weyl groups and tilting modules}, in \emph{Representation theory of algebraic groups and quantum groups}, 1--16, Adv. Stud. Pure Math. 40, Math. Soc. Japan, 2004. 

\bibitem[An2]{andersen-sum}
H.~H.~Andersen, \emph{A sum formula for tilting filtrations}, 
J. Pure. Appl. Algebra \textbf{152} (1--3) (2000), 17--40. 

\bibitem[AJ]{andersen-jantzen}
H.~H.~Andersen and J.~C.~Jantzen, \emph{Cohomology of induced representations for algebraic groups},
Math. Ann. \textbf{269} (1984), 487--525.

\bibitem[AP]{andersen-paradowski}
H.~H.~Andersen and J.~Paradowski, \emph{Fusion categories arising from semisimple Lie algebras},
Comm. Math. Phys. \textbf{169} (1995), 563--588.

\bibitem[AB]{arinkin-bezru}
D.~Arinkin and R.~Bezrukavnikov, \emph{Perverse coherent sheaves}, Mosc. Math. J. \textbf{10} (2010), 
3--29, 271.

\bibitem[ArkB]{ab}
S.~Arkhipov and R.~Bezrukavnikov, \emph{Perverse sheaves on affine flags and Langlands dual group}, with an appendix by R.~Bezrukavnikov and I.~Mirkovi\'c,
Israel J. Math. \textbf{170} (2009), 135--183. 

\bibitem[ABG]{abg}
S.~Arkhipov, R.~Bezrukavnikov, and V.~Ginzburg, {\em Quantum groups, the loop Grassmannian, and the Springer resolution}, J. Amer. Math. Soc. {\bf 17} (2004), 595--678.


\bibitem[BS]{balmer-schlichting}
P.~Balmer and M.~Schlichting, \emph{Idempotent completion of triangulated categories},
J. Algebra \textbf{236} (2001), 819--834. 

\bibitem[BBD]{bbd}
A.~Be{\u\i}linson, J.~Bernstein, and P.~Deligne, \textit{Faisceaux pervers}, in \emph{Analyse et topologie sur les espaces singuliers, I (Luminy, 1981)}, 5--172,
Ast\'erisque \textbf{100}, Soc. Math. France, 1982.

\bibitem[B1]{bezru}
R.~Bezrukavnikov, \emph{Cohomology of tilting modules over quantum groups and t-structures on derived categories of coherent sheaves},
Invent. Math. \textbf{166} (2006), 327--357.

\bibitem[B2]{bezru-perverse}
R.~Bezrukavnikov, \emph{Perverse sheaves on affine flags and nilpotent cone of the Langlands dual group}, Israel J. Math. \textbf{170} (2009), 185--206. 

\bibitem[BR]{br}
R.~Bezrukavnikov and S.~Riche, \emph{Affine braid group actions on Springer resolutions}, Ann. Sci. {\'E}c. Norm. Sup{\'e}r. \textbf{45} (2012), 535--599.

\bibitem[GM]{georgiev-mathieu}
G.~Georgiev and O.~Mathieu, \emph{Fusion rings for modular representations of Chevalley groups}, in \emph{Mathematical aspects of conformal and topological field theories and quantum groups (South Hadley, MA, 1992)}, 89--100,
Contemp. Math., 175, Amer. Math. Soc., 1994. 

\bibitem[GK]{gk}
V.~Ginzburg and S.~Kumar, \emph{Cohomology of quantum groups at roots of unity},
Duke Math. J. \textbf{69} (1993), 179--198. 

\bibitem[GG]{gordon-green}
R.~Gordon and E.~Green, \emph{Graded Artin algebras},
J. Algebra \textbf{76} (1982), 111--137. 

\bibitem[Ha]{hardesty}
W.~Hardesty, \emph{On support varieties and the Humphreys conjecture in type $A$}, Adv. in Math. \textbf{329} (2018), 392--421. 


\bibitem[Ht]{hartshorne}
R.~Hartshorne, {\em Algebraic geometry}, Graduate Texts in Mathematics, no.~52,
  Springer-Verlag, New York, 1977.

\bibitem[H1]{humphreys-conj}
J.~E.~Humphreys, \emph{Conjugacy classes in semisimple algebraic groups},
Mathematical Surveys and Monographs 43, American Mathematical Society, 1995. 

\bibitem[H2]{humphreys}
J.~E.~Humphreys, \emph{Comparing modular representations of semisimple groups and their Lie algebras}, in \emph{Modular interfaces (Riverside, CA, 1995)}, 69--80,
AMS/IP Stud. Adv. Math., 4, Amer. Math. Soc., 1997.

\bibitem[Ja1]{jantzen}
J.~C.~Jantzen, \emph{Representations of algebraic
groups, second edition}, Mathematical surveys and monographs 107, Amer. Math. Soc., 2003.

\bibitem[Ja2]{jantzen-nilp}
J.~C.~Jantzen, \emph{Nilpotent orbits in representation theory}, in \emph{Lie theory}, 1--211,
Progr. Math. 228, Birkh\"auser Boston, 2004.

\bibitem[Je]{jensen}
L.~T.~Jensen, \emph{$p$-Kazhdan--Lusztig theory}, PhD thesis, University of Bonn.

\bibitem[JW]{jw}
L.~T.~Jensen and G.~Williamson, \emph{The $p$-canonical basis for Hecke algebras}, in \emph{Categorification in Geometry, Topology and Physics}, Contemp. Math. 583 (2017), 333--361.

\bibitem[JMW]{jmw}
D.~Juteau, C.~Mautner, and G.~Williamson, \emph{Parity sheaves},
J. Amer. Math. Soc. \textbf{27} (2014), 1169--1212.

\bibitem[KL1]{kl1}
D.~Kazhdan and G.~Lusztig, \emph{Representations of Coxeter groups and Hecke algebras},
Invent. Math. \textbf{53} (1979), 165--184. 

\bibitem[KL2]{kl}
D.~Kazhdan and G.~Lusztig, \emph{Schubert varieties and Poincar\'e duality}, in \emph{Geometry of the Laplace operator (Proc. Sympos. Pure Math., Univ. Hawaii, Honolulu, Hawaii, 1979)}, 185--203,
Amer. Math. Soc., 1980. 

\bibitem[KLT]{klt}
S.~Kumar, N.~Lauritzen, and J.~F.~Thomsen, \emph{Frobenius splitting of cotangent bundles of flag varieties},
Invent. Math. \textbf{136} (1999), 603--621.

\bibitem[La]{lam}
T.~Y. Lam, {\em A first course in noncommutative rings}, Graduate Texts in
  Mathematics, vol. 131, Springer-Verlag, New York, 1991.

\bibitem[LC]{le-chen}
J.~Le and X.-W.~Chen, \emph{Karoubianness of a triangulated category},
J. Algebra \textbf{310} (2007), 452--457.

\bibitem[LS]{ls}
M.~Liebeck, G.~Seitz, \emph{Unipotent and nilpotent classes in simple algebraic groups and Lie algebras},
Mathematical Surveys and Monographs 180, American Mathematical Society, 2012.

\bibitem[Li]{lipman}
J.~Lipman, \emph{Notes on derived functors and Grothendieck duality}, in \emph{Foundations of Grothendieck duality for diagrams of schemes}, 1--259, Lecture Notes in Math. 1960, Springer, 2009.

\bibitem[L1]{lusztig-cells1}
G.~Lusztig, {\em Cells in affine Weyl groups}, Algebraic groups and related
  topics (Kyoto/Nagoya, 1983), Adv. Stud. Pure Math., vol.~6, North-Holland
  Publishing Co., Amsterdam, pp.~255--287.

\bibitem[L2]{lusztig-cells2}
G.~Lusztig, \emph{Cells in affine Weyl groups. II},
J. Algebra \textbf{109} (1987), 536--548. 

\bibitem[L3]{lusztig}
G.~Lusztig, \emph{Cells in affine Weyl groups. IV},
J. Fac. Sci. Univ. Tokyo Sect. IA Math. \textbf{36} (1989), 297--328. 

\bibitem[LX]{lx}
G.~Lusztig and N.~H.~Xi, \emph{Canonical left cells in affine Weyl groups},
Adv. in Math. \textbf{72} (1988), 284--288. 

\bibitem[M1]{mathieu}
O.~Mathieu, \emph{Filtrations of $G$-modules.}
Ann. Sci. \'Ecole Norm. Sup. (4) \textbf{23} (1990), 625--644.

\bibitem[M2]{mathieu-tilting}
O.~Mathieu, \emph{Tilting modules and their applications}, in \emph{Analysis on homogeneous spaces and representation theory of Lie groups, Okayama--Kyoto (1997)}, 145--212,
Adv. Stud. Pure Math., 26, Math. Soc. Japan, 2000. 

\bibitem[MR1]{mr}
C.~Mautner and S.~Riche, \emph{On the exotic t-structure in positive characteristic}, Int. Math. Res. Not. IMRN \textbf{2016}, 5727--5774.

\bibitem[MR2]{mr2}
C.~Mautner and S.~Riche, {\em Exotic tilting sheaves, parity sheaves on affine Grassmannians, and the Mirkovi{\'c}--Vilonen conjecture}, preprint arXiv:1501.07369, to appear in J.\ Eur.\ Math.\ Soc.

\bibitem[McN1]{mcninch}
G.~McNinch, \emph{Nilpotent orbits over ground fields of good characteristic},
Math. Ann. \textbf{329} (2004), 49--85. 

\bibitem[McN2]{mcninch2}
G.~McNinch, \emph{On the nilpotent orbits of a reductive group over a local field}, preprint available on the author's web page.

\bibitem[NPV]{npv}
D.~Nakano, B.~Parshall, and D.~Vella, \emph{Support varieties for algebraic groups}
J. Reine Angew. Math. \textbf{547} (2002), 15--49. 

\bibitem[O1]{ostrik}
V.~Ostrik, \emph{Tensor ideals in the category of tilting modules},
Transform. Groups \textbf{2} (1997), 279--287. 

\bibitem[O2]{ostrik-supp}
V.~Ostrik, \emph{Cohomological supports for quantum groups}, translated from
Funktsional. Anal. i Prilozhen. \textbf{32} (1998), no. 4, 22--34, 95,
Funct. Anal. Appl. \textbf{32} (1998), no. 4, 237--246.

\bibitem[O3]{ostrik-Kth}
V.~Ostrik, \emph{On the equivariant $K$-theory of the nilpotent cone}, Represent. Theory \textbf{4} (2000), 296--305. 

\bibitem[Pa]{paradowski}
J.~Paradowski, \emph{Filtrations of modules over the quantum algebra}, in \emph{Algebraic groups and their generalizations: quantum and infinite-dimensional methods (University Park, PA, 1991)}, 93--108,
Proc. Sympos. Pure Math. 56, Part 2, Amer. Math. Soc., 1994. 

\bibitem[Pr]{premet}
A.~Premet, \emph{Nilpotent orbits in good characteristic and the Kempf--Rousseau theory}, Special issue celebrating the 80th birthday of Robert Steinberg, J. Algebra \textbf{260} (2003), 338--366.

\bibitem[Ra]{rasmussen}
T.~E.~Rasmussen, \emph{Multiplicities of second cell tilting modules}, J. Algebra~\textbf{288} (2005), 1--19.

\bibitem[R1]{riche}
S.~Riche, \emph{Koszul duality and modular representations of semisimple {L}ie algebras}, Duke Math.~J.~\textbf{154} (2010), 31--134.

\bibitem[R2]{riche-kostant}
S.~Riche, \emph{Kostant section, universal centralizer, and a modular derived Satake equivalence}, Math. Z.~\textbf{286} (2017), 223--261.

\bibitem[R3]{riche-hab}
S.~Riche, \emph{Geometric Representation Theory in positive characteristic}, habilitation thesis, available on \texttt{https://tel.archives-ouvertes.fr/tel-01431526}.

\bibitem[RW]{rw}
S.~Riche and G.~Williamson, {\em Tilting modules and the $p$-canonical basis},
preprint arXiv:1512.08296, to appear in Ast\'erisque.

\bibitem[Se]{serre}
J.~P.~Serre, \emph{Springer isomorphisms}, appendix to G.~McNinch, \emph{Optimal SL(2)-homomorphisms}, Comment. Math. Helv. \textbf{80} (2005), 391--426.

\bibitem[Sl]{slodowy}
P.~Slodowy, \emph{Simple singularities and simple algebraic groups},
Lecture Notes in Mathematics 815, Springer, 1980.

\bibitem[So1]{soergel}
W.~Soergel, \emph{Kazhdan--Lusztig polynomials and a combinatoric[s] for tilting modules}, Represent. Theory \textbf{1} (1997), 83--114.

\bibitem[So2]{soergel-char-tilt}
W.~Soergel, \emph{Character formulas for tilting modules over Kac--Moody algebras}, Represent. Theory \textbf{2} (1998), 432--448.

\bibitem[Spa]{spaltenstein}
N.~Spaltenstein, \emph{Classes unipotentes et sous-groupes de Borel}, Lecture Notes in Mathematics 946, Springer-Verlag, 1982.

\bibitem[Sp]{springer}
T.~A.~Springer, \emph{Quelques applications de la cohomologie d'intersection}, in
\emph{S\'eminaire N.~Bourbaki, Vol. 1981/82},
Ast\'erisque \textbf{92--93} (1982), Exp. 589, 249--273.

\bibitem[Wi]{williamson}
G.~Williamson, \emph{Modular intersection cohomology complexes on flag varieties}, with an appendix by T.~Braden,
Math. Z. \textbf{272} (2012), 697--727. 

\bibitem[Xi]{xi}
N.~H.~Xi, \emph{An approach to the connectedness of the left cells in affine Weyl groups},
Bull. London Math. Soc. \textbf{21} (1989), 557--561. 

\bibitem[Yu]{yun}
Z.~Yun, \emph{Weights of mixed tilting sheaves and geometric Ringel duality}, Selecta Math. (N.S.) \textbf{14} (2009), 299--320.

\end{thebibliography}
\end{document}